\documentclass{article}
\usepackage[pdftitle={Paper}, pdfauthor={Alejandro Cabrera}]{hyperref}
\usepackage{amsmath}
\usepackage{amsfonts}
\usepackage{amssymb,latexsym}
\usepackage{txfonts}
\usepackage[all]{xy}

\usepackage{graphicx}

\newtheorem{lemma} {Lemma} [section]
\newtheorem{proposition} [lemma] {Proposition}

\newtheorem{theorem} [lemma] {Theorem}
\newtheorem{ftheorem} [lemma] {Folklore Theorem}
\newtheorem{corollary} [lemma] {Corollary}
\newtheorem{definition}[lemma] {Definition}

\newtheorem{example}[lemma] {Example}
\newtheorem{remark}[lemma]{Remark}
\newenvironment{proof}{{\sc Proof:}}{ {\hspace*{\fill} $\square$\\} }

\newcommand{\dto}{\dashrightarrow}
\newcommand{\dtod}[1]{\cdot\cdot_{#1}\dto}

\def\R{\mathbb{R}}
\def\gg{\mathfrak{h}}
\def\h{\hbar}
\def\e{\epsilon}
\def\a{\alpha}
\def\b{\beta}
\def\i{inv}
\def\P{\mathfrak{P}}
\def\G{\mathcal{G}}
\def\Fou{\mathfrak{F}}

\def\ham{\phi}

\def\px{\mathfrak{c}}
\def\zero{0^{T^*}}
\def\Eu{\mathcal{E}}
\def\homp{\mu}

\def\ut{1^{(3)}}

\def\e{\epsilon}
\def\bep{[[\e]]}
\def\Tay{\bar{ \mathcal{T}}} 
\def\ba{\bar{\a}}
\def\bb{\bar{\b}}
\def\bfu{\bar{f}}
\def\bG{\bar{G}}
\def\bS{\bar{S}}
\def\Iform{I_{formal}}
\def\bm{\bar{m}}
\def\bpi{\bar{\pi}}

\def\t{\tau}
\def\g{\gamma}
\def\RT{[RT]}
\def\NRT{NRT}
\def\sym{E}

\begin{document}

\title{Generating functions for local symplectic groupoids and non-perturbative semiclassical quantization}

\author{Alejandro Cabrera\footnote{Instituto de Matem\'atica, Universidade Federal do Rio de Janeiro, Rio de Janeiro, Brazil; ORCID ID: 0000-0003-3279-0062; \emph{email:} alejandro@matematica.ufrj.br; \emph{webpage:} \url{http://www.im.ufrj.br/alejandro/} } }

\maketitle

\begin{abstract}
	This paper contains three results about generating functions for Lie-theoretic integration of Poisson brackets and their relation to quantization. In the first, we show how to construct a generating function associated to the germ of any local symplectic groupoid and we provide an explicit (smooth, non-formal) universal formula $S_\pi$ for integrating any Poisson structure $\pi$ on a coordinate space. The second result involves the relation to semiclassical quantization. We show that the formal Taylor expansion of $S_{t\pi}$ around $t=0$ yields an extract of Kontsevich's star product formula based on tree-graphs, recovering the formal family introduced by Cattaneo, Dherin and Felder in \cite{CDF}. The third result involves the relation to semiclassical aspects of the Poisson Sigma model. We show that $S_\pi$ can be obtained by non-perturbative functional methods, evaluating a certain functional on families of solutions of a PDE on a disk, for which we show existence and classification.
\end{abstract}

\tableofcontents

\section{Introduction}
It is well known that Lie-theoretic (local) symplectic groupoid structures appear in the semiclassical limit of quantizations of Poisson manifolds (see e.g. \cite{KM} for an extensive account and \cite[\S 5.7]{CDW4} for a recent functorial perspective). Perhaps one of the simplest ways of describing this limit is to consider a star product $\star_\h$ on a coordinate domain $M$ for which
\begin{equation}\label{eq:starSP}
(e^{\frac{i}{\h}p_1} \star_\h e^{\frac{i}{\h}p_2})(x) = a_\h(p_1,p_2,x) e^{\frac{i}{\h} S_P(p_1,p_2,x)},
\end{equation}
where $p_1,p_2: M\to \R$ are linear functions and $a_\h$ is regular as $\h\to 0$. The leading contribution is then given by the fast oscillatory exponent and one can heuristically show that $S_P$ \emph{generates} a local symplectic groupoid structure on $T^*M$ (see \cite{CDF} and below). The aim of this paper is to study generating functions for local symplectic groupoids, first from a pure Poisson-geometric perspective, and then to establish rigorous relations between them and Kontsevich's star product (\cite{Kontquant}). We will see the latter both as a concrete formal power series and as a result of quantization of the field-theoretic Poisson Sigma Model. Below, we introduce the main concepts and then proceed to outline our main results.

\medskip

{\bf Lie theory for Poisson brackets.} In analogy with the classical relation between Lie brackets and the Lie groups of transformations generated by them, Poisson brackets can be seen to generate Lie-theoretic structures in the realm of groupoids, \cite{CDW,Karasev}. Let us elaborate on this idea in a way that will be useful for the contents of this paper.
Given a Poisson structure $\pi$, one can think of covectors as infinitesimal transformations of individual points and consider an associated system of ODE's given in coordinate charts by
\begin{equation}\label{eq:Ppeq}
\dot x^i = \pi^{ij}(x) p_j, \ x(0)=x_0, 
\end{equation}
where the $p_j$'s enter as linear parameters. When these parameters are small enough, the solution $x_t$ is defined up to time $t=1$, and the idea is to consider the data
$$ x_{t=1} \overset{g=(x_0,p)}{\leftarrow} x_0 $$
as an arrow in a certain category transforming its source, $x_0$, into its target, $x_{t=1}$, both seen as objects. Continuing this categorical thinking, the next idea is to introduce an operation of composition of such arrows, satisfying natural axioms including associativity. This can indeed be achieved by taking into account a symplectic structure on the set of pairs $(x,p)$, together with a Poisson bracket preserving map $(x,p)\mapsto \a(x,p)$, as we shall recall below.
The resulting structure is that of a \emph{local symplectic groupoid}, which we altogether denote by $G\equiv(P_G,\omega_G,M_G,\a_G,\b_G,\i_G,m_G)$, and whose detailed definition we recall in section \ref{sec:lsg}.

\medskip

{\bf Local symplectic groupoids.}
To continue with our introductory discussion, we recall that in a local symplectic groupoid $G$, the set of arrows is given by a symplectic manifold $(P,\omega)$, the units $M\hookrightarrow P$ define a lagrangian submanifold and the graph of the multiplication map $m$ defines a lagrangian submanifold
\begin{equation}\label{eq:grmlag} gr(m)=\{(z_1,z_2,z_3): z_3=m(z_1,z_2)\} \hookrightarrow (P,-\omega) \times (P,-\omega) \times (P,\omega) =: \P_P. \end{equation}
As a consequence, the source map $\a$ defines a \emph{symplectic realization} (see \cite{CDW}): there exists a unique Poisson structure $\pi$ on $M$ such that $\a$ preserves brackets. In this case, we say that \emph{$G$ integrates the Poisson manifold $(M,\pi)$}. Following \cite{CDW} further, the symplectic realization data $(P,\omega,M,\alpha)$ actually determines the germ of $G$ around the units completely. Indeed, consider a symplectic realization in which $\alpha$ is a surjective submersion and admits a lagrangian section $M\hookrightarrow P$; such a symplectic realization is called \emph{strict}. By \cite[Chap. III, Thm. 1.2]{CDW}, which we also recall in Section \ref{subsec:streal} below, we can associate a local symplectic groupoid structure 
\begin{equation}\label{eq:Gfdef}
(P,\omega,M,\a) \mapsto \G(P,\omega,M,\a)=G
\end{equation} 
whose germ around $M$ is uniquely characterized by the property that $M$ gives the units and $\a$ defines the source map.
A distinguished case is given when
$$ (P,\omega,M) = (U\subset T^*M, \omega_c, 0^{T^*M})$$
with $U\subset T^*M$ an open neighborhood of the zero section $0^{T^*M}$ and $\omega=\omega_c$ the canonical symplectic structure, in which we say that $G$ is in \emph{normal form}. The germ of any local symplectic groupoid is always isomorphic to one in normal form, by the lagrangian tubular neighborhood theorem.

\medskip

{\bf Generating function data.} 
The idea of \emph{generating function data} $(S,\nu)$ for $G$ is to have a \emph{reference} symplectic embedding $\nu: U_\nu\subset T^*X \hookrightarrow \P_P$, defined for some manifold $X$, and a function $S:X \to \R$ such that the graph of multiplication in eq. \eqref{eq:grmlag} becomes an exact lagrangian:
$$ \nu^{-1}(gr(m)) =  \{ (l,d_lS): l \in X \} \subset T^*X .$$
Similar generating function data appear in \cite{CDW2} in the context of \emph{symplectic microgeometry} (see Remark \ref{rmk:microgeom} below). We will focus on \emph{adapted} embeddings $\nu$ (see Section \ref{subsec:generalexist}, Definition \ref{def:adaptedframing}) for which the description of the $G$ structure is non-trivially encoded in $S$.

\medskip

{\bf The case of coordinate Poisson manifolds.} To relate to known quantization formulas, the special case in which $M$ is a \emph{coordinate space}, namely diffeomorphic to an open subset in $\R^n$, will be of special interest. In these cases, we say that $(M,\pi)$ is a \emph{coordinate Poisson manifold}. Following Karasev \cite{Karasev}, for a coordinate Poisson manifold $(M,\pi)$ there exists a canonical strict symplectic realization with $(P\subset T^*M,\omega_c,0:M\hookrightarrow T^*M)$ a neighborhood of the zero section as above and $\a\equiv \a_\pi$ a smooth map defined implicitly by an analytic formula (eq. \eqref{eq:alphapi} below, see also \cite{CD}). We denote by
$$ G_\pi := \G(T^*M,\omega_c,M,\alpha_\pi), $$ 
the corresponding (germ of) local symplectic groupoid via \eqref{eq:Gfdef} and refer to it as the \emph{canonical local symplectic groupoid} integrating $(M,\pi)$. Note that $G_\pi$ is in normal form.

On the other hand, when $M$ is a coordinate space, there is a natural symplectomorphism $\nu_c: T^*X_c \to \P_{T^*M}$ with
\begin{eqnarray}
X_c := M^* \times M^* \times M &\ni& (p_1,p_2,x), \nonumber \\
\nu_c(x_1^j dp_{1j}|_{p_1}+x_2^j dp_{2j}|_{p_2}+ p_{3j}dx_3^j|_{x_3}) &=& ((x_1,p_1),(x_2,p_2),(x_3,p_3)). \label{eq:nucoord}
\end{eqnarray}  
Given $G$ a local symplectic groupoid over a coordinate $M$ which is in normal form, a \emph{coordinate generating function} is defined to be a smooth function 
$$ S: U_S\subset X_c \to \R, \text{with $U_S$ open neighborhood of $X_0 := \{(0,0,x): x \in M\}$ in $X_c$}, $$
such that $(\nu_c,S)$ defines generating function data for $G$: this boils down to
\begin{equation}\label{eq:condS}
gr(m) =_{M^{(3)}} \ \{ ((\partial_{p_1}S,p_1) , (\partial_{p_2 }S, p_2), (x,\partial_x S)): (p_1,p_2,x)\in X_c \},
\end{equation}
where the equality holds near $M^{(3)}=\{((x,0),(x,0),(x,0)):x\in M\}\subset gr(m)$. Conversely, given a coordinate Poisson manifold $(M,\pi)$ and arbitrary function function $S:U_S\subset X_c \to \R$, we can attempt to define a local symplectic groupoid structure by eq. \eqref{eq:condS} and the required groupoid axioms result equivalent (for $M$ connected) to a quadratic PDE for $S$ called \emph{symplectic groupoid associativity equation (SGA equation)} in \cite{CDF} (see eq. \eqref{eq:SGA} below).

\medskip

{\bf Poisson-theoretic results: general existence and a formula for coordinate spaces.}  We can now state our first main result in summarized form. 

\smallskip

\underline{\textsc{First main results:}} Every local symplectic groupoid admits non-trivial adapted generating function data $(\nu,S)$ and, moreover, for every choice of adapted $\nu$ there exists a unique germ of $S$ which is fixed by the generating condition and its vanishing on units (Theorem \ref{thm:genexist}). When $M$ is a coordinate space, the embedding $\nu_c$ is adapted to the canonical local symplectic groupoid $G=G_\pi$ and the corresponding coordinate generating function $S\equiv S_\pi$ can be defined implicitly by the analytic formulas in eq. \eqref{eq:DarbS} below (Theorem \ref{thm:main1}).

\smallskip

The existence result for general $(M,\pi)$ can be deduced by combining results from \cite{CDW2}, involving generating functions for general \emph{symplectic micromorphisms}, and \cite{CDW3}, relating local symplectic groupoids to monoids in that setting. In Section \ref{subsec:generalexist}, we provide direct arguments and definitions adapted to the local symplectic groupoid geometry and comment on their relation to the general symplectic microgeometry results of \cite{CDW2,CDW3} (see Remark \ref{rmk:microgeom}).
In the coordinate case, the above $S_\pi$ will be called the \emph{canonical generating function} associated to the coordinate Poisson manifold $(M,\pi)$. It provides a universal (non-formal) solution to the SGA equation for coordinate $(M,\pi)$. From the explicit formula, it also follows that $S_{t\pi}, \ t\in [0,1]$, defines a smooth (non-formal) $1$-parameter family of generating functions for $G_{t\pi}$ (see Section \ref{subsub:smooth1par}).

\medskip

{\bf Formal expansions and Kontsevich trees.} Let $(M,\pi)$ be a coordinate Poisson manifold. In \cite{CDF}, Cattaneo-Dherin-Felder show that a certain extract of Kontsevich's quantization formula \cite{Kontquant} for coordinate Poisson manifolds yields a formal $1$-dimensional family of generating functions
$$ \bS^K \in C^\infty(X)[[\e]], \text{ $\e$ formal parameter (see Example \ref{ex:bSK} below).}$$
This formal family, called \emph{formal generating function} in \cite{CDF}, provides a formal-family solution to the SGA equation and thus it generates a formal family of (local) symplectic groupoids which 
integrates $(M,\e \pi)$. 

\underline{\textsc{Second main result:}} (Theorem \ref{thm:main2SK}) the formal Taylor series $\bar S_\pi$ of the canonical smooth family $t\mapsto S_{t\pi}$ around $t=0$ coincides with $\bS^K$.

\smallskip

This result represents an enhancement of an analogous result in \cite{CD} involving only the source (realization) map $\a\equiv \a_\pi$. In that paper, it was also noticed that the the formal Taylor expasion of $\a_{t\pi}$ around $t=0$ can be presented as a Butcher series (see \cite{Bu}): a sum over rooted trees of elementary differentials for the underlying \emph{Poisson spray equations} consisting of eq. \eqref{eq:Ppeq} together with $\dot p=0$. Moreover, in \cite{CD} it was established a relation between rooted trees, certain coefficients generalizing Bernoulli numbers and elementary differentials, on the one hand, and a subset of Kontsevich tree-graphs, their weights and their symbols on the other, thus providing an "elementary explanation" for this sub-extract of Kontsevich's quantization formula. In Section \ref{subsec:graphexp} we extend this study to provide such an elementary explanation for the full tree-level extract $\bS^K$ of the quantization formula.

It is also interesting to notice that, as a corollary of the second result above, the formal expansion $\bS^K$ is shown to be the Taylor series of a smooth (non-formal) family $S_{t\pi}$ of solutions to the non-linear SGA equation (see Section \ref{subsec:generalgenfuncs}) for any smooth $\pi$ and $t\in [0,1]$. On the other hand, $\bS^K$ rarely defines an analytic function in $\e$; as shown in \cite{Dherin}, this can only be ensured when $\pi$ itself is analytic.

\medskip

{\bf Relation to the Poisson Sigma Model (PSM): underlying functional methods.} In \cite{CFquant}, Cattaneo and Felder showed that the path integral perturbative quantization of the so-called Poisson Sigma Model (\cite{Ikeda,SchStr}) associated to a coordinate Poisson $(M,\pi)$ leads to Kontsevich's formula \cite{Kontquant} for a star product quantizing $\pi$. In Section \ref{subsec:quant} below, we apply this perspective to the relation \eqref{eq:starSP} and establish an heuristic connection between the PSM action functional and an underlying generating function $S_P$ for an integration of $(M,\pi)$. The manipulations involve formal application of stationary phase arguments to ill-defined path integrals.

Nevertheless, these heuristic considerations lead us to a concrete system of PDE's, called $(PDE)^\pi_{p_1,p_2,x}$ in Section \ref{sec:psm}, and to a modified PSM action functional $A'$, both defined for maps (or \emph{fields}) from the $2$-disk into the Poisson manifold and having $(p_1,p_2,x)\in X_c=M^*\times M^* \times M$ as external parameters. This is presented in Section \ref{subsec:psmfunc} below.

\smallskip

\underline{\textsc{Third main result:}} (Theorem \ref{thm:SP}) We show an existence and classification result for families of solutions of the system of PDE's $(PDE)^\pi_{p_1,p_2,x}$ . For any such family of solutions, the evaluation of the functional $A'$ on the family yields a function $S_P(p_1,p_2,x)$ whose germ around $X_0$ coincides with that of the canonical generating function $S_\pi$, 
$$ S_P =_{X_0} S_\pi.$$

\smallskip

This result relates the Poisson-theoretic generating function $S_\pi$ with the \emph{fields} of the PSM which provide the leading semiclassical contribution. As a consequence of the second and third main results, we obtain that the formal expansion of $S_P$, the modified PSM action functional restricted to 'semiclassical solutions', yields the formal generating function $\bS^K$ defined by Kontsevich trees (which is certainly an expected fact in the context of perturbative quantum field theory). The third result also helps clarify the relation between the PSM on the disk (with $3$ insertions in the boundary) and the "Weinstein groupoid" integration of a Poisson manifold in terms of cotangent paths modulo homotopies (see \cite{CFlie} and Remark \ref{rmk:wegd}), thus extending a result of \cite{CFgds} for the hamiltonian formalism of the PSM on a square.

\medskip

{\bf Outlook:} The collection of the results in this paper can be taken as a step towards the deeper understanding of the relation between: integration by (local) symplectic groupoids, formal and non-formal aspects of quantization. In particular, in the semiclassical limit explored here, we see that geometric constructions using (implicit) functional methods provide simpler descriptions of non-trivial formal constructions.

\bigskip

{\bf Acknowledgements:} The author is indebted with several colleagues for stimulating conversations that helped constructing the contents of this paper. In particular, with Alberto Cattaneo, Marco Gualtieri, Rui Loja-Fernandes and Gon\c{c}alo "Robin" Oliveira. I also mention that a key set of ideas regarding the relation to the PSM was kindly pointed out to me by A. Cattaneo. The author was supported by CNPq grants 305850/2018-0 and 429879/2018-0 and by FAPERJ grant JCNE  E­26/203.262/201.

\section{Notation and basic definitions} \label{sec:lsg}
In this section, we recall some notations and basic definitions involving local symplectic groupoids. We begin with general notation.
Since we will be interested in the germ of local structures, we introduce the following notation:
$$ f: X \dtod{Z} Y, f_1 \ =_Z f_2 $$
to indicate that $Dom(f)$ is an open neighborhood of a subset $Z \subset X$ and to indicate that the germ around $Z$ of two such functions coincide, respectively. For subsets $N_1,N_2 \subset X$, we say
$$ N_1 =_Z N_2 \text{ if there exists an open neighborhood $U\subset X$ of $Z$ such that } N_1\cap U = N_2\cap U.$$
Categorically, behind these notions we have the category of pairs $(X,Z\subset X)$ and morphisms given by germs of functions $X \dtod{Z} Y$ defined on open neighborhoods of the underlying subsets. More details in the context of \emph{symplectic microgeometry} can be found in the series of papers including \cite{CDW2,CDW3,CDW4}.

Let $(P,\omega)$ be a symplectic manifold. Given a function $H\in C^\infty(P)$ we denote 
\begin{equation}\label{eq:hamflow}
\ham^H: P \times \R \dtod{P\times 0} P, (z,t)\mapsto \phi^H_t(z)
\end{equation}
the induced hamiltonian flow on $P$ defined by the vector field $X_H$ satisfying $i_{X_H}\omega = dH$.
For a manifold $M$, we denote the canonical symplectic structure on $T^*M$ by $\omega_c$, for which we use the sign convention\footnote{Notice that for the linear function $H(x,p)= p_{0j} x^j$ (independent of $p$) where $p_0$ is fixed, the hamiltonian flow with $p(0)=0$ is given by $p(t)=t p_0$ (instead of $-t p_0$). This is, ultimately, the reason for our choice of sign convention in $\omega_c$.} $\omega_c = dp_j \wedge dx^j$ in canonical coordinates induced by ones on $M$. Given $H\in C^\infty(T^*M)$, the flow of the hamiltonian vector field in $(T^*M,\omega_c)$ is thus given by defined by
\begin{eqnarray}
\dot x^j &=& - \partial_{p_j}H (x,p) \nonumber \\ 
 \dot p_j &=&  \partial_{x^j}H (x,p) \label{eq:generalHam} 
\end{eqnarray}
Our convention for the induced Poisson brackets $\{,\}_c$ is taken such that $ L_{X^H} F = \{H,F\}_c , \ H,F \in C^\infty(T^*M)$. It follows that, for canonical coordinates, $ \{x^i,p_j\}_c = \delta^i_j.$

Finally, the zero section of $T^*M$ will be denoted $0^{T^*}:M \hookrightarrow T^*M,$ 
the Euler vector field on $T^*M\to M$ and the associated co-vector rescaling will be denoted $$E\equiv p_j \partial_{p_j}, \ \homp_\lambda(x,p)=(x,\lambda p).$$

\subsection{Local symplectic groupoids}\label{subsec:lsg}

A \emph{local Lie groupoid} structure (or local groupoid for short) is a collection $G\equiv (P,M,\a,\b,\i,m)$ where $P$ is a manifold, $M\hookrightarrow P$ is an embedded submanifold and
\begin{eqnarray}
\a, \b:  P \dtod{M} M, \ inv: P\dtod{M} P && \nonumber \\
m: G^{(2)} \dtod{ M^{(2)}  } P\times P \times P, & \  & \nonumber\\ \text{where } G^{(2)}:= \{(z_1,z_2): \a(z_1)=\b(z_2)\}, & M^{(2)}:= \{ (x,x):x \in M\} &
\end{eqnarray}
called source, target, inversion and multiplication, respectively, satisfying the following axioms. The maps $\a,\b$ are smooth surjective submersions, $inv$ and $m$ are smooth (the latter with respect to the pullback smooth structure on $Dom(m)$), and they satisfy the algebraic axioms of a groupoid near $M$. By this we mean that the axioms hold in the above mentioned category of pairs, so that for each axiom there is an open neighborhood of the identity arrows, embedded into composable arrows as
$$ M^{(k)}:=\{(x,..,x): x \in M \} \subset G^{(k)}:= \{(z_1,..,z_k): \a(z_l)=\b(z_{l+1}), \ l=1,..,k-1 \} ,$$
where the axiom holds (see also \cite{CMS1}). We remind the reader that this is the weakest version of a local groupoid; stroger ones demand certain axioms to hold on larger subsets (for example, requiring associativity to hold whenever both sides of the identity are defined). We will sometimes emphasize the units embedding as
$$ M \hookrightarrow P, \ x\mapsto 1_x.$$

\begin{remark}
	Some structure maps can be deduced from others, a minimal choice is to have the multiplication $m$ together with its domain and the inversion map $\i$, see e.g. \cite{KM}.
\end{remark}

Observe that any (ordinary) Lie groupoid defines a local Lie groupoid as above. Moreover, if $H\rightrightarrows M$ is a Lie groupoid and $P\subset H$ is an open neighborhood of the identities $M\hookrightarrow H$, then the restriction of the structure maps to $P$ and $P^{(2)}$ define a local Lie groupoid, called the restriction of $H$ to $P$. In \cite{FMloc}, it is shown that not every local Lie groupoid (in the sense of the present paper) appears as the restriction of some Lie groupoid $H$.

A \emph{local Lie groupoid map}, denoted $F:G_1 \to G_2$, consists of a smooth map $F:P_1 \dtod{M_1} P_2$ which takes $M_1$ inside $M_2$ and satisfies $$ F(m_1(a,b))= m_2(F(a),F(b)), \text{ for $a,b$ close enough to $M_1$.} $$
Note that the germ of $F$ defines a morphism in the category of germs of local Lie groupoids.

\begin{example} \label{ex:unitg}
 In particular, the unit groupoid $1_M$ associated to a manifold $M$ provides an example of local groupoids as defined above. Here, $P=M$ and all the structure maps are the identity (there is only one identity arrow for each object in $M$).
\end{example}

\begin{example} \label{ex:Gzeropois}
	Let $M$ be a manifold, then we denote by $G_0$ the Lie groupoid in which $P=T^*M$, source and target coincide with the projection map $T^*M \to M$, the units embedding $M\hookrightarrow T^*M$ is given by the zero section $\zero$, and the multiplication map is given by addition of covectors
	$$ m_0((x,p_1),(x,p_2)) = (x, p_1+p_2).$$
	For any open neighborhood $P\subset T^*M$ of the zero section, the restriction of the structure maps to $P$ define a local Lie groupoid.
\end{example}

\begin{example} \label{ex:Hbch}
	A local Lie groupoid with only one unit, $M=\star$ a point, is the same as a \emph{Local Lie group}. An important example, that will be used in the sequel, is the Local Lie group $H_\gg$ defined by Lie algebra $(\gg,[,])$ as:
	$$ P = \gg \text{ as a manifold, } 1_\star=0\in \gg \text{ yields the identity, } \i(p)=-p \in \gg,$$  
	and the multiplication of $p_1,p_2 \in \gg$ close enough to $0$ is given by
	$$ m_\gg(p_1,p_2) = BCH(p_1,p_2) = p_1 + p_2 + \frac{1}{2}[p_1,p_2] + ... $$
	the Baker-Campbell-Hausdorff series (see e.g. \cite{DKbook}). For further reference, $BCH(p_1,p_2)$ can also be described as $k(t=1)$ for the curve $t \mapsto k(t) \in \gg$ which is the solution of (\cite[Thm. 1.5.3]{DKbook})
	$$ \theta^R_{k(t)}(\dot k(t)) = p_1, \ k(0)=p_2,$$
	with$$ \theta^R_p:\gg \to \gg, \  \theta^R_{p}(v):= \int_0^1 du \ e^{ u \ ad_{p}} v, \ \ ad_pa=[p,a],$$
	yielding the right-invariant Maurer-Cartan form at $p\sim 0$, so that $k(t)=m_\gg(tp_1,p_2)$. 
\end{example}

A \emph{local symplectic groupoid} $G=(P, \omega,M,\a,\b,\i,m)$ is a local groupoid $(P,M,\a,\b,\i,m)$ together with a symplectic structure $\omega \in \Omega^2(U_\omega)$, where $U_\omega\subset P$ is a neighborhood of $M\hookrightarrow P$, such that the graph of the multiplication map is lagrangian near $M^{(3)}$ with respect to the symplectic structure given in \eqref{eq:grmlag}.
This condition is equivalent to the algebraic condition saying that $\omega$ is \emph{multiplicative near the units}:
\begin{equation}\label{eq:omegamult} m^*\omega - pr_1^*\omega - pr_2^*\omega =_{M^{(2)}} 0 .\end{equation}
One can show (see \cite{CDW}) that the units space $M$ inherits a unique Poisson structure $\pi$ such that $G$ \emph{integrates} $(M,\pi)$, in the sense recalled in the introduction. In this case, $\a$ defines a Poisson map and $\b$ an anti-Poisson map.

\begin{remark}\label{rmk:Gopom}
Notice that, if $G=(P, \omega,M,\a,\b,\i,m)$ defines a local symplectic groupoid integrating $(M,\pi)$, then the same data but with opposite symplectic structure, denoted $\overline{G}=(P, -\omega,M,\a,\b,\i,m)$, defines a local symplectic groupoid integrating $(M,-\pi)$.
\end{remark}

\begin{example}
	For any open neighborhood $P\subset T^*M$, the local Lie groupoid defined by the restriction of $G_0$ to $P$ (see Example \ref{ex:Gzeropois}) defines a local symplectic groupoid when $P$ is endowed with the canonical symplectic structure $\omega_c$. The induced Poisson structure on $M$ is the trivial one, $\pi=0$.
\end{example}

\begin{example}\label{ex:cotangentH}
	Let $H$ be a local Lie group with identity $e$, inverse $\iota(g)=g^{-1}$, multiplication $g,h\mapsto gh$ and Lie algebra $\gg=T_e H$. 
	The \emph{cotangent groupoid of $H$} (see \cite{CDW}), denoted $T^*H$, has arrows space given by $P=T^*H$ endowed with $\omega=\omega_c$, units space given by $M=\gg^*$, and the structure maps are given by, for $z_1 \in T_g^*H, \ z_2\in T_h^*H$ with $g,h\in H$ close enough to $e$,
	$$ \tilde \a(z_1)= L_g^*z_1 \in \gg^* \text{ source, } \tilde \b(z_1) = R_g^*z_1 \in \gg^* \text{ target, } 1_x = x \in \gg^*=T^*_eH \text{ units,}$$
	$$ \tilde{\i}(z_1)=-\iota^*z_1 \in T_{g^{-1}}^*H \text{ inversion}, \ \tilde m(z_1,z_2)= L_{g^{-1}}^*z_2 = R_{h^{-1}}^*z_1 \in T^*_{gh}H \text{ multiplication.} $$
	(Above, we denoted $L_g(h)=gh=R_h(g)$ as usual.) Endowed with this structure, $T^*H$ defines a local symplectic groupoid integrating the linear Poisson structure $\tilde \pi$ on $\gg^*$ given by
	$$ \{f,g\}_{\tilde \pi}(x) = -x([df|_x,dg|_x]), \ f,g\in C^\infty(\gg^*), \ x\in \gg^*.$$
\end{example}

Every local Lie groupoid $G$ defines an underlying Lie algebroid $Lie(G)$ as its infinitesimal counterpart and the assignment $$ G \mapsto Lie(G)$$ is functorial, see e.g. \cite{CDW} and \cite{CMS1} for the conventions used here. In the particular case of local \emph{symplectic} groupoids, the underlying Lie algebroid is naturally isomorphic to
$$ T^*_\pi M \equiv (T^* M\to M, [,],\rho)$$
where $\rho(\gamma)=\pi^\sharp(\gamma), \ \gamma \in T^*M$ and $[df_1,df_2] = d \pi(df_1,df_2)$. The isomorphism is given by 
$$ Ker(D_x \a) \to T^*_x M, \ a \mapsto -\omega^\flat(a)|_{T_x M}.$$

\subsection{Local symplectic groupoids from strict symplectic realizations}\label{subsec:streal}
Here we recall the construction \eqref{eq:Gfdef} from \cite[Chap. III, $\S$ 1]{CDW}. Let $\alpha: (P,\omega) \dtod{M} (M,\pi)$ be a \emph{strict symplectic realization} as recalled in the introduction, so that $\alpha$ is a Poisson map and $M\hookrightarrow P$ is embedded as a lagrangian submanifold such that $\alpha|_M \equiv id_M$.

\begin{example} (The \emph{spray} strict symplectic realization, \cite{CM})\label{ex:sprayreal}
 Let $(M,\pi)$ be a Poisson manifold. Generalizing a Riemannian spray, a \emph{Poisson spray} is a vector field $V\in \mathfrak{X}(T^*M)$ which is homogeneous of degree $1$ for the rescaling action $(x,p)\mapsto (x,tp), \ t\in \R$ and satisfies
 $$ Tq(V|_{(x,p)}) = \pi_x(p) \in T_xM, \text{ with } q:T^*M \to M \text{ the bundle projection.}$$
 Every Poisson manifold admits a spray. Given one such $V$, denoting its flow by $\phi^V_u$, it was proven in \cite{CM} that the following 2-form
 $$ \omega_V = \int_0^1 (\phi^V_u)^*\omega_c \ du $$
 is well defined and symplectic in a neighborhood of $0^{T^*}:M \hookrightarrow T^*M$ and that the bundle projection
 $$ q: (T^*M,\omega_V) \dtod{M} (M,\pi)$$
 defines a strict symplectic realization.
\end{example}

Given a strict symplectic realization, following \cite[Chap. III, Thm. 1.2]{CDW}, there is a unique germ of a local symplectic groupoid $G=(P,\omega,\a,\b,\i,m)$ such that the source map is the given $\a$ and the units is the given $M\hookrightarrow P$. Following the introduction, we denote this construction by
$ (P,\omega,M,\a) \mapsto G =: \G(P,\omega,M,\a)$ where the output is the germ of local symplectic groupoids defined by the strict realization.
In this local groupoid $G$, the multiplication of two elements can be recovered by hamiltonian flows of $\a$-basic functions on $P$, as follows. Given a function $f \in C^\infty(M)$, we have the associated hamiltonian flow $\ham^{\a^*f}_u$ for the pull-back $\a^*f \in C^\infty(P)$. This flow is left-invariant for the multiplication on $G$ so that
$$ \ham^{\a^*f}_u(z) = m(z,\ham^{\a^*f}_u(\a(z))),$$
for $z$ close enough to the units $M\hookrightarrow G$. 
%
%
Then, for $z_2=\ham^{\a^*f}_{u=1}(\a(z_1))$, we get that
$$ m(z_1, z_2) = \ham^{\a^*f}_{u=1}(z_1).$$
Recall that the target $\b$-fibers are given as the leaves of the symplectic orthogonal distribution to the $\a$-fibers. Using $\b$, we can also form $2$-parameter families of composable elements. For example, noticing that $\ham^{\b^*f}_u$ is right-invariant for $m$, we have, for $x\in M$ and $f_1,f_2 \in C^\infty(M)$,
\begin{equation}\label{eq:2parfam}
(u_1,u_2) \mapsto (\ham^{\b^*f_1}_{u_1}(x), \ham^{\a^*f_2}_{u_2}(x), \ham^{\b^*f_1}_{u_1}\ham^{\a^*f_2}_{u_2}(x) ) \in gr(m)\subset P \times P \times P
\end{equation}
We notice that since $\{ \b^*f_1,\a^*f_2 \}_\omega =0$, the corresponding hamiltonian flows on $(P,\omega)$ commute.
In the following section, we will be interested in some induced 1-parameter subfamilies of the above. The first is obtained by setting $u_1=1$, denoting $z=\ham^{\b^*f_1}_{1}(x)$ and $f=f_2$,
\begin{equation}\label{eq:curve1}
u \mapsto (z, \ham^{\a^*f}_{u}(\a(z)), \ham^{\a^*f}_{u}(z) ) \in gr(m)\subset P \times P \times P
\end{equation}
Restricting the 2-parameter family to the diagonal $u_1=u_2$, we get another curve
\begin{equation}\label{eq:curve2}
u \mapsto \gamma_{f_1,f_2,x}(u):= (\ham^{\b^*f_1}_{u}(x), \ham^{\a^*f_2}_{u}(x), \ham^{\b^*f_1+\a^*f_2}_{u}(x) ) \in gr(m)\subset P \times P \times P
\end{equation}
We shall use these curves to get formulas for local symplectic groupoid generating functions.

\begin{example} (The spray local symplectic groupoid, \cite{CMS1,CMS2})\label{ex:spraygd}
	Recal the spray strict symplectic realization of Example \ref{ex:sprayreal} defined by a Poisson spray $V$. The associated (germ of) local symplectic groupoid 
	$$ G_V := \G(T^*M,\omega_V,0^{T^*},q)$$
	is called spray local symplectic groupoid, and was studied in  \cite{CMS1,CMS2}. The source, target and inversion maps of $G_V$ are given by 
	$$ \a_V(a) = q(a), \ \b_V = q \phi^V_{u=1}(a), \ \i_V(a) = - \phi^V_{u=1}(a)  $$
	near the identities.
	In \cite{CMS1}, it was also shown that it provides a model for any other local symplectic groupoid over the same $(M,\pi)$: if $G'$ is another local symplectic groupoid integrating $(M,\pi)$, then there is a naturally defined isomorphism of germs
	$$ exp_V: G_V \to G'. $$
\end{example}

\section{Existence and characterization of generating function data}\label{sec:cangenf}
In this Section, we first show how to construct generating function data for arbitrary local symplectic groupoids. We then specialize to the case of coordinate Poisson manifolds and provide explicit formulas for the underlying generating functions.

\subsection{The general existence result}\label{subsec:generalexist}
As mentioned in the introduction, the first general existence result (Theorem \ref{thm:genexist} below) can be deduced by combining more general results obtained in the context of symplectic microgeometry by Cattaneo-Dherin-Weinstein, in particular from two papers \cite{CDW2,CDW3} in the series devoted to that subject. Below, we detail the special arguments and ingredients needed to prove Theorem \ref{thm:genexist} directly and mention how they relate to the general theory of symplectic microgeometry in Remark \ref{rmk:microgeom} below. 

Let $G$ be a local symplectic groupoid integrating $(M,\pi)$. Our first aim is to distinguish symplectic embeddings $\nu$ that are adapted to $G$, in a sense made precise below, which will help us in defining interesting generating function data $(\nu,S)$ for $G$.
To that end, we first recall that the tangent space at an identity arrow naturally splits as
\begin{equation*} T_{1_x}P = T1(T_xM)\oplus Ker(D_{1_x}\b) = Ker(D_{1_x}\a)\oplus T1(T_xM).\end{equation*}
Using this decompositions, we have the following description of the tangent space at $(1_x,1_x,1_x) \in P\times P \times P$
to the graph of the multiplication map,
\begin{eqnarray} T_{(1_x,1_x,1_x)} gr(m_G)& = & Im\left(C_x:T_xM \times Ker(D_{1_x}\a)  \times Ker(D_{1_x}\b) \to T_{(1_x,1_x,1_x)}(P\times P \times P)\right) \nonumber \\ 
C_x(v,a,b)&:=& (a+T1(v),b+T1(v),a+b+T1(v))
 \label{eq:T1grm}
\end{eqnarray}
To fix the normal directions to our tubular neighborhoods, let us consider a lagrangian subspace $N_x \subset (T_{1_x}P,\omega|_{1_x})$ for each $x\in M$, defining a smooth subbundle $N\subset TP|_{1(M)}$ which is transverse to $T1(TM)$ (which is also lagrangian by definition). The vector space
\begin{equation}\label{eq:defEN}
E^N_x := \{ (T1(v_1),T1(v_2), c): v_1,v_2\in T_xM, \ c \in N_x \}\subset T_{(1_x,1_x,1_x)}(\overline P \times \overline P \times P)
\end{equation}
defines a lagrangian complement to $T_{(1_x,1_x,1_x)} gr(m_G)$ inside $T_{(1_x,1_x,1_x)} \P_P$ for each $x$.

\begin{example}\label{ex:EN0}
	Consider the symplectic groupoid $G_0$ given by addition of covectors, as in Example \ref{ex:Gzeropois}. In this case, $(P,\omega,1)=(T^*M,\omega_c,0)$ and thus we can take $N_x =\{\frac{d}{dt}|_{t=0}(ta): a \in T^*_xM  \}\subset T_{0_x}(T^*M)$ to be the natural "vertical" lagrangian complement to $Im(T_x1)=Im(T_x0)$. The resulting complement to $T_{(1_x,1_x,1_x)} gr(m_{G_0})$ is given by
	$$ E^N_x := \{ (T0(v_1),T0(v_2),\frac{d}{dt}|_{t=0}(ta)), v_1,v_2\in T_xM, \ a \in T^*_xM\}.$$
\end{example}
We are now ready to define our special symplectic embeddings. 

\begin{definition}\label{def:adaptedframing}
An {\bf adapted framing for $m_G$ near the units} is a symplectic embedding
$$ \nu: (T^*X,\omega_c) \dtod{0^{T^*X}(X_0)} \overline{P} \times \overline{P} \times P=\P_P$$
with
$$ X:= T^*M\times_M T^*M \supset X_0:=\{ (0_x,0_x): \ x\in M \},$$
satisfying:
\begin{enumerate}
	\item $\nu(0^{T^*X}(0_x,0_x)) = (1_x,1_x,1_x)=:1^{(3)}_x \in gr(m_G)$ for all $x\in M$,
	\item for each $l_0=(0_x,0_x)\in X_0, \ x\in M$, the differential $D_{0^{T^*X}(l_0)} \nu$ maps the natural "vertical" lagrangian $\{\frac{d}{dt}|_{t=0}(tA): A\in T_{l_0}^*X \}\subset T_{0^{T^*X}(l_0)}(T^*X)$ onto $E^N_x \subset T_{(1_x,1_x,1_x)}\P_P$, as given in eq. \eqref{eq:defEN}, for some lagrangian complement $N$ of $Im(T1)$ in $(T_1P,\omega|_1)$.
\end{enumerate}
Such a framing $\nu$ is said to be {\bf based on $G_0$} if there exists a symplectomorphism $\mu:(P,\omega) \dtod{1(M)} (T^*M,\omega_c)$ such that $\mu(1(M))=0^{T^*M}(M)$ and $$Im(\nu) =_{1^{(3)}(M)} (\mu\times \mu\times \mu)^{-1}(gr(m_{G_0})),$$ where $G_0$ is the symplectic groupoid structure on $(T^*M,\omega_c)$ given in Example \ref{ex:Gzeropois} which integrates $(M,\pi=0)$.
\end{definition}

Depending on the adapted framing $\nu$ chosen, we will show that we will get a corresponding generating function $S$. The most interesting situation is when $\nu$ describes a simple reference lagrangian neighborhood and the complexity of multiplication $m_G$ is encoded in $S$. In the next subsection, we will see that for a coordinate $M$ the embedding $\nu_c$ defined in \eqref{eq:nucoord} satisfies this property. For arbitrary $M$, a general interesting type of framing is given by those based on $G_0$ and the following Lemma shows that they always exist.
\begin{lemma}\label{lem:existframe}
	For any local symplectic groupoid $G$, there exists an adapted framing near the units for its multiplication map $m_G$ which is based on $G_0$.
\end{lemma}
\begin{proof}
	The proof consists in applying the lagrangian tubular neighborhood theorem in two instances. First, choose a lagrangian complement $N$ for $Im(T1)$ inside $T_1P$ as above. By the lagrangian tubular neighborhood theorem, there is a symplectic embedding
	$$ \mu: (T^*M,\omega_c) \dtod{0} (P,\omega), \ \mu(0_x)=1_x,$$
	such that $D_{0_x}\mu$ takes the natural vertical lagrangian in $T_{0_x}(T^*M)$ into $N_x$. Consider the induced local symplectic groupoid structure $\tilde G$ on $\tilde P=T^*M$, so that $\mu$ defines an isomorphism of germs near the units. The Lemma will follow if we show that its statement holds for $\tilde G$. To that end, we observe that, on top of $gr(m_{\tilde G})$, we have an additional lagrangian $L_0=gr(m_0)\subset \P_{T^*M}$ where
	$$ m_0 :X=T^*M\times T^*M\to T^*M, \ a,b\mapsto a+b$$
	which we can use as a reference. (This corresponds to $gr(m_{G_0})$ for the groupoid $G_0$ given in Example \ref{ex:Gzeropois}.) Notice that $(0_x,0_x,0_x)\in gr(m_{\tilde G})\cap L_0$ for each $x\in M$. We then use the lagrangian tubular neighborhood theorem once more, this time applied to $L_0 \subset \P_{T^*M}$, yielding
	$$ \tilde \nu: T^*X \dtod{0^{T^*X}} \P_{T^*M}, \ \tilde \nu(0^{T^*X}(a,b))=(a,b,a+b)$$  
	and chosen so that $D_{0^{T^*X}(0_x,0_x)}\tilde \nu$ maps the natural vertical lagrangian in $T_{0^{T^*X}(0_x,0_x)}(T^*X)$ into the one given in Example \ref{ex:EN0}. This finishes the proof.
\end{proof}

We now show that, relative to any adapted framing near the units, the graph of the multiplication map results \emph{horizontal}.
\begin{lemma}
	Let $G$ be a local symplectic groupoid and $\nu$ an adapted framing for $m_G$ near the units. Then, there exists a smooth $1$-form $\theta\equiv \theta^{G,\nu}: X \dtod{X_0} T^*X$ which is closed, $d\theta=0$, and such that
	$$ \nu^{-1}(gr(m_G)) =_{0^{T^*X}(X_0)} \{(l,\theta|_l): l\in X \text{ near } X_0 \}.$$
	Moreover, the germ of $\theta$ around $X_0$ is uniquely determined by $G$ and $\nu$.
\end{lemma}

\begin{proof}
		We denote 
		$$L:= \nu^{-1}(gr(m_G)) \subset T^*X.$$
		Notice that $L\neq \emptyset$ since it contains $0^{T^*X}(X_0)$ by the definition of adapted framing near the units. Since $\nu$ is a symplectic embedding and $gr(m_G)\subset \P_{T^*M}$ is an embedded submanifold, then $L\subset T^*X$ is an embedded lagrangian submanifold.
		
		Let us denote $Q:L \to X$ the restriction of the canonical projection $T^*X \to X$ to $L$. We first show that the differential of $Q$ at any point $\hat l_0(x):=0^{T^*X}(0_x,0_x)\in L$ is an isomorphism.
		Within this proof, we shall ommit the inclussion $0^{T^*X}$ of $X$ inside $T^*X$ to avoid overcomplicated formulas, and notice that we thus have $\hat l_0(x)\equiv l_0(x)=(0_x,0_x) \in X_0 \subset L$. We want to show that 
		$$D_{l_0(x)}Q: T_{l_0(x)}L \to T_{l_0(x)}X$$ 
		is an isomorphism. Since $dim(L)=3dim(M)=dim(X)$, it is enough to show that the above map is injective.
		
		To that end, for each $x\in M$, consider the linear map
		$$ A_x: T^*_x M \oplus T^*_x M \oplus T_x M \to T_{l_0(x)}(T^*X), \ (p_1,p_2,v) \mapsto X^{\nu^*H_{p_1,p_2}}|_{l_0(x)} + D_xl_0(v) $$ 
		where $x\mapsto l_0(x)\in X_0$ is seen as a map, the hamiltonian vector field of $\nu^*H_{p_1,p_2}$ is taken in $(T^*X,\omega_c)$ and 
		$$ H_{p_1,p_2}: \P_{P} \dtod{\ut_x} \R, (z_1,z_2,z_3) \mapsto -f_1(\b(z_1)) - f_2(\a(z_2)) + \left(f_1(\b(z_3))+f_2(\a(z_3))\right)$$
		with $f_j \in C^\infty(M)$ denoting any functions satisfying $d_xf_j = p_j$ for $j=1,2$. We first claim that
		\begin{equation}\label{eq:ImATL} Im(A_x) =  T_{l_0(x)}L.\end{equation}
		To see this, we recall that $\nu(l_0(x))=\ut_x \in gr(m)$ and compute, using that $\nu$ is symplectic,
		$$ T\nu \circ A_x(p_1,p_2,v) = X^{H_{p_1,p_2}}|_{\ut_x} + D_x\ut(v)$$
		where now the hamiltonian vector field of $H_{p_1,p_2}$ is taken in $\P_{P}$. Notice that the function $H_{p_1,p_2}$ was defined so that its hamiltonian flow starting at $\ut_x \in \P_{T^*M}$ yields the curve given in \eqref{eq:curve2}, which lies in $gr(m)$ by construction. It thus follows that $Im(A_x) \subset  T_{l_0(x)}L$. We now show that $A_x$ is injective, so that eq. \eqref{eq:ImATL} follows by dimension counting. If $(0,0,v)\in Ker(A_x)$ then $v=0$ since $D_x\ut$ is injective.
		Now, consider local canonical coordinates $(x^j,y_j)$ on $(P,\omega)$ defined near $1_x$ so that $1(M)$ is cut out by $y_j=0$. Using that $\a(1_x)=x=\b(1_x)$, it is easy to verify that
		$ X^{H_{p_1,p_2}}|_{\ut_x} = 0$ iff $\partial_{x^j}|_{y=0}f_1(\beta(x,y)) = 0 =\partial_{x^j}|_{y=0}f_2(\a(x,y))$ iff $p_1=p_2=0$. This shows that $A_x$ is injective as wanted.
		
		Finally, to verify that $D_{l_0(x)}Q$ is injective, we only need to show that 
		$$ K_x:=Ker(TQ \circ A_x) = 0.$$
		Suppse that $(0,0,v)\in K_x$, since $Q|_{X\hookrightarrow T^*X} = id_X$ and $l_0$ is an immersion, this implies that $v=0$. Suppose that $(p_1,p_2,0) \in K_x$, this is equivalent to
		$$ TQ(X^{\nu^*H_{p_1,p_2}}|_{l_0(x)}) = 0.$$
		Using the fact that $\nu$ is a symplectic embedding, and that it was chosen so that its differential at $l_0(x)\in X$ maps the natural horizontal-plus-vertical lagrangian splitting $T_{l_0(x)}(T^*X)\simeq T_{l_0(x)}X \oplus T^*_{l_0(x)}X $ into the lagrangian splitting $T_{1^{(3)}_x}L_0 \oplus E_x \simeq T_{\ut_x}\P_{T^*M}$, with $E_x$ given in \eqref{eq:defEN}, we conclude that
		$$ \langle dH_{p_1,p_2}|_{\ut_x}, e\rangle=0 \ \forall e \in E_x \subset T_{\ut_x}\P_P$$
		Using $\a(1_x)=x=\b(1_x)$ again and the definition of $E_x$, we get that, in canonical coordinates near $1_x\in P$ as above, the previous condition implies $p_1=p_2=0$.
		This shows that $K_x=0$, completing the proof of $D_{l_0(x)}Q$ being an isomorphism.
		
		To finish the proof of the Lemma, we recall a well known generalization of the inverse function theorem.
		\begin{ftheorem}\label{thm:geninvfun}
			Let $f:X_1 \to X_2$ be a smooth map between smooth manifolds and $N_j \subset X_j, \ j=1,2$ be embedded submanifolds. Assume that $f|_{N_1}$ defines a diffeomorphism onto $N_2$, and that $D_xf$ is an isomorphism for all $x\in N_1$. Then, there exist open neigborhoods $U_j$ of $N_j$ in $X_j$, $j=1,2$, such that $f|_{U_1}$ is a diffeomorphism onto $U_2$.
		\end{ftheorem}
		
		By theorem \ref{thm:geninvfun}, since $D_{l_0(x)}Q$ is an isomorphism for each $x\in M$, we conclude that $L$ is \emph{horizontal} near $0^{T^*X}$: there exist neigborhoods $V\subset X$ of $X_0$ and $W\subset T^*X$ of $0^{T^*X}$, together with a smooth 1-form $\theta \in \Omega^1(V)$, such that
		$$ L\cap W = \{ (l,\theta|_l): l \in V\}.$$
		Since $L$ is lagrangian, then $d\theta =0$. 
\end{proof}

Finally, we show that $\theta$ admits a potential, thus defining a generating function for $gr(m_G)$ relative to $\nu$.
To that end, we consider the homogeneous structure 
$$h_\lambda(a,b)= (\lambda a, \lambda b), \ \lambda\in \R, (a,b)\in T^*M\times_M T^*M = X.$$
The fixed points of this structure are given by $Im(h_0)=X_0$, and the idea is to use the contracting homotopy induced by $h_\lambda$ on differential forms. 
Since $0^{T^*X}(0^{T^*M}_x,0^{T^*M}_x) \in \nu^{-1}(gr(m_G))$ for all $x\in M$, it follows directly that
$$ h_0^*\theta^{G,\nu} = 0$$
for any adapted framing near the units $\nu$. For any such $\nu$, we consider the induced $\theta\equiv \theta^{G,\nu}$ and we claim that
$$ S:= \int_0^1 du \ \frac{1}{u} i_\Eu h^*_u\theta \text{ satisfies } dS=\theta,$$
where $\Eu|_l = \frac{d}{d\lambda}|_{\lambda=1} h_\lambda(l)\in T_l X$ defines the Euler vector field associated to $h_\lambda$. To verify this claim, first notice that $S$ is well defined in a neighborhood of $X_0\subset X$ since $u \mapsto h^*_u\theta$ is smooth in $u$ and vanishes when $u=0$, as shown before. To verify that $S$ is a potential for $\theta$, one uses (as in a relative Poincar\'e Lemma)
$$ \theta = h_1^*\theta = h_0^*\theta + \int_0^1 du \ \frac{d}{du}[h_u^*\theta]$$
and compute $\frac{d}{du}[h_u^*\theta]= \frac{1}{u}d(h^*_u i_\Eu \theta)$ from $h_{u_1u_2} = h_{u_1}\circ h_{u_2}$.
The function $S$ thus defined satisfies $S(l_0)=0$ for any $l_0\in X_0$, as a consequence of $h_0(l_0)=l_0$ and $\Eu|_{l_0} =0$. 
Lastly, we observe that the germ of $S$ near $X_0$ is uniquely determined by $G$, $\nu$ and the condition  $S|_{X_0}=0$. Indeed, $\theta$ is completely determined by $G$ and $\nu$ near $X_0$, so if there is another $\tilde S$ satisfying $\tilde S(l_0)=0 \forall l_0\in X_0$ and also generating $gr(m_G)$ through $\nu$,  then, $d\tilde S=dS$ as a consequence of the previous formula, and the possible additive constant for the germ of $\tilde S - S$ around each connected component of $X_0\simeq M$ is fixed to zero by $\tilde S(l_0)=0 \forall l_0\in X_0$. We note that, when $M$ is connected, the germ of $S$ is determined by $S(l_0)=0$ for any particular $l_0\in X_0\simeq M$. 

We summarize the above results in the following:
\begin{theorem}\label{thm:genexist}
	Every local symplectic groupoid $G$ admits an adapted framing $\nu$ for $m_G$ near the units,
	$$ \nu: T^*X \dtod{0^{T^*X}(X_0)} \P_P, \ X=T^*M\times_M T^*M\supset X_0=\{(0_x,0_x): \ x\in M \}, \ \nu(0^{T^*X}(0_x,0_x))=(1_x,1_x,1_x),$$ 
	which can be taken based on $G_0$ (see Definition \ref{def:adaptedframing}). For every adapted framing $\nu$, there exists a unique germ of functions  $S\equiv S_\nu:X \dtod{X_0} \R$ such that $(\nu,S)$ defines generating function data for $G$,
	$$  \ \nu^{-1}(gr(m_G))=_{0^{T^*X}(X_0)} \{ (l,dS|_l): l\in X \text{ near } X_0  \}, \text{ and }S|_{X_0}=0.$$
\end{theorem}

\begin{remark}\label{rmk:microgeom}
Let us explain the relation of the above results to the theory of symplectic microgeometry; in particular to \cite{CDW2,CDW3}. First, in \cite{CDW3} it was shown that the germ of local symplectic groupoid multiplication defines a \emph{symplectic micromorphism}. On the other hand, in \cite{CDW2} it was shown that symplectic micromorphisms in general admit a description through generating functions. More specifically, after taking $G$ to a normal form, $P\simeq T^*M$, the pair $(gr(m_G),C:=1^{(3)}_M)$ defines a \emph{lagrangian submicrofold} in $(\overline{T^*M}\times \overline{T^*M} \times T^*M, M^3)$ which turns out to be a \emph{conormal deformation} of the conormal $N^*C$ to $C$, see \cite[Def. 6]{CDW2}. Noting that $N^*C \simeq T^*M\times_M T^*M$, our adapted framing $\nu$ corresponds to the symplectomorphism germ $K$ appearing in that definition together with an associated $1$-form germ $\beta$ (see \cite[Rmk. 7]{CDW2}) which corresponds to the $\theta_{G,\nu}$ in the Lemma above. The fact that general lagrangian submicrofolds are described by these kind of generating function data is proven in \cite[Thm. 8 and Rmk. 10]{CDW2}, assuming a clean intersection hypothesis with the ambient's zero section, and using similar arguments as in the Lemma above. The arguments presented in this subsection can be thus seen as a specialization of similar ones appearing in a more general context in \cite{CDW2}.
\end{remark}

\subsection{Generating functions in coordinate spaces}\label{subsec:generalgenfuncs}
In the following subsections, we will restrict our attention to coordinate Poisson manifolds and derive explicit formulas providing their integration. The setting is as follows: we say that $M$ is a \emph{coordinate space} if it is an open subset of a (finite dimensional, real) vector space; in this case we denote $M^*$ the dual vector space, so that $T^*M=M \times M^*$. We will denote by
$$q(x,p)=x, r(x,p)=p$$
the two projections onto $M$ and $M^*$, respectively. We also use the notation 
$$ x \mapsto p(x), \ p \in M^*$$
when regarding $p:M \to \R$ as a linear function on $M$.
Across the following subsections, we will consider a fixed (smooth) Poisson structure $\pi\equiv \pi^{ij}(x)$ on $M$ and refer to $(M,\pi)$ as to a {\bf coordinate Poisson manifold}.
We will also restrict our attention further to local symplectic groupoids $G$ integrating a coordinate Poisson manifold $(M,\pi)$ such that $$(P_G,\omega_G,M_G)=(T^*M,\omega_c,0^{T^*M}).$$

\label{subsub:sga}

We recall from the introduction that a \emph{coordinate generating function} for such a $G$ is a smooth function 
\begin{equation}\label{eq:defScoord} S: X \dtod{{X_0}} \R, \ X:= M^*\times M^* \times M \supset X_0 := \{(0,0,x):x \in M\}\end{equation}
such that $(\nu_c,S)$, with $\nu_c$ given in \eqref{eq:nucoord},  defines generating function data for $G$, i.e.
\begin{equation*}
gr(m_G) =_{M^{(3)}} \{ ((\partial_{p_1}S,p_1) , (\partial_{p_2 }S, p_2), (x,\partial_x S)): (p_1,p_2,x)\in X \}\subset T^*M\times T^*M \times T^*M.
\end{equation*}
The existence of such coordinate generating functions follows directly from the general Theorem \ref{thm:genexist}.
\begin{corollary}\label{prop:existS}
	Let $G$ be a local symplectic groupoid integrating the coordinate Poisson manifold $(M,\pi)$ satisfying $(P_G,\omega_G,1)=(T^*M,\omega_c,0^{T^*M})$. Then, there exists a coordinate generating function $S:X\dtod{X_0} \R$ for $G$, as in eq. \eqref{eq:defScoord}, whose germ  near $X_0$ is uniquely defined by $S|_{X_0}=0$.
\end{corollary}
\begin{proof}
	By Theorem \ref{thm:genexist}, we only need to check that $\nu_c$ given in eq. \eqref{eq:nucoord} defines an adapted framing near the units for $m_G$. By the hypothesis that $1_x = 0_x\in T^*M$, this follows immediately: indeed, the differential of $\nu_c$ at $A=0^{T^*X}(0,0,x)$ maps the vertical lagrangian subspace inside $T_A(T^*X)$ into the $E_x^N$ given in Example \ref{ex:EN0}.
\end{proof}
 
\begin{remark} (coordinate generating functions vs. general ones)
	Suppose that $M$ is a coordinate space. Notice that there can be adapted frames near the identities, $\nu$, which are different from the coordinate induced one $\nu_c$ of eq. \eqref{eq:nucoord}. In particular, denoting $gr(m_0)\in \P_{T^*M}$ the graph of covector multiplication as in Example \ref{ex:Gzeropois},
	$$ \nu_c^{-1}(gr(m_0)) = \{ ((p_1,x),(p_2,x),(x,p_1+p_2)\} \subset T^*(M^*\times M^* \times M),$$
	while a natural alternative choice of $\nu$, which uses $gr(m_0)$ as a reference (as in Lemma \ref{lem:existframe}), leads to
	$$ \nu^{-1}(gr(m_0)) = \{ ((p_1,0),(p_2,0),(x,0)\} \subset T^*(M^*\times M^* \times M).$$
	The corresponding generating functions are thus also different, even for the same $G_0$:
	$$ S_{\nu_c}(p_1,p_2,x) = (p_1+p_2)x, \ \ S_{\nu} = 0.$$
\end{remark}

We observe that if $S$ is a coordinate generating function for $G$, since $\zero(x)=(x,0)$ defines the groupoid identities, then the source $\a=\a_G$ and target $\b=\b_G$ maps of $G$ are necessarily given by
$$ \a(x,p)=\partial_{p_2}S(p,0,x) $$
$$ \b(x,p) = \partial_{p_1}S(0,p,x)$$
for $p\sim 0$. We also observe that the inversion map is given by $inv(x,p)=(x,-p)$ iff 
$$\partial_{p_1}S(p,-p,x)=\partial_{p_2}S(p,-p,x) \text{ and } \partial_xS(p,-p,x) = 0,$$
for all $(x,p)$ close to $\zero$. This is equivalent to the condition that the function $T^*M \dtod{\zero} \R, \ (x,p)\mapsto S(p,-p,x)$ is locally constant.

\begin{remark}\label{rmk:SforGop}
	If $S$ is a coordinate generating function for a $G$ integrating $(M,\pi)$, then the same $S$ is also a coordinate generating function for $\overline{G}$ in which the symplectic structure is the opposite (recall Remark \ref{rmk:Gopom}), which integrates $(M,-\pi)$.
\end{remark}


\begin{example}\label{ex:S0} We know that $G_{0}$, as defined in Example \ref{ex:Gzeropois} by addition of covectors in $(T^*M,+)$, yields an integration of $(M,\pi=0)$. Using that $M$ is a coordinate space, it is easy to verify that
	$$ S_0(p_1,p_2,x) := (p_1+p_2)(x) $$
	defines a coordinate generating function for $G_0$. In the case $\pi$ a constant Poisson structure on $M$, an integration $G_\pi$ will be described in Example \ref{ex:Gpiconst} below, and the function
	$$ S(p_1,p_2,x):=(p_1+p_2)(x) + \frac{1}{2} \pi(p_1,p_2)$$
	defines a coordinate generating function for $G_\pi$. 
\end{example}

\begin{example}\label{ex:Sbch}
	Let $\gg$ be a Lie algebra and consider $M=\gg^*$ endowed with the linear Poisson structure $\tilde \pi$, as defined in Example \ref{ex:cotangentH} together with its integration $T^*H_\gg$. Then, the function
	$$S(p_1,p_2,x): = BCH(p_1,p_2)(x),$$
	with $BCH$ denoting the Baker-Campbell-Hausdorff series for $\gg$ (see Example \ref{ex:Hbch}), defines a coordinate generating function for $T^*H_\gg$. This can be verified by checking \eqref{eq:condS} directly using the definition of the multiplication $\tilde m$ on $T^*H_\gg$.
\end{example}

Suppose that $S$ is a coordinate generating function for $G$. The associativity condition for the multiplication $m$ of $G$ results equivalent, for connected $M$, to system of partial differential equations for $S$ which was considered in \cite[\S 1.2]{CDF} where it  was called the \emph{symplectic groupoid associativity} (SGA) equation: for $(p_1,p_2,p_3,x)$ near $0\times 0 \times 0 \times M\subset M^*\times M^*\times M^* \times M$, 
\begin{equation}\label{eq:SGA} S(p_1,p_2,\bar x) + S(\bar p,p_3,x) - \bar p(\bar x) = S(p_1, \tilde p,x) + S(p_2,p_3,\tilde x) - \tilde p (\tilde x),\end{equation}
where $\bar x, \bar p, \tilde x, \tilde p$ are defined as functions of $(p_1,p_2,p_3,x)$ via
\[ \ \bar x=\partial_{p_1}  S(\bar p,p_3,x), \ \ \bar p=\partial_x S(p_1,p_2,\bar x), \  \tilde x = \partial_{p_2} S(p_1,\tilde p,x)  , \ \tilde p = \partial_x S(p_2,p_3,\tilde x).  \]
In \cite{CDF}, they proceed to obtain a formal family of solutions of the SGA equation of the form $S_\e(p_1,p_2,x) \in C^\infty(X)[[\e]]$, where $\e$ is a formal parameter. The authors focus on families which define formal deformations of the solution $S_0$ associated to $\pi=0$ (Example \ref{ex:S0} above) and their general solution is based on the tree-level part of Kontsevich's star product, which will be discussed in Section \ref{sec:formal} below. The analogue of the SGA equations for general (non-coordinate) Poisson manifolds is studied in \cite{Yuki}.

Observe that the data of the germ of the local symplectic groupoid $G$ on $P= T^*M$ admitting a coordinate generating function is equivalent to the data provided by a function $S:X \to \R$ which satisfies the SGA equation. To directly extract the underlying Poisson structure on $M$, using the fact that $\a$ must be a Poisson map and $1_x=(x,0)$, we arrive to the relation 
\begin{equation} \label{eq:pifromS} \pi^{ij}|_x = 2\partial_{p_{1i}}\partial_{p_{2j}}S|_{(0,0,x)}.\end{equation}
In this case, we simply say that \emph{the solution of the SGA equation $S$ integrates $(M,\pi)$}.

\begin{remark} (Non-formal solutions to the SGA equation)
	Corollary \ref{prop:existS} can be seen as a Lie-theoretic method, with focus on the underlying local symplectic groupoid structure, to produce explicit, non-formal, smooth solutions for the SGA equation. One advantage of this method is that it is not perturbative (i.e. not solving for successive  corrections along $t\pi$ to the case $t=0$) but, rather, a consequence of smooth implicit methods. Moreover, we will see below that the solution can be characterized by explicit analytic formulas.
\end{remark}

\begin{remark} (The SGA equation and lagrangian relations)\label{rmk:lagrels} 
	Let us consider lagrangian relations between symplectic manifolds, as described in e.g. \cite{GSbook} (see also \cite{CDW2}). The multiplication map $m$ of $G$ can be understood as a lagrangian relation
	$$ m:T^*M \times T^*M \dto T^*M \text{ meaning } gr(m)\hookrightarrow \overline{T^*M} \times \overline{T^*M} \times T^*M \text{ is lagrangian,} $$
	where cotangent bundles are endowed with their canonical symplectic structures, $\omega_c$, and the overline denotes the opposite symplectic structure, $-\omega_c$. The composition $m\circ (m\times id):T^*M\times T^*M \times T^*M \dto T^*M$ also defines a a lagrangian relation and we denote the underlying lagrangian submanifold by
	$$ \Gamma_{(12)3} \hookrightarrow \overline{T^*M} \times \overline{T^*M} \times \overline{T^*M} \times T^*M.$$
	Just as $gr(m)$ admits generating function data $(\nu_c,S)$, the lagrangian $\Gamma_{(12)3}$ admits generating function data $(\nu_c',S_{(12)3})$ where $\nu_{c}': T^*(M^*\times M^* \times M^* \times M) \simeq \overline{T^*M} \times \overline{T^*M} \times \overline{T^*M} \times T^*M$ is the obvious extension of $\nu_c$ and
	$$ S_{(12)3}:M^*\times M^* \times M^* \times M \dtod{0\times 0 \times 0 \times M} \R, \ S_{(12)3}(p_1,p_2,p_3,x):= S(p_1,p_2,\bar x) + S(\bar p,p_3,x) - \bar p(\bar x),$$
	where $\bar x, \bar p$ are the functions of $p_1,p_2,p_3,x$ induced by $S$ as defined in eq. \eqref{eq:SGA}. Analogously, $m\circ(id\times m)$ defines a lagrangian relation with underlying $\Gamma_{1(23)}\hookrightarrow \overline{T^*M} \times \overline{T^*M} \times \overline{T^*M} \times T^*M$ admiting generating function data $(\nu_c',S_{1(23)})$ with
	$$ S_{1(23)}(p_1,p_2,p_3,x):=  S(p_1, \tilde p,x) + S(p_2,p_3,\tilde x) - \tilde p (\tilde x),$$
	and $\tilde x,\tilde p$ the functions defined as in eq. \ref{eq:SGA}. The SGA equation is then equivalent to the statement that the two generating functions coincide,
	$$ S_{(12)3}  =_{0\times 0 \times 0 \times M} S_{1(23)}.$$
	Notice that this implies the associativity identity $m\circ (m\times id) = m\circ(id \times m)$ near the identities. The structure of the generating functions $S_{(12)3}$ and $S_{1(23)}$ can be understood in terms of composition of lagrangian relations, following the general procedure of \cite[\S 5.6, \S 5.7]{GSbook} (see also \cite[\S 3.3]{CDW2}). We mention that, in this procedure, the terms appearing in $S_{(12)3}$ and $S_{1(23)}$ which involve the function $\px(x,p)=p(x)$ come from the need of inserting the 'flip' transformation 
	$$ \Fou: T^*M \to \overline{T^*(M^*)}, (x,p) \mapsto (\tilde p, \tilde x)= (p,x)$$
	after applying $(m\times id)$ and $(id\times m)$ and before composing with $m$. Following \cite{GSbook} further, we observe that $\Fou$ is the 'semiclassical limit' version of Fourier transform in the context of Fourier integral operators. Moreover, the oscillatory integrals underlying the SGA equation where identified in \cite{CDF} and related to quantization formulas (see also Section \ref{subsec:quant} below and \cite{CDW4}).
	%
	%
	%
\end{remark}

\begin{remark} (The simplicial meaning of the SGA)\label{rmk:simplSGA}
	The function $S$ induces an isomorphism $I_S:X \dtod{X_0} gr(m)$. Composing $S:X \to \R$ with the inverse $I_S^{-1}$ and with the identification $gr(m)\simeq G^{(2)}, (z_1,z_2,m(z_1,z_2))\mapsto (z_1,z_2)$, we obtain a \emph{local $2$-cochain for $G$}, $ \tilde S : G^{(2)} \dtod{ M^{(2)} } \R.$
	The SGA equation \eqref{eq:SGA} is equivalent to 
	$$ \delta \tilde S =_{M^{(3)}} F : G^{(3)} \dtod{M^{(3)}} \R,$$
	where $\delta$ denotes the simplicial differential on local groupoid cochains and
	$$ F(z_1,z_2,z_3) = \px (m(z_2,z_3)) - \px (m(z_1,z_2)), \ \px:T^*M \to \R, \px (x,p)=p(x).$$
\end{remark}

\subsubsection*{Integral formulas characterizing $S$}\label{subsub:formulasS}


We now show how to obtain formulas for a given coordinate generating function $S$ for $G$, expressed as integrals along flows. 
Let
$$ I_S: M^* \times M^* \times M  \ _{(0 \times 0 \times M)} \dto T^*M \times T^*M \times T^*M $$
$$I_S(p_1,p_2,x)= ((\partial_{p_1}S,p_1) , (\partial_{p_2 }S, p_2), (x,\partial_x S))$$
so that $Im(I_S)$ coincides with the graph of multiplication $gr(m)$ near the identities embedding $M^{(3)}\hookrightarrow T^*M \times T^*M \times T^*M$. 
The idea is to consider a curve $$u\mapsto \gamma(u)=((x_1(u),p_1(u)),(x_2(u),p_2(u)),(x_3(u),p_3(u)) ) \in gr(m)$$ which stays close enough to the identities $M^{(3)}$ so that, since $S$ is a generating function,
\begin{equation}\label{eq:gaIS}
\gamma(u) = I_S(p_1(u), p_2(u),x_3 (u) ).
\end{equation}
Using
$$ S((p_1(1),p_2(1),x_3(1) ) = S(p_1(0),p_2(0),x_3(0) ) + \int_0^1du \ \frac{d}{du}S(p_1(u),p_2(u),x_3(u) )$$ 
and expressing $\partial_{p_1}S$, $ \partial_{p_2}S$ and $\partial_x S$ in terms of curve components via \eqref{eq:gaIS}, we arrive at the following formula:
\begin{equation}\label{eq:genScurve}
S((p_1(1),p_2(1),x_3(1) ) = S(p_1(0),p_2(0),x_3(0) ) + \int_0^1du \ \left[ \dot p_1(u) (x_1(u)) + \dot p_2(u) (x_2(u)) + p_3(u)(\dot x_3(u))\right]
\end{equation}
To get concrete expressions, we first consider the 2-parameter family 
$$ (u_1,u_2) \mapsto \gamma(u_1,u_2) \in gr(m) $$
defined by eq. \eqref{eq:2parfam} with $f_1=p_1\sim 0$, $f_2=p_2 \sim 0$ and $x=x_0$, and apply \eqref{eq:genScurve} twice. After some standard calculus manipulations involving $\{\b^*p_1,\a^*p_2 \}_c =0$, one arrives to the system

\begin{eqnarray}
S(rz_1(1),rz_2(1),qz_3(1,1)) =& S(0,0,x_0) + \px (z_1(1)) + \px (z_2(1)) + \nonumber \\
& \ +  \int_0^1 	du_1 \int_0^1 du_2 \{L_E(\a^*p_2),\b^*p_1, \}_c(z_3(u_1,u_2)), \nonumber \\
z_1(u_1)= \phi^{\b^*p_1}_{u_1}(x_0,0), \ \  z_2(u_2)& =  \phi^{\a^*p_2}_{u_2}(x_0,0) , \ \ z_3(u_1,u_2)=\phi^{\b^*p_1}_{u_1}\phi^{\a^*p_2}_{u_2}(x_0,0) ,
\label{eq:S2parint}
\end{eqnarray}
where we recall that $E=p\partial_p$ is the Euler vector field on $T^*M\to M$ and $\px (x,p)= p(x)$.
In particular, when $S(0,0,x_0)=0 \ \forall x_0$ and taking $p_2=0$ and $p_1=0$ alternatively, we obtain
\begin{equation}\label{eq:Sp0x} S(p,0,x) = p(x) = S(0,p,x) \end{equation}
for $p\sim 0$.
Finally, we obtain a formula for $S$ which only uses the strict symplectic realization data $(T^*M,\omega_c,0^{T^*},\a)$ underlying $G$.

\begin{proposition}
	Let $\a:T^*M \dtod{\zero} M$ be smooth map such that $(T^*M,\omega_c,0^{T^*},\a)$ defines a strict symplectic realization of the coordinate Poisson $(M,\pi)$. Denote by $G=\G(T^*M,\omega_c,0^{T^*},\a)$ the induced local symplectic groupoid and $S$ a representative of the unique germ of coordinate generating functions for $G$ given in Corollary \ref{prop:existS}. Then, for all $(p_1,p_2,x_1)$ with $p_1,p_2$ small enough, the following system of equations holds
	\begin{eqnarray}\label{eq:Sint1dgen}
	S(p_1,rz_2(1),qz_3(1))  = p_1(x_1) + \px (z_2(1)) + \int_0^1du \ [L_E(\a^*p_2)(z_2(u)) - L_E(\a^*p_2)(z_3(u))], \nonumber \\
	z_2(u)=\phi_u^{\a^*p_2}(\a(x_1,p_1),0),  \ z_3(u)=\phi_u^{\a^*p_2}(x_1,p_1).
	\end{eqnarray}
\end{proposition}

\begin{proof}
	The proof consists on evaluating \eqref{eq:genScurve} for the curve given by \eqref{eq:curve1} with $f=p_2$ and $z=z_1$, yielding
	$$ u \mapsto ((x_1,p_1), \phi_u^{\a^*p_2}(\a(x_1,p_1),0), \phi_u^{\a^*p_2}(x_1,p_1) ) \in gr(m).$$
	The final expression comes as straightforward consequences of: the fact that $S(0,0,x)=0$ by construction, integration by parts, eq. \eqref{eq:Sp0x}, and the definition \eqref{eq:hamflow} of the hamiltonian flows.
\end{proof}

\begin{remark}
	Formula \eqref{eq:Sint1dgen} can be useful to find quantizations of coordinate Poisson manifolds through Fourier Integral Operators, as described in \cite{CDW4}. In particular, when $\pi=\omega^{-1}$ is constant and symplectic, via Fourier transform we can arrive from $S_\pi$ (see Example \ref{ex:S0}) to the well-known formula for the integral kernel of Moyal product in terms of symplectic areas of triangles.
\end{remark}

\subsubsection*{Relation to local groupoid 2-cocycles}

In this subsection, we consider a local symplectic groupoid $G$ with $(P_G,\omega_G,M)=(T^*M,\omega_c,\zero)$ and $M$ a coordinate space. We denote by $\theta_c \in \Omega^1(T^*M)$ the Liouville 1-form $\theta_c \equiv p dx$, so that $\omega_c = d\theta_c$ with our conventions. 
We first show that there is a germ of a groupoid $2$-cocycle $C\equiv C_G$ canonically associated to $G$.

Since $\omega_c$ is multiplicative in $G$ (recall Section \ref{subsec:lsg}), we have that 
$$ \delta \omega_c := pr_2^*\omega_c - m^*\omega_c + pr_1^*\omega_c  =_{M^{(2)}} 0,$$
where $\delta: \Omega^k(G^{(l)}) \to \Omega^k(G^{(l+1)})$ denotes the simplicial differential associated to the local groupoid structure (we follow the conventions summarized in \cite{CDvan}). From the commutation between $\delta$ and de Rham differential $d$ on forms, $d(\delta \theta_c) =0$. By an argument analogous to that in the proof of Theorem \ref{thm:genexist}, it follows that that on a contractible neighbourhood of $M^{(2)}\hookrightarrow G^{(2)}$,
\begin{equation}\label{eq:2cocycle} \delta \theta_c  = d C, \ \ C: G^{(2)} \dtod{M^{(2)}} \R\end{equation}
and that the germ of $C$ is uniquely determined by $C|_{M^{(2)}}=0$. In homological terms, $C$ defines a \emph{local groupoid $2$-cochain}, which is \emph{normalized}. It also follows from the definitions that $C$ defines a $2$-cocycle $\delta C=0$.
	We say that $C\equiv C_G$ is the {\bf canonical $2$-cocycle germ} associated to $G$.

\begin{remark}
	A similar construction can be performed to associate to a general local symplectic groupoid $G$, not necessarily over a coordinate space $M$, a germ of a groupoid $2$-cocycle $C: G^{(2)}\dtod{M^{(2)}} \R$ depending on auxiliary data. The auxiliary data needed in the general case consists of a choice of a potential $\theta$ for $\omega_G$ near the identities, $d\theta =_M \omega_G$, and a contracting homotopy for a neighborhood of $M^{(2)}$ inside $G^{(2)}$. With these, one can fix $C$ by $\delta \theta =_{M^{(2)}} dC$ (where the contraction is used) and $C|_{M^{(2)}}=0$.
\end{remark}

We now relate this $2$-cocycle $C$ to the underlying coordinate generating function, $S$, given in Corollary \ref{prop:existS}. To that end, we recall that a generating function $S$ induces a map
\begin{equation}\label{eq:Jcocy} J:X \dtod{X_0} G^{(2)}, \ (p_1,p_2,x) \mapsto ((\partial_{p_1}S|_{(p_1,p_2,x)}, p_1), ((\partial_{p_2}S|_{(p_1,p_2,x)}, p_2)) \end{equation}
which defines a diffeomorphism onto a neighborhood of $M^{(2)}$ inside $G^{(2)}$.

\begin{lemma}\label{lem:CandS}With the definitions above,
	\[ J^*C = L_{\Eu} S - S\] where $\Eu = p_1 \partial_{p_1} + p_2 \partial_{p_2}$ is the Euler vector field on $X\to X_0$. 
\end{lemma}
\begin{proof}
	The Lemma follows directly by evaluating $J^*(\delta \theta_c)$ in \eqref{eq:2cocycle}, and using the identities
	$$ J^*pr_2^*(\theta_c) = p_{2j}d(\partial_{p_{2j}}S), \ J^*m^*(\theta_c)=\partial_{x^j}S dx^j, \  J^*pr_1^*(\theta_c) = p_{1j}d(\partial_{p_{1j}}S),$$
	which in turn follow directly from the definitions, and recalling that $S|_{X_0}=0$ to check the normalization condition.
\end{proof}

In Remark \ref{rmk:2cocynatffam} below, we provide an integral formula relating $C$ and $S$ in the case of "natural $1$-parameter families" (to be defined in Section \ref{subsub:smooth1par} below). 

\subsection{The canonical $G_\pi$ integrating coordinate Poisson manifolds}\label{subsec:canG}

We now begin our construction of a canonical local symplectic groupoid (germ) $G_\pi$ integrating a coordinate Poisson manifold $(M,\pi)$. The first ingredient is a realization map $\a\equiv \a_\pi$ introduced by Karasev \cite{Karasev}, which we now proceed to describe.

For $p\in M^*$, denote $\varphi_{\pi,p}:M \times \R \ _{M\times 0}\dto M, (x,t)\mapsto \varphi_{\pi,p}^t(x)$ the flow of the differential equation \eqref{eq:Ppeq} mentioned in the introduction. We consider the map $\a_\pi: M \times M^* \ _{M\times 0}\dto M, (x,p)\mapsto \a_\pi(x,p)$ defined implicitly by the analytic expression
\begin{equation}
\label{eq:alphapi}\int_0^1 du \ \varphi_{\pi,p}^u (\a_\pi(x,p)) = x \ \text{for all } p \ \text{close to $0$.}
\end{equation}
Since $\varphi^u_{\pi,0} = id_M, \forall u$, by the implicit function theorem, it follows that $\alpha$ is indeed well defined as a smooth function in a neighborhood of $M\times 0$ (see details in \cite{CD}).
From \cite{Karasev}, we recall that $\a_\pi$ defines a \emph{strict symplectic realization} of $(M,\pi)$ whose lagrangian section is $\zero:M \hookrightarrow T^*M, x \mapsto (x,0)$ (see also \cite{CD} for its connection to other realizations).

\begin{example}\label{ex:apisimpleexs}
	When $\pi=0$, then $\a_\pi(x,p)=x$. A bit more generally, when $\pi^{ij}(x)=\pi^{ij}$ is constant in $x$, then $\varphi_{\pi,p}^u(x_0)^i= x^i_0 + u\pi^{ij}p_j$ so that 
	$$ \a_\pi(x,p)^i= x^i - \frac{1}{2}\pi^{ij}p_j = (x + \frac{1}{2}\pi^\sharp p)^i.$$
\end{example}

\begin{example}\label{ex:apilin}
	Let $\gg$ be a Lie algebra and $M=\gg^*$ be endowed with the linear Poisson structure $\tilde \pi$ introduced in Example \ref{ex:cotangentH}. With these ingredients, then
	$$ \a_{-\tilde \pi} = \tilde \a\circ \Fou: T^*M = \gg^* \times \gg \dtod{\gg^*\times 0} \gg^* $$
	where $\Fou: T^*\gg^* \to T^*\gg, \ \Fou(x,p) = (p,x) \in T^*\gg$ denotes the standard 'flip' (see also Remark \ref{rmk:lagrels}) and the map $\tilde \a: T^*\gg \dtod{T^*_0\gg} \gg^*$ is the source map in the cotangent groupoid construction of Example \ref{ex:cotangentH} applied to the local Lie group $H_\gg$ defined by $\gg$ in Example \ref{ex:Hbch}. To show this, one needs to recall that (see \cite[Thm. 1.5.3]{DKbook}) the left-invariant Maurer-Cartan form $\theta_p^L(v):=D_pL_{p^{-1}}(v), \ p\in H_\gg, v\in \gg$ in $H_\gg$ admits the following integral formula for $p\sim 0$,
	$$ \theta^L_p(v) = \int_0^1 du \ e^{-u \ ad_p}v.$$
	(Also recall the similar formula for $\theta^R_p$ in Example \ref{ex:Hbch}.)
\end{example}

For later reference, we also record the following special properties of $\a_\pi$.
\begin{lemma}\label{lma:propapi} If $(x,p)\in Dom(\a_\pi)$,
	\begin{equation}\label{eq:propsalphaini} \a_{t\pi}(x,p)=\a_{\pi}(x,tp), \ \  p(\a_\pi(x,tp)) = p(x), \ \forall t\in [0,1].\end{equation}
	As a consequence, for any $p_0 \in M^*$ seen as a function on $M$, the corresponding hamiltonian flow (recall \eqref{eq:hamflow}) satisfies the following properties
	\begin{equation}\label{eq:rhamu} r\ham^{\a_\pi^*p_0}_u(x_0,0) = up_0, \ \ p_0(q\ham^{\a_\pi^*p_0}_u(x_0,0))=p_0(x_0), \ \forall u. \end{equation}
\end{lemma}
\begin{proof}
	The first identity in \eqref{eq:propsalphaini} follows directly from $\varphi_{t\pi,p}=\varphi_{\pi,tp}$, since \eqref{eq:Ppeq} is linear in both $\pi$ and $p$. The second identity in \eqref{eq:propsalphaini} will be a direct consequence of the skewsymmetry of $\pi^{ij}$ in \eqref{eq:Ppeq}. Taking $t\in[0,1]$ fixed, $\a_{\pi}(x,tp)$ is defined by the equation
$$ x = \int_0^1 du \ \varphi_{\pi,tp}^u (\a_{\pi}(x,tp)) = \int_0^1 du \ \varphi_{\pi,p}^{tu} (\a_{\pi}(x,tp)) $$
where we have used the fact that \eqref{eq:Ppeq} is linear in $p$ in the second equality, so its rescaling is equivalent to time rescaling. On the other hand, 
$$ p(\varphi_{\pi,p}^s(y)) = p(y) \ \forall s,$$
which follows by taking $d/ds$, using \eqref{eq:Ppeq} and the fact $p_ip_j \pi^{ij}(y)=0$ due to skewsymmetry. Then,
$$ p(x) = \int_0^1 du \ p(\varphi_{\pi,p}^{tu} (\a_{\pi}(x,tp)))=\left( \int_0^1 du \right) p(\a_\pi(x,tp)), $$
finishing the second identity in \eqref{eq:propsalphaini}. The first identity in \eqref{eq:rhamu} follows by writing Hamilton's equation on $(T^*M,\omega_c)$ for $H=\a_\pi^*p_0$:
$$ \dot x^i = - p_0(\partial_{p_i}\a_\pi(x,p)), \ \dot p_i = p_0(\partial_{x^i}\a_\pi(x,p)) $$
By the identity we first proved, it follows that $p(u)=up_0$ solves the second equation above. Finally, for the second identity in \eqref{eq:rhamu}, we write $\ham^{\a_\pi^*p_0}_u(x_0,0) = (x(u),up_0)$ by the previous item and compute
$$\frac{d}{du} p_0(q\ham^{\a_\pi^*p_0}_u(x_0,0)) = -p_{0j}\partial_{p_j} [p_0 \circ \a_\pi] (x,up_0) = -\frac{d}{d\e}|_{\e=0} [p_0(\a_\pi(x, u p_0 +  \e p_0))] =0 $$
where we used \eqref{eq:propsalphaini} to conclude $p_0(\a_\pi(x, u p_0 +  \e p_0)) = p_0(x)$ is independent of $\e$. This finishes the proof of the Lemma.
\end{proof}

Now, we define the local symplectic groupoid which is our main object of study in this subsection.
\begin{definition}
	Let $(T^*M,\omega_c,\zero,\a_\pi)$ be the strict symplectic realization defined by eq. \eqref{eq:alphapi} above and denote by $\G$ the construction of eq. \eqref{eq:Gfdef}, detailed in Section \ref{subsec:streal}. The (germ of the) local symplectic groupoid $$G_\pi:=\G(T^*M,\omega_c,0^{T^*M},\a_\pi)$$ will be called the \emph{canonical local symplectic groupoid} associated to $(M,\pi)$.
\end{definition}

The canonical local symplectic groupoid structure $G_\pi$ then has $\zero$ as identity map and $\a_\pi$ as its source map. We now characterize the inversion map.

\begin{lemma}
	The inversion map $\i$ on $G_\pi$ is given by $\i(x,p) = (x,-p)$.
\end{lemma}
\begin{proof}
	Our arguments will be based on the spray groupoid $G_V$ construction, recalled in Examples \ref{ex:sprayreal} and \ref{ex:spraygd}, applied to the flat Poisson spray for the coordinate Poisson manifold $(M,\pi)$ (see \cite{CD}), $$ V|_{(x,p)} = \pi^{ij}(x)p_j \partial_{x^i}.$$
	Notice that this vector field is equivalent to the system of ODEs for $t \mapsto (x(t),p(t))$ defined by eq. \eqref{eq:Ppeq} together with $\dot p_j =0$. In particular, as recalled from \cite{CMS1} in Example \ref{ex:spraygd}, we know there is an (exponential) isomorphism $exp_V:G_{V} \to G_\pi$ into the canonical $G_\pi$. It follows that, close to the units in $G_\pi$, the inversion is fully characterized by
	\begin{equation}\label{eq:invcheck}  \i \circ exp_V =_M  exp_V \circ \i_V ,\end{equation}
	where $\i_V$ is the inverse in the spray groupoid $G_V$, whose formula was recalled in Example \ref{ex:spraygd}. Let us now get a formula for $exp_V$. Since $exp_V$ is a local symplectic groupoid isomorphism, it is an isomorphism of the underlying strict symplectic realizations, and its germ thus uniquely characterized  by (see \cite{CDW}) sending units to units, defining a symplectomorphism near the units and by relating the source maps:
	$$ \a_\pi (exp_V (x,p)) = q(x,p) =x , \text{ for all $p$ small enough}.$$
	On the other hand, from the definition of $\a_\pi$ in eq. \eqref{eq:alphapi}, fixing $p$ small, we see that
	$$ \gamma^p(x): = \int_0^1 \ \varphi^u_{\pi,p}(x) \ du $$
	defines a local inverse to $x \mapsto \a_\pi(x,p)$. The map
	$$ \psi_V:(x,p) \mapsto (\gamma^p(x),p) \text{ for $(x,p)$ near $0^{T^*}$,} $$
	thus relates the source maps, $\a_{\pi}\circ \psi_V = q$, is the identity on units $\zero$ and it is straightforward to check that it also relates the symplectic structures, $\psi_V^*\omega_c = \omega_V$ (see \cite{CD}). Then, the germ of $\psi_V$ near $\zero$ must coincide with that of $exp_V$.
	
	Finally, we verify eq. \eqref{eq:invcheck} with $\i(x,p)=(x,-p)$ as in the statement of the Lemma. Considering $a=(x,p)$ with $p$ close enough to zero, we compute
	\begin{eqnarray*} exp_V \circ \i_V(x,p) &=& exp_V( - \phi^V_{u=1}(a)) \\ 
		&=& exp_V(\varphi^{1}_{\pi,p}(x), -p) \\
		&=& \left( \int_0^1 \ \varphi^u_{\pi,-p}(\varphi^{1}_{\pi,p}(x)) \ du, -p\right) \\
		&=& \left( \int_0^1 \ \varphi^{1-u}_{\pi,p}(x) \ du, -p\right) \\
		&\overset{s=1-u}{=}& \left( \int_0^1 \ \varphi^{s}_{\pi,p}(x) \ ds, -p\right) \\
		&=& (\gamma^P(x),-p) = \i\circ exp_V(x,p).
	\end{eqnarray*}
The Lemma is thus proven.
\end{proof}

As a consequence, the target map $\b_\pi$ for $G_\pi$ is given by
$$\b_\pi(x,p)=\a_\pi(x,-p).$$

\begin{example}\label{ex:Gpiconst}
	For $\pi=0$, we have that $G_{\pi=0}=(T^*M,+)$ the groupoid given by addition of covectors (Example \ref{ex:Gzeropois}). When $\pi^{ij}(x)=\pi^{ij}$ is constant in $x$, then $\a_\pi$ was given in Example \ref{ex:apisimpleexs} and the other structure maps of $G_\pi$ are then given by
$$ \b_\pi(x,p)= x - \frac{1}{2} \pi^{\sharp} p, m_\pi((x_1,p_1),(x_2,p_2))= (x_1+\frac{1}{2}\pi^\sharp(p_2), p_1+p_2)$$
where we denoted $\pi^\sharp(p)^i = \pi^{ji}p_j$. Observe that, in this case, $G_\pi$ is isomorphic to the action groupoid associated to the abelian group $(M^*,+)$ acting on $M$ through $\pi^\sharp$.
\end{example}

\begin{example}\label{ex:Gpilin}
 For a Lie algebra $\gg$, we consider the associated local Lie group $H_\gg$ of Example \ref{ex:Hbch} and the linear Poisson structure $\tilde \pi$ on $M=\gg^*$ defined in Example \ref{ex:cotangentH}, so that $T^*H_\gg$ integrates $(\gg^*,\tilde \pi)$. Using the computation of Example \ref{ex:apilin}, we get that $\Fou(x,p)=(p,x)$ induces an isomorphism
 $$ \Fou:G_{\pi} \simeq \overline{T^*H_\gg}, \text{ where } \pi = -\tilde \pi $$
 where the overline denotes the same local Lie groupoid but endowed with the opposite symplectic structure (see Remark \ref{rmk:Gopom}).  By Remark \ref{rmk:SforGop} and Example \ref{ex:Gpilin}, it follows that the $S$ given by the BCH series (see Example \ref{ex:Sbch}) is also a coordinate generating function for $G_\pi$, with $\pi = -\tilde \pi$ the underlying linear Poisson structure.
\end{example}

\begin{remark}\label{rmk:rightinvarODE} (Right-invariant ODEs) The solution of the ODE for a curve $t\mapsto g(t)\in \a_\pi^{-1}(y)$ given by
	$$ -\omega_c^\flat(TR_{g^{-1}}\frac{d}{dt}g) = (\b_\pi(g),p) , \ g(0)=(y,0) \text{ is given by } g(t)=\ham_{t}^{-p\b_\pi}(y,0),$$
	where $\phi^H_t$ is the hamiltonian flow in $(T^*M,\omega_c)$ (see Section \ref{sec:lsg}). This follows from the flow of $p\b_\pi$ being right invariant for $m_\pi$, $\omega_c$ being multiplicative (eq. \eqref{eq:omegamult}), $\b_\pi$ being an anti-Poisson map (see Section \ref{sec:lsg}) and by the first property in \eqref{eq:rhamu} (recall $\b_\pi(x,p)=\a_\pi(x,-p)$). This fact will be used in Section \ref{sec:psm} below.
\end{remark}

\subsection{The canonical generating function and the natural smooth family}\label{subsec:cangenfun}

From the previous subsection, we know that the canonical local symplectic groupoid $G_\pi$ admits a coordinate generating function, $S$, whose germ is uniquely fixed by $S(0,0,x)=0$. In this subsection, we observe that the special properties of the realization map $\a_\pi$ allow us to refine the general formulas obtained in \ref{subsub:formulasS} for $S$. Moreover, we observe that the formulas alone can be shown to define a function independently of any existence result.

\medskip

We thus start by writing down the formulas that define the function entering our main theorem. Recalling the notation \eqref{eq:hamflow} for hamiltonian flows in $P=(T^*M,\omega_c)$, given any map $\alpha: M \times M^* \ _{M\times 0}\dto M$ we consider the system
\begin{eqnarray}\label{eq:DarbS} 
\beta(x,p)&=&\a(x,-p) \nonumber \\
x &=& q\phi^{\b^*p_1 + \a^*p_2}_{u=1}(x_0,0) \nonumber \\
S(p_1,p_2,x) &=& (p_1+p_2)(x_0) - \int_0^1 du \  L_E (\b^*p_1+\a^*p_2)( \phi^{\b^*p_1 + \a^*p_2}_{u} (x_0,0) )
\end{eqnarray}
where $p_1,p_2 \in M^*$ are seen as linear functions on $M$, $x\in M$ and we recall the Euler vector field $E=p_j \partial_{p_j}$.

\begin{lemma}
	For any smooth map $\a$ as above, the system \eqref{eq:DarbS} defines a smooth function $$ S: M^* \times M^* \times M  \ _{(0 \times 0 \times M)}\dto \R$$ which satisfies $S(0,0,x)=0, \forall x$.
\end{lemma}
\begin{proof}
	Since for $p_1=0=p_2$ the flow is the identity, by the implicit function theorem we can solve for $x_0\equiv x_0(p_1,p_2,x)$ near $0\times 0 \times M$ in the second equation above. This leaves the third eq. as an ordinary definition of a function $S$. From the last equation in the system, it is clear that $S(0,0,x) = 0$.
\end{proof}

We summarize the results obtained thus far in the following theorem, which is the main result of this subsection.

\begin{theorem}\label{thm:main1}
	Let $\pi$ be a Poisson structure on a coordinate space $M$ and let $\alpha_\pi$ be the induced realization map defined by \eqref{eq:alphapi}. Then, the function $S\equiv S_\pi$ defined by the system \eqref{eq:DarbS} with $\a=\a_\pi$ is a coordinate generating function for the canonical local symplectic groupoid $G_\pi=\G(T^*M,\omega_c,0^{T^*M},\alpha_\pi)$. Moreover, any other coordinate generating function for $G_\pi$ satisfying $S(0,0,x)=0\forall x$ must have the same germ as $S_\pi$ near $p_1=p_2=0$.
\end{theorem}

\begin{proof}
	By Corollary \ref{prop:existS}, we know that $G_\pi$ admits a coordinate generating function $S$ (with embedding $\nu_c$) and that the germ of it is uniquely determined by $S(0,0,x)=0\forall x$. Then, we only need to show that this function $S$ solves the system \eqref{eq:DarbS}, with $\a=\a_\pi$, near $p_1,p_2=0$. To that end, fixing $x_0\in M$ and for $p_1,p_2$ close enough to zero, we consider the formula \eqref{eq:genScurve} evaluated at the curve $u \mapsto \gamma(u)\in gr(m)$ defined in \eqref{eq:curve2} out of $x=x_0 \in M$ and $f_1=p_1$, $f_2=p_2$. It is straightforward to arrive from this expression to the system \eqref{eq:DarbS} using properties \eqref{eq:rhamu} and the definition of the involved hamiltonian flows, thus finishing the proof.
\end{proof}

The function $S_\pi$ defined in Theorem \ref{thm:main1} will be called the {\bf canonical generating function} integrating $(M,\pi)$. We also state that $S_\pi$ satisfies the general formulas obtained in \ref{subsub:formulasS}, as well as the SGA equation.
\begin{corollary}
	The germ of the canonical generating function $S_\pi$ satisfies \eqref{eq:S2parint}, \eqref{eq:Sp0x} and \eqref{eq:Sint1dgen}. Also, the function $S_\pi$ is a solution to the SGA equation \eqref{eq:SGA}. 
\end{corollary}
Notice that, in particular, for any $\pi$ on the coordinate space $M$, we can always find a smooth solution of the SGA equation integrating $(M,\pi)$ (recall eq. \eqref{eq:pifromS}).

\begin{remark}\label{rmk:simplifSpi}
	Because of the special properties \eqref{eq:rhamu} of $\a_\pi$, the formulas cited in the above Corollary have the following simplifications:
	$$ \px (\phi^{\a^*p_0}_1(x_0,0))= p_0(x_0)= \px (\phi^{\b^*p_0}_1(x_0,0)), \ L_E(\a^*p_0)(\phi_u^{\a^*p_0}(x_0,0)) = 0 .$$
\end{remark}

\begin{example}\label{ex:Spiconstandlin}
	For $\pi$ being constant (see Example \ref{ex:Gpiconst}) or $\pi$ being linear (see Example \ref{ex:Gpilin}) the germs of the associated canonical generating functions $S_\pi$ coincide, by the uniqueness property in Corollary \ref{prop:existS}, with the ones given in Examples \ref{ex:S0} and \ref{ex:Sbch}, respectively.
\end{example}

\subsubsection*{Induced smooth $1$-parameter families}\label{subsub:smooth1par}

We end this subsection showing how the above constructions extend to smooth 1-parameter families. Before that, we introduce the nomenclature to be used. By a \emph{smooth family of functions} $\hat f\equiv (f_t: X \dtod{Z} Y)_{t\in [0,1]}$ we mean a smooth map defined as
$$ \hat f: X \times [0,1] \dtod{Z \times [0,1]} Y. $$
By a \emph{smooth family of local groupoids} $\hat G\equiv (G_t)_{t\in [0,1]}$ we mean a local groupoid $\hat G$ with units space $ M \times [0,1]$ such that the natural projection onto $[0,1]$, endowed with the unit groupoid structure (recall Example \ref{ex:unitg}), defines a local groupoid morphism and all the structure maps are smooth families. We say that $\hat G$ is \emph{symplectic} when it comes endowed with a multiplicative $2$-form $\hat \omega$ whose restriction $\omega_{G_t}$ to each $t$-fiber is symplectic. We will focus on the case in which the family of symplectic structures $\omega_{G_t}$ is constant in $t$ and in which the units are given by a constant map $1:M \hookrightarrow P$.

Observe that a smooth family of local symplectic groupoids $\hat G$ induces a smooth family of Poisson structures $\hat \pi \equiv (\pi_t)_{t\in [0,1]}$ on $M$ so that the source map of $G_t$ is a Poisson map for each $t$. In this case, we say that \emph{the smooth family $\hat G$ integrates the smooth family $(M,\hat \pi)$}.

\begin{corollary}\label{cor:smfam}
	The constructions $\pi \mapsto G_\pi$ and $\pi \mapsto S_\pi$ defined above, when applied to $t\pi$ for a given $\pi$ and $t\in [0,1]$, define smooth $1$-parameter families
	$$ \hat G_\pi: = (G_{t\pi})_{t\in [0,1]}, \ \ \hat S_{\pi}:=(S_{t\pi})_{t\in [0,1]}$$
	 of local symplectic groupoids integrating $(M,t\pi)_{t\in [0,1]}$ and of corresponding coordinate generating functions, respectively. Moreover, the following identities hold:
	\begin{eqnarray} S_{t\pi} (\lambda p_1, \lambda p_2, x) = \lambda S_{\lambda t \pi}(p_1,p_2,x) &,& \  S_0(p_1,p_2,x) =  (p_1 + p_2)(x) \nonumber \\
	  S_{t\pi}(p_0,p_0,x) &=& 2  p_0 (x).  \label{eq:Spinaturprop}
	 \end{eqnarray}
\end{corollary}

\begin{proof}
	From the first property in \eqref{eq:propsalphaini} and the definition of hamiltonian flows \eqref{eq:hamflow}, it follows that
	$$ \phi^{\lambda \b_\pi^*p_1 + \lambda \a_\pi^*p_2}_{u} \circ \homp_\lambda = \homp_\lambda \circ  \phi^{\b_{\lambda \pi}^*p_1 + \a_{\lambda\pi}^*p_2}_{u}$$
	recalling the rescaling action $\homp_\lambda(x,p)=(x,\lambda p)$. Also, by direct computation using \eqref{eq:propsalphaini},
	$$ [L_E( \b_\pi^*p_1 +  \a_\pi^*p_2)]\circ \homp_\lambda = L_E(\b_{\lambda\pi}^*p_1 + \a_{\lambda\pi}^*p_2). $$
	The first identity in \eqref{eq:Spinaturprop} thus follows. The second identity follows from noticing that for $\pi=0$, the definition of $\a_\pi$ in \eqref{eq:alphapi} reduces to
	$$ \a_0 (x,p) = x = \b_0(x,p) \text{ independent of $p$},$$
	so that $L_E(\b^*p_1+\a^*p_2) =0$.
	Finally, the last identity follows from
	$$ r \phi^{\b^*p_0 + \a^*p_0}_{u} (x_0,0) = 2up_0, \ \   L_E (\b^*p_0+\a^*p_0)( \phi^{\b^*p_0 + \a^*p_0}_{u} (x_0,0)) = 0.$$
	These, in turn, can be proven exactly following the proofs of \eqref{eq:rhamu} and recalling that $\a_\pi(x,p)=\b_\pi(x,-p)$ for $G_\pi$.
\end{proof}

Following the nomenclature used in \cite{CDF} for formal families (see also Section \ref{subsec:fromS}), we introduce the following:
\begin{definition}\label{def:Snatural}
 We say that a smooth family $\hat S \equiv (S_t)_{t\in [0,1]}$ of generating functions satisfying \eqref{eq:Sp0x} and identities \eqref{eq:Spinaturprop} defines a \emph{natural family of generating functions}.
\end{definition}
Notice that, for natural families of generating functions, $S_t$ is homogeneous of degree $1$ with respect to the rescaling action $(p_1,p_2,x,t) \mapsto (\lambda p_1, \lambda p_2, x, \lambda^{-1} t) $ on $X \times \R$.

The families $\hat G_\pi$ and $\hat S_\pi$ defined in Corollary \ref{cor:smfam} will be referred to as \emph{the natural smooth family of local symplectic groupoids, and of generating functions, respectively}, both integrating $(M,t \pi)_{t\in [0,1]}$. Notice that, in this family $S_{t\pi}$, the infinitesimal direction of deformation of $S_0$ is given by $\pi$.

\begin{example}\label{ex:natStlin}
		Consider $M=\gg^*$ endowed with the linear Poisson structure $\pi=-\tilde \pi$ defined by a Lie algebra structure on $\gg$, as in Example \ref{ex:Gpilin}. Using properties \eqref{eq:Spinaturprop} and Example \ref{ex:Sbch}, the natural smooth natural family $\hat G_\pi$ integrating $(M,t\pi)$ has
	$$t \mapsto  S_t(p_1,p_2,x) = \frac{1}{t}BCH(tp_1,t p_2)( x)$$
	as the natural smooth family of coordinate generating functions.
\end{example}

\begin{remark}\label{rmk:formulatseppi}
	For general $\pi$, as a consequence of eq. \eqref{eq:Sint1dgen}, and the simplifications of Remark \ref{rmk:simplifSpi} coming from the special properties of $\a_\pi$, we obtain that
	$$ S_{t\pi}(p_1,p_2,q\phi_t^{p_2\a_{\pi}}(x_1,tp_1)) = p_1(x_1) + p_2(\a_{t\pi}(x_1,p_1)) - \frac{1}{t}\int_0^t du \ [L_E(p_2\a_\pi)](\phi^{p_2\a_\pi}_u(x_1,tp_1))$$
\end{remark}

\begin{remark}\label{rmk:2cocynatffam}
	We now obtain a formula relating the canonical $2$-cocycle (see before Lemma \ref{lem:CandS}) and a natural family of generating functions. Let $(S_t)_{t\in [0,1]}$ be a natural family of generating functions, in the sense of Definition \ref{def:Snatural}, with underlying family of local symplectic groupoids $(G_t)_{t\in [0,1]}$ having constant symplectic structure $\omega_{G_t} = \omega_c\ \forall t$. On the one hand, we have the associated smooth family of canonical $2$-cocycles $C_t\equiv C_{G_t}$ defined by eq. \eqref{eq:2cocycle}. On the other, we have a smooth $t$-family of maps $J_t:X \dtod{X_0} G_t^{(2)}$ induced by $S_t$ via eq. \eqref{eq:Jcocy}. 
	It follows directly from the first of the naturality identities \eqref{eq:Spinaturprop} that 
	$$ L_\Eu S_t = S_t + t \partial_t S_t,$$
	so that, from Lemma \ref{lem:CandS}, we obtain the formulas
	\[J_t^* C_t = t \partial_t S_t ; \  \ S_t = (p_1+p_2)(x) + \int_0^t \frac{1}{s} J_s^* C_s \ ds.\]
\end{remark}

	

	




\section{The induced formal families and Kontsevich's tree-level formula}\label{sec:formal}

In this Section, we consider formal families of local symplectic groupoids and make the connection to the tree-level part of Kontsevich's quantization formula. To explain the context of these results, we go back to the key relation\footnote{For a formal star product $\star_\h$, the right hand side of \eqref{eq:starSP} can be understood as a formal power series on $\h,\h^{-1},p_1$ and $p_2$.} \eqref{eq:starSP} and consider that the left hand side is given by Kontsevich's star product \cite{Kontquant} associated to a coordinate Poisson manifold $(M,\pi)$. This $\star_\h$ is a formal power series in $\h$ involving Kontsevich graphs (see Appendix \ref{asub:Ktrees}) and the semiclassical contribution $S_P\equiv \bar{S}^K$ was singled out and studied by Cattaneo-Dherin-Felder in \cite{CDF} (see Example \ref{ex:bSK} below). From general considerations (see \cite[\S 7]{CDF}) it is expected that $S_P$ satisfies the SGA equation as a consequence of $\star_\h$ being associative. In \cite{CDF}, $\bar{S}^K$ is indeed shown to define a formal family of solutions of the SGA equation along $\e\pi$ (with $\e$ a formal parameter) as a consequence of underlying relations for the defining Kontsevich weights. The main result of this section, Theorem \ref{thm:main2SK} below, states that the formal Taylor expansion of the natural smooth family $\hat S_\pi$ (see subsection \ref{subsub:smooth1par}) around $t=0$ yields the Kontsevich-tree formal family $\bar{S}^K$ and, thus, that the SGA equation for $\bar{S}^K$ can be seen to hold as a consequence of non-formal Lie-theoretic constructions for $\hat{S}_\pi$. In subsection \ref{subsec:graphexp}, we also discuss how to derive the structure of the Kontsevich-graph expansion for $\bar{S}^K$ in terms of elementary methods involving Butcher series \cite{Bu} (see Appendix \ref{asub:RT}) applied to the analytic formulas of Section \ref{sec:cangenf}. This extends analogous results of \cite{CD} for the source map $\a_\pi$.

\subsection{Formal families of maps and groupoid structures}

Our main objects of study in this section will involve formal families of local symplectic groupoids. To specify what we mean by those, we first review some standard notation and facts about series on a formal parameter.

First, we denote formal power series on the formal parameter $\e$ with real coefficients by $\R\bep$.
We will think of $\R\bep$ as infinite jets of functions on a formal neighborhood of $t=0$ on an interval $I=[0,\delta)$ so that, thinking of the inductive limit as $\delta \to 0$, and use the notation $ \R\bep = C^\infty(\Iform).$
This idea can be formalized, but we will only be using it for notational clarity.
For any manifold $N$, we define
$$ C^\infty(N)\bep=C^\infty(N)\otimes \R\bep \in \bar f_\e = \sum_{n\geq 0} \e^n \ f_n, \ \text{with $f_n \in C^\infty(N)$.}$$
We can then think of $C^\infty(N)\bep$ as $C^\infty(N \times \Iform)$ and thus as defining \emph{formal families of functions on $N$}. Given a smooth family $\hat f\equiv (f_t)_{t\in [0,\delta]}$ in the sense of Section \ref{subsec:cangenfun}, there is an induced formal $\e$-family defined by the Taylor series of $f_t$ around $t=0$, denoted
$$ \Tay_\e(\hat f) = \sum_{n\geq 0} \frac{\e^n}{n!} \ \partial^n_t f_t|_{t=0}.$$
(Every formal family is the formal Taylor series of some smooth family, wildly non-unique.)
A \emph{formal family of maps} $\bfu$, $$\bfu: N_1 \times \Iform \to N_2,\text{ is defined to be an algebra morphism }
 \bfu^*: C^\infty(N_2) \to C^\infty(N_1)\bep.$$
Ordinary maps $f:N_1 \to N_2$ can be seen as constant formal families defined by extending $f^*$ linearly over $\R\bep$. (Notice that an $\bfu_\e \in C^\infty(N)\bep$ can be seen as a formal family of maps $\bfu:N\times \Iform \to \R$, where $\bfu_\e \equiv \bfu^*(\e)$.) 
Composition of such formal families of maps is defined in the obvious way as $ (\bfu\circ \bar g)^* = \bar g^* \circ \bar f^*,$
where $\bar g^*$ is extended $\R\bep$-linearly (geometrically, the extension corresponds to the pullback along the projection $N\times \Iform \to N$). The formal Taylor expansion procedure can be extended to smooth families of maps 
$$\hat f: N_1 \times [0,1] \to N_2 \text{ yielding } \Tay(\hat f): N_1 \times \Iform \to N_2$$
and satisfying 
$$ \Tay(\hat g\circ \hat f) = \Tay(\hat g)\circ \Tay(\hat f)$$
as a consequence of the chain rule.
The formal pullback $\bfu^*$ can be naturally extended to formal families of forms $\bfu^*:\Omega^k(N_2)\bep \to \Omega^k(N_1)\bep$ and, also, one can naturally define formal families of tensors on $N$. In particular, we consider formal families of Poisson structures to be
$$ \bpi \in \mathcal{X}^2(M)\bep, \ \bpi_\e = \sum_{n\geq 0} \e^n \ \pi_n, \ \pi_n \in \mathcal{X}^2(N) ,$$
which satisfy the Jacobi identity $[\bpi_\e,\bpi_\e]=0$ to all orders in $\e$ (here $[,]$ is extended bilinearly over $\R\bep$).

\begin{definition}
	A \emph{formal ($1$-parameter) family of local symplectic groupoids} (with constant symplectic structure) is the data $\bar G_\e=(P,\omega,M, \ba,\bb,\bar{\i},\bar{m})$ where the structure maps are formal families of maps satisfying the obvious adaptation of the local groupoid axioms to the composition rules given above, and 
	$$ \bar m_\e^*\omega - pr_1^*\omega - pr_2^*\omega =_{M^{(2)}} 0.$$
\end{definition}
In the above definition we focused on formal families with constant spaces and symplectic structure not to deviate from our main objects of study, but the definition can be obviously generalized.

Notice that given a smooth $1$-parameter family of local symplectic groupoids $\hat G\equiv (G_t)_{t\in [0,1]}$, as defined in Section \ref{subsec:cangenfun}, with constant symplectic structure $\omega$, there is an induced formal family of local symplectic groupoids $\bar G$ defined by $(P,\omega,M)$ together with the formal Taylor series around $t=0$ of the sturcture maps of $G_t$. We denote this induced formal family as
$$ \bG = \Tay(\hat G).$$
For any formal family $\bar G$, there is a unique formal family of Poisson structures $\bpi$ on $M$ such that the source $\ba$ defines a Poisson map, i.e.
$$ \{\ba_\e^*f_1,\ba_\e^*f_2\}_\omega = \ba_\e^*\left(\bpi_\e(df_1,df_2) \right), \ \forall f_1,f_2 \in C^\infty(M)\bep.$$
The proof of this fact is analogous to the case of ordinary local symplectic groupoids, when expressed in terms of pullbacks along maps. We say that $\bG$ \emph{integrates the formal family} $(M,\bpi)$.
Given an ordinary Poisson structure $\pi$ on $N$, the formal family defined by
$$ \bpi_\e = \e \pi$$
will be called the \emph{natural formal family of Poisson structures induced by $\pi$} and will be our main focus. We now observe that this family has an integration to a formal family of local symplectic groupoids as a simple consequence of the formal Taylor expansion procedure applied to the our earlier results on smooth families.

\begin{proposition}\label{lma:tayinttay}
	If $\hat G = (G_t)_{t\in [0,1]}$ is a smooth family of local symplectic groupoids integrating $(M,\pi_t)_{t\in [0,1]}$, then $\Tay(\hat G)$ is a formal family of local symplectic groupoids integrating $(M, \Tay_\e(\pi_t))$. In particular, when $\hat G = \hat G_{\pi}$ is the natural smooth family associated to $\pi$ (see Section \ref{subsec:cangenfun}), then $$\bG_\pi := \Tay(\hat G_{\pi})$$ defines a formal family of local symplectic groupoids integrating $\e \pi$.
\end{proposition}
We will refer to $\bG_\pi$ as the \emph{natural formal family of local symplectic groupoid (germs)} integrating $(M,\e \pi)$.
%
%

\bigskip


\begin{example}
	Let $M$ be a coordinate space and $\pi$ a constant Poisson structure on $M$. We want to describe a formal family of symplectic groupoids $\bG_\e$ integrating the natural formal family $(M,\e \pi)$. Following Proposition \ref{lma:tayinttay}, we conclude that $\bG_\e=\Tay_\e(G_{t\pi})$ provides such an integration, where $G_{t\pi}$ is the natural family of canonical symplectic groupoids integrating $t\pi, \ t\in [0,1]$. The structure maps on $\bG_\e$ are just the same as the ones given in Example \ref{ex:Gpiconst} with $\pi$ replaced by $\e \pi$, since it only enters linearly.
\end{example}




\begin{remark}\textsc{(Formal groupoids vs. formal families of groupoids)} In \cite{Kar}, Karabegov introduces "formal symplectic groupoids" whose structure maps are formal Taylor expansions in the variables $p$ representing the normal directions to $\zero: M \hookrightarrow T^*M$. These are, a priori, different from the above formal $1$-dimensional families in which only the family parameter $\e$ is formal. Nevertheless, observe that, when special rescaling properties hold, for example $\a_{t\pi}(x,p)=\a_\pi(x,tp)$ (Lemma \ref{lma:propapi}) for the canonical Karasev source / realization map defined by \eqref{eq:alphapi}, then the two expansions in $t$ or $p$ can be related and yield essentially the same result.
\end{remark}

\subsection{Taylor expansion of the natural family and the Kontsevich-trees formal family}\label{subsec:fromS}
In this subsection, we restrict ourselves to the case in which $M$ is a coordinate space, as in Section \ref{sec:cangenf}.
We also restrict ourselves to formal families of local symplectic groupoids of the form 
$$\bar G=(T^*M,\omega_c,0^{T^*M}, \ba,\bb,\bar{\i},\bar{m})$$ integrating a formal family of underlying coordinate Poisson manifolds $(M,\bpi_\e)$. We say that 
$$ \bS: M^* \times M^* \times M \times \Iform \to \R$$
is a {\bf formal family of coordinate generating functions for $\bar G$} when
$$ \bm_\e^* p_1 = p_1, \bm_\e^* p_2 = p_2, \bm_\e^*x_3 = x, $$ 
$$ \bm_\e^* x_1 = \partial_{p_1}\bS_\e, \ \bm_\e^* x_2 = \partial_{p_2}\bS_\e, \ \bm_\e^*x_3 = \partial_x \bS_\e,$$ 
where, on the left hand sides, we are thinking of the natural coordinate functions for $((x_1,p_1),(x_2,p_2),(x_3,p_3)) \in T^*M \times T^*M \times T^*M = \P_{T^*M}$ and, on the right hand sides, we are using the identification $\nu_c:T^*X \simeq \P_{T^*M}$ of eq. \eqref{eq:nucoord} and think of functions of the natural coordinates for $(p_1,p_2,x)\in X = M^* \times M^* \times M$. This provides the direct formal-family analogue of the coordinate generating functions considered in Section \ref{subsec:generalgenfuncs}.

\begin{example}\label{ex:formalS0andct}
	The groupoid $G_0$ integrating $\pi=0$ can be naturally extended to a constant formal family $\bG_0$. The function
	$$ \bS_0(p_1,p_2,x) = (p_1+p_2)(x),$$
	thought of as a constant formal family, defines a formal family of coordinate generating functions for $\bG_0$. More generally, when $\pi$ is a constant Poisson structure on $M$ and $\bG=\bG_\pi$ is given by the natural family, then
	$$ \bS_\e(p_1,p_2,x) = (p_1+p_2)(x) + \frac{\e}{2} \pi(p_1,p_2)$$
	defines a formal coordinate generating function for $\bG_\pi$. 
\end{example}

\begin{example}\label{ex:bSK}\textsc{(The Kontsevich-trees generating function.)}
	In this example we summarize the definition of a formal family of generating functions, $\bS^K_\pi$, associated to any coordinate $(M,\pi)$ as introduced in \cite{CDF}. The formula for $\bS^K_\pi$ can be seen as an extract from Kontsevich's star product formula \cite{Kontquant} for quantizing $(M,\pi)$ in which only certain tree-type graphs, called \emph{Kontsevich trees} $\Gamma \in T_{n,2}$, are considered. Specifically\footnote{The formal generating function depends on the Poisson structure	through the symbols of the Kontsevich operators. In \cite{CDF}, the scaling of the Poisson structure is different from the one used here: with the present notation, they work with $\bS^K_{2\pi}$ instead of $\bS^K_\pi$, resulting in the fact that their formal family integrates the Poisson structure $2\epsilon\pi$	instead of our $\epsilon\pi$; see also \cite[\S 4.1]{CD}.},
	\begin{equation}\label{eq:defSK}
	\bS^K_\pi (p_1,p_2,x):= (p_{1}+p_{2})x+\sum_{n=1}^{\infty}\frac{\epsilon^{n}}{n!2^n}\sum_{\Gamma\in T_{n,2}}W_{\Gamma}\hat{B}_{\Gamma}\left(\pi\right)(p_{1},p_{2},x),
	\end{equation}
	where the Kontsevich symbols $\hat{B}_\Gamma(\pi)$ associated to each graph $\Gamma$ are recalled in Appendix \ref{asub:Ktrees} and the weights $W_\Gamma$ were defined in \cite{Kontquant} (see \cite{CD} for the conventions used here). We refer to $\bS^K_\pi$ as to the {\bf Kontsevich-trees formal family}.
	This formal family generates a formal family of symplectic groupoids, that we shall denote $\bG_\pi^K$, which is described in detail in \cite{CDF} and has $P=T^*M$, $\omega=\omega_c$ and $1_x = (x,0)$. 
\end{example}

\begin{remark}\label{rmk:bGKandbGpi}
	In \cite{CD}, it was shown that the source map of $\bG^K_\pi$ coincides with $\Tay((\a_{t\pi})_{t\in [0,1]})$, the formal Taylor expansion along the natural family of source maps (defined by \eqref{eq:alphapi} using $t\pi$). 	A formal-family adaptation of the result stating that the underlying strict symplectic realization data fully characterizes the germ of the local symplectic structure around the units implies that the germ of $\bG^K_\pi$ and that of the canonical formal family $\bG_\pi$ coincide. We shall also provide an alternative argument for this fact below.
\end{remark}

As happened with smooth families of local symplectic groupoids in the previous subsection, we obtain the following a direct result of taking Taylor expansions of known smooth families.
\begin{proposition}\label{prop:tayS}
	Let $\hat G$ be a smooth family of local symplectic groupoids having $\hat S$ as a smooth family of coordinate generating functions. Then, $\Tay (\hat S)$ defines a formal family of coordinate generating functions for $\bG=\Tay(\hat G)$. In particular, $$ \bS_\pi := \Tay(\hat S_\pi),$$
	where $\hat S_\pi$ denotes the natural family associated to $(M,\pi)$ (as defined in Section \ref{subsec:cangenfun}), defines a formal family of generating functions for the natural formal family $\bG_\pi$.
\end{proposition}
We refer to $\bS_\pi$ defined above as to the \emph{natural formal family of generating functions integrating $(M,\e\pi)$}. 

\begin{example}
	Consider $M=\gg^*$ endowed with the linear Poisson structure $\pi$ as in Example \ref{ex:Gpilin}. Following Example \ref{ex:natStlin}, the natural formal family of generating functions integrating $(M,\e \pi)$ is given by the formula
	$$ \bS_\e(p_1,p_2,x) = \frac{1}{\e}BCH(\e p_1,\e p_2)( x).$$
	Notice that only non-negative powers of $\e$ appear above, by the homogeneity properties of the BCH series.
\end{example}

We now turn to consider the SGA equation for formal families. Recall from Section \ref{subsub:sga}, that a coordinate generating function satisfies the SGA equation and that the analogue of this equation for formal families $\bS$ was introduced in \cite{CDF}. From that paper we also recall the following notions: a formal family $\bS$ is called a \emph{natural deformation} of $S_0$ given Example \ref{ex:formalS0andct} if 
\begin{equation} \label{eq:formaldefS0} \bS_\e = \bS_0 + \sum_{n\geq 1}\e^n \ \bS_n ,\end{equation}
where $\bS_n\in C^\infty(X)$ satisfy:
\begin{enumerate}
	\item $\bS_n(p_1,p_2,x)$ is a polynomial on $p_1,p_2$,
	\item $\bS_n(\lambda p_1,\lambda p_2,x)= \lambda^{n+1} \bS_n(p_1,p_2,x)$
	\item $\bS_n(p,0,x)=0=\bS_n(0,p,x)$
	\item the homogeneous part of degree $k$, $\bS^{(k)}_n(p_1,p_2,x)$, of $\bS_n$ on the variable $p_1$ satisfies $\bS^{(k)}_n(p,p,x)=0$.
\end{enumerate}

As a corollary of Proposition \ref{prop:tayS}, we obtain:
\begin{corollary}\label{cor:bSpiSGA}
	The natural formal family of generating functions $\bS_\pi$ integrating $(M,\e\pi)$ is a solution to the SGA equation for formal families, as defined in \cite{CDF}. Moreover, $\bS_\pi$ defines a natural deformation of $\bS_0$ in the sense recalled above.
\end{corollary}
\begin{proof}
	The first statement is a direct consequence of $\bS_\pi$ being a generating function for $\bG_\pi$, namely of Proposition \ref{prop:tayS}, and standard properties of taking Taylor series. The second statement is a consequence of properties \eqref{eq:Spinaturprop} of the smooth family $\hat S_\pi$, as we now detail. To begin with, it is clear from the second eq. in \eqref{eq:Spinaturprop} that $\bS_\pi$ takes the form \eqref{eq:formaldefS0} above. Similarly, item (3) in the definition above follows from property \eqref{eq:Sp0x} of $S_\pi$. Item (2) is a direct consequence of the first (rescaling) property of the natural smooth family $\hat S_\pi$ in eq. \eqref{eq:Spinaturprop}, by using the chain rule along $t\mapsto \lambda t$. Item (1) follows from item (2) since an homogeneous smooth function, here each $\bS_n$, must be a polynomial. Finally, item (4) is a direct consequence of Taylor expanding the third property in eq. \eqref{eq:Spinaturprop}.
\end{proof}

Finally, we recall the following result from \cite{CDF} which characterizes the Kontsevich-trees formal family.
\begin{theorem}(\cite[Thm. 1]{CDF})
	For each Poisson structure $\pi$ on the coordinate space $M$, there exists a unique natural deformation of the trivial formal family $\bS_0$ of generating functions which satisfies the SGA equation for formal families and integrates $\e \pi$ (recall eq. \eqref{eq:pifromS}). Moreover, this formal family is given by the Kontsevich-trees formal family $\bS^K_\pi$  recalled in Example \ref{ex:bSK}.
\end{theorem}

\begin{remark}
	The uniqueness part of the cited Theorem above can be proven using the same line of reasoning as in Section \ref{subsec:generalexist} but where the intermediate results must be replaced by their direct formal-family analogoues. The author thinks all those intermediate results indeed hold.
\end{remark}

Combining the above cited Theorem with the results of this subsection, in particular with Corollary \ref{cor:bSpiSGA}, we obtain our second main result of the paper:

\begin{theorem}\label{thm:main2SK}
	Let $(M,\pi)$ be a coordinate Poisson manifold and denote $\hat S_\pi$ the induced natural smooth family of coordinate generating functions integrating $(M,t\pi)_{t\in [0,1]}$ (see Section \ref{subsec:cangenfun}). Then, the formal Taylor expansion of $\hat S_\pi$ around $t=0$ coincides with the Kontsevich-trees formal family (see Example \ref{ex:bSK}): 
	$$ \bS^K_\pi = \Tay(\hat S_\pi).$$ 
\end{theorem}
From this result it follows that the formal family of local symplectic groupoids $\bG^K_\pi$ agrees with the canonical formal family $\bG_\pi$ (as previewed in Remark \ref{rmk:bGKandbGpi}). Notice that from the rescaling properties (eqs. \eqref{eq:Spinaturprop}) and the fact that $\e$ is formal, one can think of $\bG_\pi$ (and hence of $\bG^K_\pi$) as formal families of (global) Lie groupoids.
Finally, we observe that the fact that the Kontsevich-trees formal family $\bS^K_\pi$ is the Taylor expansion of a smooth family which also satisfies the (non-linear) SGA equation for $t\in [0,1]$ is non-trivial.

\subsection{The graph expansion of the canonical formal family: from Butcher series to Kontsevich trees}\label{subsec:graphexp}

In this subsection, we will show how a Kontsevich-tree graph expansion emerges structurally from the properties of the canonical smooth family $\hat S_\pi$ by applying standard Butcher series techniques \cite{Bu} (see also \cite{CD} and Appendix \ref{asub:RT}).
To that end, we focus on the formula for $S_{t\pi}$ given in Remark \ref{rmk:formulatseppi} and compute the structure of its Taylor expansion at $t=0$. 

To emphasize the relevant underlying properties, independently of $S_{t\pi}$ being a solution of the SGA equation, we consider an arbitrary map $\a:T^*M \dtod{0} M$ satisfying $\a(x,0)=x$ and a $t$-family of functions $S_t: X \dtod{X_0} \R$ defined by the formula
\begin{eqnarray}\label{eq:Stgena}
S_t(p_1,p_2,x(t)) = p_1(x_1) + p_2\a(x_1,tp_1) - \frac{1}{t}\int_0^t du \ [L_E(p_2\a)](\phi_u^{p_2\a}(x_1,tp_1))\nonumber \\ 
x(t) = q\phi^{p_2\a}_t(x_1,tp_1) \label{eq:Saforgraphs}
\end{eqnarray}
As mentioned above, the definition is chosen so that, when $\a = \a_\pi$ as defined in eq. \eqref{eq:alphapi}, then $S_t = S_{t\pi}$ yields the natural family of generating functions.

We first proceed to understand the more superficial part of $\bS_\e:=\Tay_\e(S_t)_t$, writing
\begin{eqnarray}
\bS_\e(p_1,p_2,x) &=& \sum_{n\geq 0} \e^n \ S_n(p_1,p_2,x) \nonumber \\
\Tay_\e(x(t)_t) &=& x_1 +\e \sum_{n\geq 0} \e^n \ \delta_n(p_1,p_2,x_1) \nonumber \\
\Tay_\e(\a(x,tp))_t &=& x + \sum_{m\geq 1} \e^m \ \a_m(x,p)  \nonumber \\
\Tay_\e\left( \frac{1}{t}\int_0^t du \ [L_E(p_2\a)](\phi_u^{p_2\a}(x_1,tp_1)) \right)_t &=& \sum_{m\geq 1} \e^m \ I_m(p_1,p_2,x_1)\nonumber
\end{eqnarray}
where we have used the hypothesis $\a(x,0)=x$ in the last two lines. From these we conclude, by Taylor expansion at $t=0$ of the defining equation \eqref{eq:Saforgraphs}, that $S_n$ can be recursively computed by
\begin{eqnarray}
S_0(p_1,p_2,x_1)& =& (p_1+p_2)(x_1) \nonumber \\
(m\geq 1) \ S_m(p_1,p_2,x_1) &=& -(p_1+p_2)\delta_{m-1}(p_1,p_2,x_1) \nonumber \\
&& - \sum_{k=1}^{m-1}\sum_{s=0}^{m-k-1}\sum_{l_1+\dots+l_k=s, l_j\geq 0} \frac{1}{k!} D^k_xS_{m-k-s}(\delta_{l_1},\dots,\delta_{l_k})|_{(p_1,p_2,x_1)}  \nonumber \\
&+& p_2 \a_m(x_1,p_1) - I_m(p_1,p_2,x_1) \label{eq:Smrecurs}
\end{eqnarray}
where $D^k_xS|_{(p_1,p_2,x)}$ denotes the k-th derivative of $S$ w.r.t. $x$, seen as a $k$-multilinear map $M^k \to \R$. 

\begin{remark}\label{rmk:S1}
	We can compute the first term $S_1$, yielding 
	$$ S_1(p_1,p_2,x) = -(p_1+p_2)(\delta_0(p_1,p_2,x)) + p_2(\a_1(x,p_1)) - I_1(p_1,p_2,x).$$
\end{remark}
Next, we begin with our use of Butcher series \cite{Bu} which we recall in Appendix \ref{asub:RT} (see also \cite{CD} for their use in the context of symplectic realizations). There are two main motivations for the appearance of these series: expanding the flow $\phi_u^{p_1\a}$ in terms of elementary differentials of the underlying hamiltonian vector field $X\equiv X^{p_1\a}$, and the result of \cite{CD} which expresses $\Tay_\e(\a_{t\pi})_t$ as a Butcher series (see eq. \eqref{eq:tayapi} below). 
We begin by considering the first motivation, and recall that
\begin{equation}\label{eq:Xpa} X|_{(x,p)}= -\partial_{p_j}(p_1\a)|_{(x,p)} \ \partial_{x^j} + \partial_{x^j}(p_1\a)|_{(x,p)} \ \partial_{p_j}.\end{equation}
From the standard theory of Butcher series, we get the following expansions in terms of topological rooted trees, $\RT$,
\begin{eqnarray}
\Tay_\e(x(t)_t) &=& x_1 +\sum_{n\geq 1} \e^n \sum_{\g \in \RT_n} \frac{1}{\sigma(\g)\g!} D^{x^j}_\g X (x_1,\e p_1)  \nonumber \\
\Tay_\e\left( \frac{1}{t}\int_0^t du \ [L_E(p_2\a)](\phi_u^{p_2\a}(x_1,tp_1)) \right)_t &=& 
\frac{1}{\e} \sum_{m\geq 1} \e^m \sum_{\g\in \RT_m} \frac{a'_\g}{\sigma(\g)} F^{L_E(p_2\a),X}_\g (x_1,\e p_1) \nonumber
\end{eqnarray}
where the elementary differential symbols $ D^{x^j}_\g X$ and $ F^{L_E(p_2\a),X}_\g$ are recalled in Appendix \ref{asub:RT}, as well as the symmetry factor $\sigma(\g)$, the tree factorial $\g!$ and the coefficients $a'_{[\g_1,\dots,\g_n]}=1/(\g_1!(|\g_1|+1)\dots \g_n!(|\g_n|+1))$.
These expressions, together with the recursive formula \eqref{eq:Smrecurs} imply that $S_m$ can be expressed in terms of elementary differentials for the hamiltonian vector field $X$ and the $\a_n$'s. Nevertheless, we need to go deeper into the structure in order to eventually get to Kontsevich trees.

The key observation is the following result, \cite[Thm. 24]{CD}: for $\a=\a_\pi$ defined by equation \eqref{eq:alphapi},
\begin{equation}\label{eq:tayapi}
 \Tay_\e (\a^j_\pi (x,tp) )_t = x^j + \sum_{m\geq 1} \e^m \sum_{\t\in \RT_m} \frac{c^B_\t}{\sigma(\t)} D^{x^j}_\t V_\pi(x,p)
\end{equation}
where the $\tau \mapsto c^B_\tau \in \R$ are universal coefficients (defined by iterated integrals and  generalizing Bernoulli numbers, see \cite{CD}) and the vector field $V_\pi$ is given by $ V_\pi|_{(x,p)} = \pi^{jk}(x) p_k \partial_{x^j}.$
(Notice that $V_\pi$ defines the ODE \eqref{eq:Ppeq} recalled in the introduction.) To continue with our structural exploration, we thus  assume the following form for $\a$,
\begin{equation}\label{eq:aVgen}
 \Tay_\e (\a^j(x,tp) )_t = x^j + \sum_{m\geq 1} \e^m \sum_{\t\in \RT_m} \frac{c_\t}{\sigma(\t)} D^{x^j}_\t V(x,p), \ \ V|_{(x,p)}= a^{jk}(x)p_k\partial_{x^j}
\end{equation}
for some coefficients $\t \mapsto c_\t$ and some smooth field of skew-symmetric matrices $a^{jk}(x)=-a^{kj}(x)$. We observe that we kept $V$ being homogeneous of degree $1$ with respect to $p \mapsto \lambda p$; this will play an important role in the sequel. Also notice that for arbitrary $c$ and $a$, the corresponding $S_t$ will not be a solution of the SGA equation.

The following step is to substitute expression \eqref{eq:aVgen} for $\a$ inside our earlier Butcher series expressed in terms of $X\equiv X^{p_2\a}$. Using the Leibniz rule for the derivatives involved, the results can be arranged graphically in {\bf networks of rooted trees}, $\rho \in \NRT$, and their symbols, $\sym^{p_2,V}_\rho(x,p)$, which are both defined in Appendix \ref{asub:networks} (see Figure \ref{fig:network_Kgraph}). With those definitions, we get the following expressions for our formal Taylor expansions: 
\begin{eqnarray}
\delta_n(p_1,p_2,x_1) &=& \sum_{\rho \in \NRT^*_{n+1}} c^\delta_\rho \sym^{p_2,V,x^j}_\rho (x_1,p_1)  \nonumber \\
p_{2j}\a^j_m(x_1,p_1) &=& \sum_{\rho \in \NRT_m} c^\a_\rho \sym^{p_2,V}(x_1,p_1) \\
I_m(p_1,p_2,x_1) &=& 
\sum_{\rho \in \NRT_m} c^I_\rho \sym^{p_2,V}_\rho(x_1,p_1) \nonumber
\end{eqnarray}
for some coefficient functions $c^\delta,c^\a, c^I: NRT^*\to \R$. Notice that $c^\a_\rho=0$ unless $\g(\rho)$ has only one vertex $r$, in which case $c^\a_\rho=c_{\rho(r)}$, and that $c^\delta, c^I$ can be expressed in terms of $c$'s and tree combinatorial factors.

These expressions, together with the recursive formula \eqref{eq:Smrecurs}, imply that $S_n$ is similarly given as a sum over networks: there exists a coefficient function $\tilde w: \NRT \to \R$ such that
	$$ \Tay_\e(S_t)_t(p_1,p_2,x) = (p_1+p_2)(x)+ \sum_{m\geq 1} \e^m \sum_{\rho \in \NRT_m} \tilde w_\rho \ \sym^{p_2,V}_\rho(x,p_1).$$
This can be proven by induction on $m\geq 1$, where the initial $m=1$ step follows from the expression given in Remark \ref{rmk:S1}, and the inductive step follows by applying the Leibniz rule in \eqref{eq:Smrecurs} and interpreting the resulting terms as the insertion of one network into another. The coefficients $\tilde w$ can be (non-uniquely) recursively defined  using eq. \eqref{eq:Smrecurs} in terms of the $\t \mapsto c_\t$ which define $\a$ through $\eqref{eq:aVgen}$. (This recursion for the $\tilde w$'s is not unique since the symbols $\sym^{p,V}_\rho$ are not always functionally independent, as the case $V=0$ shows.)

Finally, we observe below that for $V=V_\pi$ the network symbol $\sym^{p_2,V}_\rho$ (see Appendix \ref{asub:networks}) coincides with an associated Kontsevich symbol $\hat B^\pi_{\Gamma_\rho}$. We begin associating to a network $\rho \in \NRT_m$ a Kontsevich tree
$$ \Gamma_\rho \in KT_{m,2}$$
defined as follows. Recall that a Kontsevich graph consists of aerial and terrestrial vertices and edges (see Appendix \ref{asub:Ktrees}). All the internal vertices of the network become aereal and all the edges (internal or skeleton) of the network become aereal too. The internal edges are right-pointing in $\Gamma_\rho$ and taken with their orientation towards the root, while the skeleton edges are taken with their additional orientation and become left-pointing in $\Gamma_\rho$. For each vertex in a $\rho(v)$, which is not the source of an oriented external edge, we include a left-pointing edge into the first terrestrial vertex. When a vertex is the root in a $\rho(v)$, we also include a right-pointing edge into the second terrestrial vertex. It is easy to see that the resulting graph $\Gamma_\rho$ is indeed a Kontsevich tree and, moreover, that
	$$ \sym^{p_2,V_\pi}_\rho(x,p_1) = \pm \hat B^\pi_{\Gamma_\rho}(p_1,p_2,x)$$
	where $\hat B^\pi_{\Gamma_\rho}$ denotes the Kontsevich symbol recalled in Appendix \ref{asub:Ktrees} and the $\pm$ sign depends on $\rho$. (See Figure \ref{fig:network_Kgraph} and Examples \ref{ex:kontsymb}, \ref{ex:netsymb} for illustrations of the involved computations.)

We thus conclude the following:
\begin{proposition}\label{prop:structSK}
	Let $S_t$ be a family of functions defined by eq. \eqref{eq:Stgena} from a map $\a$ satisfying \eqref{eq:aVgen} and with $V=V_\pi$. Then, there exists a coefficient function $\tilde w: \NRT \to \R$ such that
	$$ \Tay_\e(S_t)_t(p_1,p_2,x) = (p_1+p_2)(x)+ \sum_{m\geq 1} \e^m \sum_{\rho \in \NRT_m}  \tilde w_\rho \ \hat B^\pi_{\Gamma_\rho}(p_1,p_2,x).$$
	Moreover, the coeficients $\tilde w$ can be (non-uniquely) recursively defined in terms of the coefficients $\t \mapsto c_\t$ which define the map $\a$.
\end{proposition}

We thus see that the structure of the canonical formal family $\bS_\pi=\bS^K_\pi$ in terms of symbols for Kontsevich trees is a consequence of the formula in Remark \ref{rmk:formulatseppi} for the underlying smooth canonical family $\hat S_\pi$. Theorem \ref{thm:main2SK} implies that, when $\a=\a_\pi$ (equiv. when $c=c^B$), the $\tilde w$ coefficients can be taken to be the underlying Kontsevich weights (up to signs and repetition factors), universally for all $\pi$'s.

\begin{remark} (Comparing the weights.)
In \cite{CD}, it was shown that the coefficients $c=c^B$ appearing in eq. \eqref{eq:tayapi} coincide (up to a sign) with the corresponding Kontsevich weight,
$$ c^B_\t = (-1)^{|\t|} W_{\Gamma_\t}, \ \t \in \RT$$
where $\t$ is seen as a network with underlying skeleton having only one vertex. This result implies that these particular Kontsevich weights can be computed by iterated integrals, since the $c^B$'s do by \cite[Thm. 24]{CD}. Proposition \ref{prop:structSK} above then poses a combinatorial challenge: extend the above result of \cite{CD} by finding a universal recursion for the $\tilde w_\rho$'s in terms of the $c$'s and show that they coincide (up to signs and repetitions) with the corresponding Kontsevich weights $W_{\Gamma_\rho}$ when $c=c^B$. The recursion for the $\tilde w$ coefficients can turn out to be an interesting formula for computing the associated Kontsevich weights.
\end{remark}

\section{Relation to Poisson Sigma Model (PSM) action functional}\label{sec:psm}

In this section, the idea is revisit relation \eqref{eq:starSP} where now we think of Kontsevich's $\star_\h$ as being given by the field-theoretic quantization of the PSM following Cattaneo-Felder's work \cite{CFquant}. In this setting, $\star_\h$ is formally expressed as an integral over a "space of fields" and we first, in subsection \ref{subsec:quant}, heuristically characterize which fields yield the semiclassical contribution $S_P$ in terms of an underlying concrete (non-formal) system $(PDE)^\pi$ of PDEs and a (modified) action functional $A'$ (described in detail in subsection \ref{subsec:psmfunc}). The main result of this section, Theorem \ref{thm:SP} below, states that the germ of the canonical generating function $S_\pi$ (see Section \ref{subsec:cangenfun}) is recovered by evaluating a $A'$ on solutions to the system $(PDE)^\pi$. Moreover, we show that such solutions always exist and are classified by certain local groupoid triangles (see Definition \ref{def:triangGpi}). These manipulations provide a non-perturbative functional description of the semiclassical contribution $S_P$ in the PSM; the perturbative description can be found in \cite{CF01bis, CDF}. We also discuss how these constructions establish a concrete relationship between the integration of $(M,\pi)$ and the PSM in the disc (with $3$ insertions at the boundary), generalizing the description of \cite{CFgds} for the PSM on a square.

\subsection{Heuristic motivation from field-theoretic quantization}\label{subsec:quant}
Let us consider a coordinate Poisson manifold $(M,\pi)$.
Following Cattaneo and Felder \cite{CFquant}, Kontsevich's star product \cite{Kontquant} quantizing $(M,\pi)$ can be heuristically expressed as
$$ f_1 \star_\h f_2 (x) = \int_{(X,\eta)\in Z} f_1(X(z_1))f_2(X(z_2))\delta^M_{x}(X(z_3)) e^{\frac{i}{\h}A(X,\eta)}\ \mu_\h(X,\eta).$$
In the above formula, $Z\ni (X,\eta)$ is a (infinite-dimensional) space of maps with $X:D \to M$ and $\eta\in \Omega^1(D,M^*)$, $D$ is the $2$-disk, $\delta^M_x$ is Dirac's delta distribution concentrated at $x\in M$, $z_1,z_2,z_3\in \partial D$ are cyclically oriented points on the boundary, 
$$ A(X,\eta):=\int_D \left[ \eta_j \wedge dX^j + \frac{1}{2} \pi^{jk}(X) \eta_j \wedge \eta_k\right]$$
is the Poisson sigma model (PSM) action functional (\cite{Ikeda,SchStr}) and $\mu_\h(X,\eta)$ is an heuristic 'gauge-invariant' measure on $Z$. (The precise meaning of $\mu_\h$ is studied under the Batalin-Vilkovisky formalism, see \cite{CFquant}.) The definition of $Z$ involves imposing boundary and 'gauge-fixing' conditions:
$$ i_{\partial D}^*\eta = 0, \ d\star \eta = 0,$$
respectively, where $\star$ denotes the Hodge star operator on $D$.

If we heuristically consider the expression given in eq. \eqref{eq:starSP} as an oscillatory integral on $Z$ (see e.g. \cite{GSbook} for the rigorous finite-dimensional theory), and using the Fourier transform expression for Dirac's delta,
$$ \delta^M_{x}(y) = (2\pi\h)^{-dim(M)/2} \int_{p_3\in M^*} e^{-\frac{i}{\h} p_3(y-x)},$$
we heuristically conclude that the function $S_P$ appearing \eqref{eq:starSP} must be given by
$$ S_P(p_1,p_2,x) = A'(X,\eta,p_3),$$
where the \emph{modified PSM action functional}\footnote{This functional could be called the 'PSM action with sources'.} $A'$ is given by
\begin{eqnarray}\label{eq:A'}
A'(X,\eta,p_3)&=& A(X,\eta) + p_{1j}X^j(z_1) + p_{2j}X^j(z_2) - p_{3j}(X^j(z_3)-x).
\end{eqnarray}
and where $(X,\eta,p_3)\in Z\times M^*$ must solve\footnote{In principle, these provide only particular semiclassical contributions consisting of critical points of $A'$ without being restricted to a gauge-fixed subset; as we see in the following subsection, it turns out that these are the contributions that yield $S_\pi$ irrespective of the gauge fixing condition.} the following critical point equations
\begin{eqnarray*}
	d \eta_j &=& - p_{1j} \delta_{z_1} - p_{2j} \delta_{z_2} +  p_{3j}\delta_{z_3} - \frac{1}{2}\partial_{x^j}\pi^{ab}(X) \eta_a \wedge \eta_b \nonumber \\
	d X^j&=& \pi^{kj}(X)\eta_k , \ \ \ X(z_3)= x \nonumber
\end{eqnarray*}
where $\delta_z$ now denotes Dirac's delta distribution supported on $z\in D$. The above system of equations, together with the boundary conditions, suggest a concrete system of PDEs which will be studied in the next subsection. Our main aim is to show that $A'$ restricted to that system's solutions coincides with (the germ of) the canonical generating function  $S_{\pi}$ 
given in Section \ref{subsec:cangenfun} for $(M,\pi)$.

\begin{remark}
	It is interesting to notice that the path integral describing $f_1 \star_\h f_2$ behaves heuristically as an oscillatory integral defining a "Fourier Integral" type of operator $D\subset L^2(M\times M) \to L^2(M)$. In the finite-dimensional case (see \cite{GSbook}) one considers such operators where their integral kernels can be obtained as push-forward along fibrations $p:Z\to M\times M \times M$ of simple kernels defined on $Z$. In the path integral case, $Z$ is formally the space of fields and the fibration is given by evaluation at the marked points $(X,\eta)\mapsto (X(z_1),X(z_2),X(z_3))$. Mimicking the rigorous theory, $S_P$ defines a canonical relation $T^*M\times T^*M \dto T^*M$ (the "wave front of the operator", see also Remark \ref{rmk:lagrels}) which is obtained by taking relative critical points of $A'$ and the role of the "gauge-fixing" condition defining $Z$ is to render these relative critical points non-degenerate.
\end{remark}

\subsection{A system of PDE's and its relation to local symplectic groupoids}\label{subsec:psmfunc}

Let us summarize the conclusion of the previous subsection's heuristic considerations. Let $(M,\pi)$ be a coordinate Poisson manifold. Denoting 
$$D=\{z\in \mathbb{C}: |z|\leq 1 \}\subset \mathbb{C}\equiv \R^2$$ the unit disk, $z_1,z_2,z_3 \in \partial D$ three cyclically oriented points (counterclockwise), we consider a system of PDEs for
$$ X: D \to M, \ \eta \in \Omega^1(D,X^*T^*M) \simeq \Omega^1(D,M^*), \ p_3\in M^*,$$
which is defined by parameters $x\in M, \ p_1,p_2 \in M^*$ by
\begin{eqnarray}
d \eta_j +  \frac{1}{2}\partial_{x^j}\pi^{ab}(X) \eta_a \wedge \eta_b &=&  -p_{1j} \delta_{z_1} - p_{2j} \delta_{z_2} +  p_{3j}\delta_{z_3}  \nonumber \\
d X^j&=& \pi^{ij}(X)\eta_i  \text{ at points in $int(D)$} \nonumber\\
\delta_{z_3}(X)&=& x \nonumber \\
i^*_{\partial D} \eta &=&0  \label{eq:pdes}
\end{eqnarray}
where $\delta_z \in \Omega^2(D)$ denotes the Dirac delta current 2-form supported on $z\in D$. We refer to system \eqref{eq:pdes} as to $(PDE)^\pi_{p_1,p_2,x}$. As we will see shortly, it is clear that $X$ in a solution is typically discontinuous at $z_k$'s and we thus need to specify an extension of the standard delta distributions $\delta_{z_k}$ to such functions. We take the following extension,
\begin{equation}\label{eq:deltaext}
\delta_{z_k}(X) = \underset{\e \to 0}{lim} \frac{1}{\pi}\int_{a(\e)}^{b(\e)} X(z_k + \e e^{i\theta}) \ d\theta, 
\end{equation}
where $\theta \in [a(\e),b(\e)]\subset [0,2\pi]$ defines a small arc in the disc around $z_k$ (and the complex numbers are only used for notational simplicity). Notice that $lim_{\e\to 0}(b(\e)-a(\e)) = \pi$.

In the perturbative computations recalled in the previous subsection, one complements the above system with the \emph{gauge fixing condition} $d \star \eta =0$, for $\star$ denoting the (euclidean) Hodge star on $D$. The effect is to make the solution of the whole system unique, at least perturbatively. In the following subsection, we provide a gauge-fixing independent proof of our main result (and comment on gauge-theoretic interpretation at the end).

\begin{remark}
	The Jacobi identity for $\pi$ ensures an integrability condition for the system \eqref{eq:pdes}. Namely, at a point $z$ in the interior of the disk $D$, the integrability condition $d(dX^j)|_z=0$ is equivalent, by the first and second equations in \eqref{eq:pdes}, to 
	$$\sum_{cyclic \ 1,2,3} \pi^{i_1k}\partial_{x^k}\pi^{i_2i_3}|_{X(z)} \eta_{i_1}\wedge \eta_{i_2} = 0, \forall i_3.$$
\end{remark}
We also consider the functional $A'$ of eq. \eqref{eq:A'} extended according to the definition of the $\delta_{z_k}$'s, so that
\begin{equation}\label{eq:A'reg}
A'(X,\eta,p_3)= \int_D \left[ \eta_j \wedge dX^j + \frac{1}{2} \pi^{jk}(X) \eta_j \wedge \eta_k\right] + p_{1j}\delta_{z_1}(X^j) + p_{2j}\delta_{z_2}(X^j) - p_{3j}(\delta_{z_3}(X^j)-x).
\end{equation}
Observe that solutions to $(PDE)^\pi_{p_1,p_2,x}$ are critical points for $A'$.

The main idea of this section is to recover the canonical generating function $S_\pi$ (see Section \ref{sec:cangenf}) by evaluating $A'$ on the solutions of $(PDE)^\pi_{p_1,p_2,x}$. We will make this precise in the next subsection, but we begin by detailing this construction in the case $\pi$ is a constant Poisson structure.

\begin{example} \textsc{(Solving the PDE's for $\pi$ constant)}\label{ex:conjpiconst}
	Let  $\pi^{ij}(x)=\pi^{ij}$ be a constant Poisson structure on $M$. In this case, the equations for $\eta$ in the system \eqref{eq:pdes}, complemented with the perturbative gauge fixing $d\star \eta=0$, decouple from those of $X$, yielding
	$$  d \eta_j =  -p_{1j} \delta_{z_1} - p_{2j} \delta_{z_2} +  p_{3j}\delta_{z_3} ,\ \ i^*_{\partial D} \eta =0, \ \ d \star \eta = 0.$$
	We first observe that $p_3=p_1+p_2$ is a necessary condition for the solution to exist. Indeed, this follows by integrating over $D$ the first equation above, and using $\int_D d\eta = \int_{\partial D} \eta =0$ by the boundary condition. The solution $\eta$ can be explicitly written in terms of so-called \emph{Kontsevich propagators} for the disc (see e.g. \cite[eq. (14)]{CF08}): for $z_0\in D$,
	\begin{equation}
	\Gamma_{z_0} := \frac{1}{4\pi i}\left[ d_z ln\left( \frac{(z-z_0)(1-z\bar z_0)}{(\bar z - \bar z_0)(1-\bar z z_0)}\right)+ z d\bar z - \bar z dz \right]
	\end{equation}
	(We note that $\Gamma_{z_0}$ is a smooth real valued $1$-form on $D\setminus z_0$; the multiplication in $\mathbb{C}$ is only used as an auxiliary.) When $z_0\in \partial D$, then $i_{\partial D}^*\Gamma_{z_0} = 0$, $d\Gamma_{z_0}= \delta_{z_0}-\frac{1}{\pi}dx\wedge dy$ and $d\star \Gamma_{z_0} = 0$. The second equation for $d\Gamma_{z_0}$ holds in the usual sense of distributions: for a test function $f$ on $D$, denoting $R_\e $ the disc $D$  minus a ball of radius $\e$ around $z_0$, 
	$$  -lim_{\e\to 0}\int_{R_\e} df\wedge\Gamma_{z_0} = lim_{\e\to 0}\left(\int_{\gamma_\e} f\Gamma_{z_0} + \int_{R_\e} f \frac{1}{\pi} dx\wedge dy \right) = f(z_0)- \frac{1}{\pi}\int_D f dxdy,$$
	for $\gamma_\e(\theta)=z_0 + \e e^{i\theta}$ as in \eqref{eq:deltaext}, using $d\Gamma_{z_0} = -\frac{1}{\pi}dx\wedge dy$ on $R_\e$ and by a key estimate (verified by straightforward computation)
	\begin{equation}\label{eq:gammapiconstepsilon} \gamma_\e^*\Gamma_{z_0} = \frac{1}{\pi} d\theta + O(\e).\end{equation}
	The explicit solution for $\eta$ is then
	$$ \eta = -p_1 \Gamma_{z_1} - p_2 \Gamma_{z_2} + (p_1+p_2) \Gamma_{z_3}.$$
	Let us now analyze the behaviour of $X$. Using the second equation in \eqref{eq:pdes} and integrating along a path $\gamma_k$ as in Figure \ref{fig:discpaths}, we get
	\begin{equation} \label{eq:Xjumppiconst} X^j(\gamma_k(1)) - X^j(\gamma_k(0)) = \int_{\gamma_k} dX^j = \pi^{ij}\int_{\gamma_k}\eta_i = \pi^{ij}\int_{R'_k}d\eta_i = -\pi^{ij} p_{ki},\end{equation}
	where, on the third equality, we used Stokes by completing $\gamma_k$ with the piece of the boundary $[\gamma_k(1),\gamma_k(0)]\subset \partial D$ enclosing a region $R_k'$ with  its standard right-hand-rule orientation. This shows that $X|_{\partial D}$ has finite jumps at the $z_k$'s. For $z\in D\setminus \{z_1,z_2,z_3\}$, we have
	$$ X(z) = X(z_0) + \int_{\gamma} \pi^\sharp \eta,$$
	for $\gamma:[0,1]\to \setminus \{z_1,z_2,z_3\}$ any path from $z_0$ to $z$ (since $d\eta=0$ away from the $z_k$'s). In particular, $X|_{\partial D\setminus \{z_1,z_2,z_3\} }$ is locally constant by $i_{\partial D}\eta=0$. Using this expression, we can evaluate our extension \eqref{eq:deltaext} of $\delta_{z_k}$ on $X$, yielding
	$$ \delta_{z_k}(X) = X(z_k^+)+\frac{\sigma_k}{2}\pi^\sharp p_k=\frac{1}{2}( X(z_k^+) + X(z_k^-)).$$
	(The $z^\pm$ notation is with respect to the counter-clockwise orientation on $\partial D$; $\sigma_k=-1$ for $k=1,2$ and $\sigma_3=1$.)
	Finally, the condition $\delta_{z_3}(X)=x$ in the PDE system, together with formula \eqref{eq:Xjumppiconst} for the jumps of $X$ at the $z_k$'s, determine $X$ completely and we obtain the following three values that $X|_{\partial D-\{z_1,z_2,z_3\}}$ takes:
	\begin{equation}\label{eq:Xpiconstboundary} X(z_3^+) = x - \frac{1}{2}\pi^\sharp(p_1+p_2), \ X(z_3^-) = x + \frac{1}{2}\pi^\sharp(p_1+p_2), \ X(z_1^+)=X(z_2^-)=x+\frac{1}{2}\pi^\sharp(p_1) - \frac{1}{2}\pi^\sharp(p_2). \end{equation}
	We can now evaluate the functional $A'$ of eq. \eqref{eq:A'reg} on the solution $(X,\eta)$, yielding
	\begin{eqnarray*}
	A'(X,\eta,p_3) &=& -\frac{1}{2}\int_{\partial D} X^j\eta_j + \frac{1}{2}\int_D X^j d\eta_j +p_{1j}\delta_{z_1}(X^j) + p_{2j} \delta_{z_2}(X^j) \\
	&=&  \frac{1}{2}[(p_1+p_2)(x) + p_{1j}\delta_{z_1}(X^j) + p_{2j} \delta_{z_2}(X^j) ]\\
	&=& (p_1+p_2)(x) + \frac{1}{2} \pi(p_1,p_2)\\
	&=& S_\pi(p_1,p_2,x)
	\end{eqnarray*}
	and thus recovering the canonical generating function $S_\pi$ for $\pi$ constant (see  Examples \ref{ex:S0}, \ref{ex:Spiconstandlin}). Note that this conclusion holds for any other solution $(X,\eta)$ of \eqref{eq:pdes} (not necessarily satisfying the gauge fixing condition $d\star \eta =0$) as long as the estimate \eqref{eq:gammapiconstepsilon} holds.
\end{example}

We observe that, in the above example, $X$ sends the interior of $D$ into the interior of a (euclidean) triangle $T$ in $M$, extending to $\partial D\setminus \{z_1,z_2,z_3\}$ as a locally constant function which takes each arc into the corresponding vertex of $T$ (given in eq. \eqref{eq:Xpiconstboundary}) and such that taking a small arc $\gamma^k_\e$ around $z_k, k=1,2,3$ as in eq. \eqref{eq:deltaext} we obtain a curve from vertex $X(z_k^+)$ to vertex $X(z_k^-)$ which tends (by the estimate \eqref{eq:gammapiconstepsilon}) to the corresponding triangle's line segment edge as $\e \to 0$. This is depicted in Figure \ref{fig:discpaths}(b). The triangle $T\subset M$ lies inside the symplectic leaf of $(M,\pi)$ given by an affine subspace containing $x$. For later reference, we apply this construction to $M=\R^2$ endowed with canonical symplectic $\pi$, choosing $(p_1,p_2,x)$ so as to obtain a map into the standard $2$-simplex:
\begin{equation}\label{eq:defc2simplex} c: D\setminus \{z_1,z_2,z_3\} \to \Delta_2 = \{(t,s)\in \R^2: t+s\leq 1, t,s \geq 0 \},\end{equation}
with $c(z_2^+)=(1,0)$, $c(z_1^+)=(0,0)$ and $c(z_3^+)=(0,1)$.

\newpage

\begin{figure}
		\hskip-1cm\includegraphics[scale=0.5]{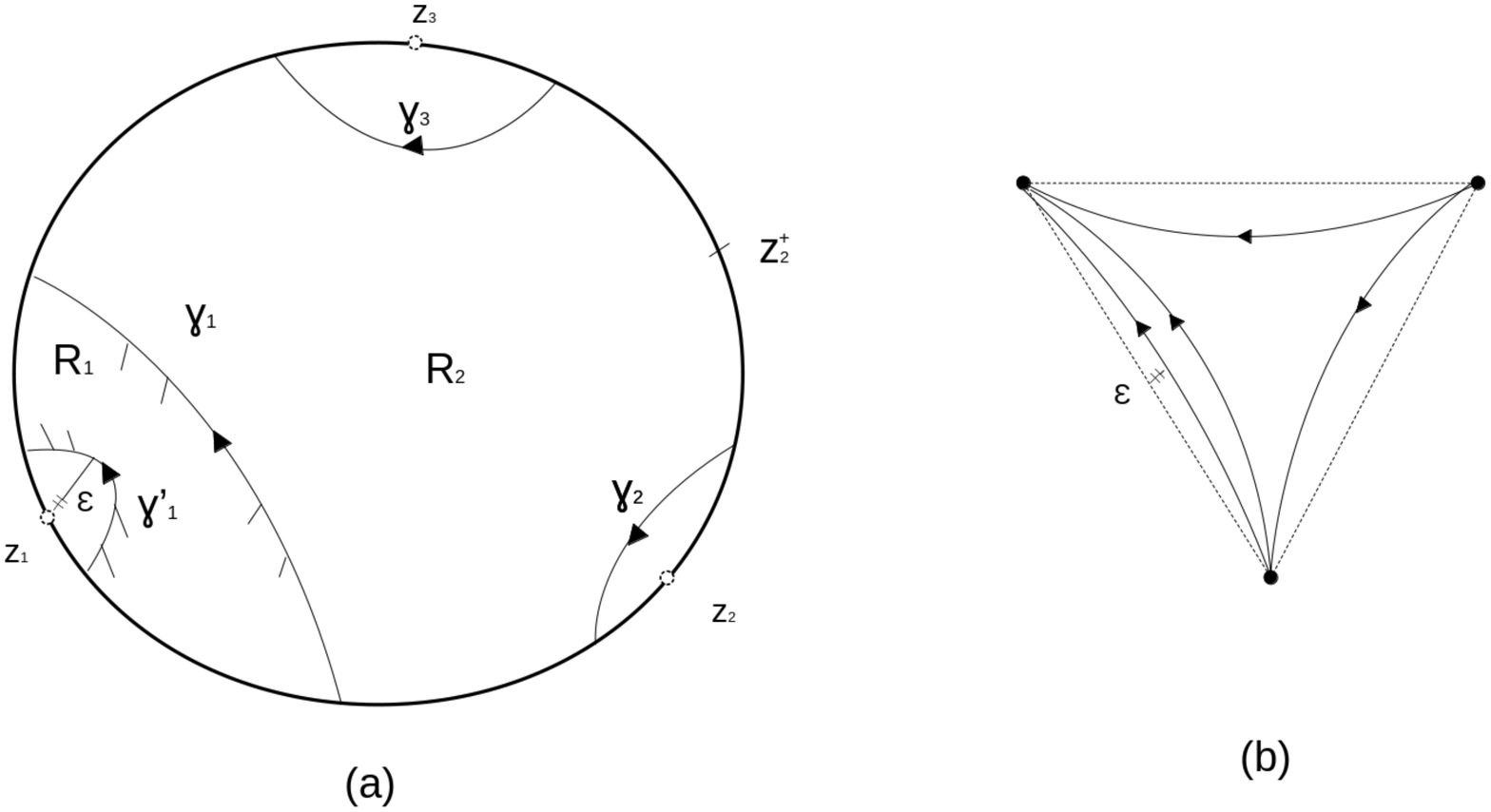} 
	\caption{\label{fig:discpaths} (a): the disc with three marked points in the boundary. Paths $\gamma_k, \ k=1,2,3$ are chosen, as well as an alternative path $\gamma_1'$, and regions $R_1, \ R_2$ are highlighted. (b): an image triangle through $X$ and the corresponding paths.}
\end{figure}
The picture in the above example actually generalizes to general coordinate $(M,\pi)$, yielding non-linear triangles in an integrating local groupoid. We will make this precise in the following subsection, but we anticipate some general features here.

If $\gamma:[0,1] \to D\setminus \{z_1,z_2,z_3\}$ is a smooth path, the composition
$$ a: = (X,\eta)\circ T\gamma: T[0,1]\to T^*_\pi M$$
defines a Lie algebroid morphism, called \emph{cotangent path}. (This condition is equivalent to the pullback of the second equation in \eqref{eq:pdes}, see \cite{CFlie,CFgds}.) If $G$ is any (local) symplectic groupoid integrating $(M,\pi)$, such an algebroid morphism can be integrated (using Lie's second theorem, see e.g. \cite{CMS1} for the local Lie case where $a$ must be "small enough" and next subsection) to a unique path in $G$ defined by
$$ \tilde g: [0,1]\to P_G, \ -\omega_G^\flat(DR_{{\tilde g}^{-1}}d\tilde g ) = a , \ \tilde g(0)=1_{X(\gamma(0))}.$$
If we apply this construction to the paths $\gamma_k$ in Figure \ref{fig:discpaths} we obtain three paths $\tilde g_{k}(t), \ k=1,2,3$. From this picture is clear that the locally constant $X|_{\partial D}$ is actually discontinuous as in Example \ref{ex:conjpiconst}. The interesting part is that, since the pullback along a parameterization $[0,1]\times [0,1]\to R_1$ of the enclosed region between different choices of the $\gamma_k$'s (see Figure \ref{fig:discpaths}) yields a \emph{cotangent homotopy} by definition (the boundary condition on $\eta$ is a key point here), it follows (see \cite{CFlie}) that the endpoints $g_k:= \tilde g_k(t=1)$ are independent of the path taken "around" each $z_k$. Moreover, thinking of the restriction of $(X,\eta)$ to the region $R_2$ of Figure \ref{fig:discpaths}, we get that (see the next subsection and \cite{CFlie})
$$ g_3 = m_G(g_1,g_2).$$
We thus see  that the solutions of our PDE's provide a particular parameterization of the $2$-simplices in $G$ which describe groupoid multiplication. As in Example \ref{ex:conjpiconst}, the idea is that, in this parameterization, the groupoid elements being multiplied are determined by the singularities of the solutions $(X,\eta)$ near the $z_k$'s, which are in turn controlled by the $p_k$'s. In the next subsection, we make these remarks precise. 

\begin{remark}\label{rmk:wegd}
This picture holds, in particular, for $G=W(T^*_\pi M)\rightrightarrows M$ the Weinstein groupoid associated to $(M,\pi)$ (see \cite{CFgds,CFlie}). In that case, elements are cotangent homotopy classes and multiplication is given by concatenation of paths. Its relation to the PSM with source $[0,1]\times [0,1]$ was established in \cite{CFgds}. We now see that the relation to the PSM on the disk with $3$-boundary insertions is through the above parameterized $2$-simplices in the graph of multiplication; see more details below.
\end{remark}

\begin{remark} (Special solutions of the PDE's: identities and inverses.)
	For arbitrary $\pi$, $X=x$ constant and $\eta=0=p_3$ is a solution of $(PDE)^\pi_{0,0,x}$. More generally, when $p_1=p$ and $p_2=0$ (or viceversa) then $p_3=p$ and $\eta = -p \Gamma_{z_1} + p \Gamma_{z_3}$ provides a solution. This corresponds to $(x,0)$ being an identity in the underlying groupoid. Similarly, when $p_1=-p_2=p$, then $p_3=0$ and $\eta = -p\Gamma_{z_1}+p\Gamma_{z_2}$ determine a solution corresponding to the inverse being $(x,p)\mapsto (x,-p)$. The source and target maps of $(x,p)$ on the underlying local groupoid can be computed by $\delta_{z_2}(X)$ and $\delta_{z_1}(X)$, respectively, for $X$ a solution corresponding to $p_1=p, \ p_2=0$ and $p_1=0, \ p_2=p$, respectively.
\end{remark}

\subsection{Recovering $S_\pi$ from the PSM action functional}

In this subsection, we prove that the canonical generating function $S_\pi(p_1,p_2,x)$ can be recovered by evaluating $A'$ on solutions of $(PDE)^\pi_{p_1,p_2,x}$, as well as an existence and classification result for the latter. To that end, we first specify the class of solutions that we will be working on. The focus is on smooth maps defined on the punctured disc
$$ D_* : = D\setminus \{z_1,z_2,z_3\}$$
with specific singular behavior near the $z_k$'s.

\begin{definition}\label{def:strongsol}
	A (strong) {\bf  solution of $(PDE)^\pi_{p_1,p_2,x}$} consists of smooth maps $X:D_* \to M, \ \eta \in \Omega^1(D_*, M^*)$ and $p_3\in M^*$ satisfying
	\begin{eqnarray}
	(d \eta_j +  \frac{1}{2}\partial_{x^j}\pi^{ab}(X) \eta_a \wedge \eta_b)|_z  = 0\ , \ d X^j|_z=\pi^{ij}(X)\eta_i|_z, \ \forall z\in int(D), \nonumber	
	\\ i_{\partial D}^*\eta|_{z\in D_*} =0, \ z_{k,\e}^*\eta = \sigma_k p_k \frac{d\theta}{\pi} + O(\e), k=1,2,3 \text{ and }
	 \ \underset{\e \to 0}{lim} \int_{a(\e)}^{b(\e)} X(z_{3,\e}(1,\theta)) \ \frac{d\theta}{\pi}  = x, \label{eq:strongsol}
	\end{eqnarray}
	where $$ z_{k,\e}(\rho,\theta) = z_k + \e\rho e^{i\theta} \in D_*$$
	are polar coordinates around $z_k$'s and $\sigma_k$ are signs given by $\sigma_1=\sigma_2=-1=-\sigma_3$. 
	
	A (normalized) {\bf family of solutions $(X_\cdot,\eta^\cdot,p_3)$ for $(PDE)^\pi$} consists of a smooth map
	$$ U\times TD_* \to T^*M, (p_1,p_2,x,z,\dot z) \mapsto (X_{p_1,p_2,x}(z),i_{\dot z}\eta^{p_1,p_2,x}|_z,p_3(p_1,p_2,x)) $$
	defining a strong solution of $(PDE)^\pi_{p_1,p_2,x}$ for each $(p_1,p_2,x)$, where $U\subset M^*\times M^* \times M$ is an open neighborhood of $0\times 0 \times M$ and such that ("normalization condition") $X_{0,0,x}=x$ is the constant map and $\eta^{0,0,x}=0=p_3(0,0,x)$.
\end{definition}

Since we do not consider other types of solutions, we often drop the "strong" and "normalized" from the notation. Notice that, as in Example \ref{ex:conjpiconst}, the differential of $d\theta /\pi$ as a distribution yields the delta supported at the underlying $z_k\in \partial D$, so the first and forth equations in \eqref{eq:strongsol} above imply the first equation in \eqref{eq:pdes}. Also notice that the last equation for $X$ is $\delta_{z_3}(X)=x$ with definition \eqref{eq:deltaext}. Hence, a strong solution defines a distributional solution of the original system \eqref{eq:pdes}.

The main result of this section states that evaluating $A'$ on families of solutions yields a generating function for $G_\pi$, that this function agrees with the (germ of the) canonical one $S_\pi$ and that such families always exist and are classified by certain triagles in $G_\pi$.

\begin{theorem} \label{thm:SP}
	Let $(M,\pi)$ be a coordinate Poisson manifold, $A'$ be the (modified) PSM functional (see eqs. \eqref{eq:A'reg} and \eqref{eq:deltaext}) and  $G_\pi$ be the canonical local symplectic groupoid integrating $(M,\pi)$ with generating function $S_\pi$ (see Section \ref{subsec:cangenfun}). Then,
	\begin{enumerate}
		 \item families of solutions for $(PDE)^\pi$ exist and are their germs around $0\times 0 \times M\subset M^*\times M^* \times M$ are classified by families of $G_\pi$-triangles (see Definition \ref{def:triangGpi} below);
		 
		\item any family of solutions $(X_\cdot,\eta^\cdot,p_3)$ for $(PDE)^\pi$ defines a function
		$$ S_P: M^*\times M^* \times M\dtod{0\times 0\times M} \R,  \ S_P(p_1,p_2,x):= A'(X_{p_1,p_2,x},\eta^{p_1,p_2,x},p_3(p_1,p_2,x)),$$
		 which is a coordinate generating function for the canonical local symplectic groupoid $G_\pi$ (see Section \ref{subsec:canG}) and satisfies $S_P(0,0,x)=0$.
	\end{enumerate}
Hence, by Theorem \ref{thm:main1}, the germ of any $S_P$ obtained in this way coincides with that of $S_\pi$,$$ S_P =_{0\times 0 \times M} S_\pi.$$
\end{theorem}

We provide a proof of this theorem in the main text below.
We first notice that the normalization condition on the definition of a family of solutions implies directly
$$ S_P(0,0,x) = 0.$$
Next, we begin with the proof of the existence and classification result, Theorem \ref{thm:SP}(1.), starting with the classification part.

Let $(X,\eta,p_3)$ be a (strong) family of solutions of $(PDE)^\pi$ as in Definition \ref{def:strongsol}. Notice that the first two equations in that definition imply that, for each $(p_1,p_2,x)$ fixed, the corresponding $(X,\eta)$ defines a Lie algebroid morphism $T D_* \to T^*_\pi M$. Hence, fixing any local symplectic groupoid $G$ integrating $(M,\pi)$, by Lie's second theorem for local Lie groupoids (see e.g. \cite{CMS1}) we get:  for $p_1,p_2$ small enough (relative to the domain of the structure maps of $G$), we can integrate this Lie algebroid morphism to a unique smooth map satisfying 
$$ g: D_* \to P_G, \ -\omega_G^\flat(DR_{g^{-1}}dg) = (X,\eta), \ g(z_2^+) = 1_{X(z_2^+)}.$$
(The last base point condition at $z_2^+$ is chosen for later convenience; also notice that $(X,\eta)$ is factored as the total space of the algebroid is $T^*_\pi M= M\times M^*$ is.) From this definition, it follows that the image of $g$ lies in $\a^{-1}(X(z_2^+))$ and that $ \beta\circ g = X$.
Moreover, by the boundary condition on $\eta$, it follows that $g|_{\partial D \setminus \{z_1,z_2,z_3 \} }$ is a locally constant map.

We apply the above construction with $G=G_\pi$ the canonical local symplectic groupoid of Section \ref{subsec:canG} and obtain a smooth family of maps
$$ (M^*\times M^* \times M) \times D_* \dtod{0\times 0 \times M \times D_* } T^*M, \ (p_1,p_2,x, z)\mapsto g_{p_1,p_2,x}(z).$$
We now analyze how the singular behavior of $(X,\eta)$ near the $z_k$'s determines the behavior of $g$ near those points.

\begin{lemma}\label{lem:limeg}
	With the notations above,
	$$ lim_{\e\to 0} g(z_k + \e e^{i\theta}) = \ham^{-\sigma_k p_k\beta_\pi}_{u_k(\theta)}(g(z_k^+)), \ k=1,2,3,$$
	where $\ham^H_u$ denotes the hamiltonian flow on $(T^*M,\omega_c)$ (see Section \ref{sec:lsg}),  $u_k(\theta) = \frac{1}{\pi} (\theta - lim_{\e\to 0}a_k(\e))$ and $[a_k(\e),b_k(\e)]\subset [0,2\pi]$ is the domain of $\theta$ so that $z_k + \e e^{i\theta} \in D$.
\end{lemma}
\begin{proof}[of the Lemma]
 Denote $\gamma_{k,\e}(\theta)=z_k + \e e^{i\theta}$ the curves in $D$ with $\theta \in [a_k(\e),b_k(\e)]$ as in the statement. On the one hand, by the definition of strong solution $(X,\eta)$ we have 
 $$ \gamma_{k,\e}^*\eta = \sigma_k p_k \frac{d\theta}{\pi} + O(\e).$$
 Setting $g_{k,\e}(\theta):=g(\gamma_{k,\e}(\theta))g(z_k^+)^{-1}$, the definition of $g$ implies that $g_{k,\e}$ are solutions of the ODEs
 $$- \omega_c^\flat(TR_{g_{k,\e}^{-1}}\frac{d}{d\theta} g_{k,\e}) = (\b_\pi(g_{k,\e}), \frac{\sigma_k}{\pi} p_k  + O(\e)), \ g_{k,\e}(a_k(\e))=(X(z_k^+),0).$$
 On the other hand, in Remark \ref{rmk:rightinvarODE} we observed that the solution of the ODE for a curve $t\mapsto g(t)\in \a_\pi^{-1}(y)$,
 $$ -\omega_c^\flat(TR_{g^{-1}}\frac{d}{dt}g)) = (\b_\pi(g),p) , \ g(0)=(y,0) \text{ is given by } g(t)=\ham_{t}^{-p\b_\pi}(y,0).$$
 The Lemma thus follows directly by taking $\e \to 0$ and continuity of solutions with respect to parameters.
\end{proof}

We now observe some direct a consequences of the Lemma, using the fact that the flows $\ham_u^{p\b_\pi}$ are right invariant in $G_\pi$ (see Section \ref{sec:lsg}). First, the groupoid elements $g_1,g_2,g_3$ in $G_\pi$ defined by integrating the cotangent paths $(X,\eta)\circ T\gamma_k:T [0,1]\to T^*_\pi M$, with $\gamma_k$ as in Figure \ref{fig:discpaths}, can be computed along the paths $z_{k,\e}(1,\theta)$ (notice the $z_{3,\e}$ has the opposite orientation of $\gamma_3$, and $g_3=g(z_3^+)$) and, since they are independent of $\e$ as discussed in the previous subsection, the Lemma allows us to compute them by taking $\e\to 0$. We arrive to the following identities
\begin{eqnarray} g_2 = \ham^{p_2\b_\pi}_{u=1}(X(z_2^+),0)=g(z_1^+), \ g_1 =  \ham^{p_1\b_\pi}_{u=1}(\b_\pi(g_2),0)=g(z_3^+)g(z_1^+)^{-1}, \nonumber \\  \ham_{u=1}^{-p_3\b_\pi}(g(z_3^+)) = (X(z_2^+),0).\label{eq:gksfromXeta}
\end{eqnarray}
The above identities involving the map $g(z)$ follow from its definition, the fact that algebroid homotopies integrate to homotopies with fixed endpoints (see \cite{CFlie}) and uniqueness of solutions together with right invariance of the underlying differential equations. In particular, $g_3=m_\pi(g_1,g_2)$ as observed from general principles in the previous subsection and, by $\ham^{-H}_u = \ham^{H}_{-u}$, we also conclude
$$ g_3= \ham_{u=1}^{p_3\b_\pi}(X(z_2^+),0).$$
We further claim that we can compute the components of $g_k\in M\times M^*$ as
\begin{equation}\label{eq:gkprojs}
g_k = (\delta_{z_k}(X), p_k), k = 1,2,3,
\end{equation}
where $\delta_{z_k}(X)$ is defined as an integral in eq. \eqref{eq:deltaext}.
To get the $r$-projections of the $g_k$'s, using the first property of $\a_\pi$ in eq. \eqref{eq:rhamu} (recalling $\b_\pi(x,p)=\a_\pi(x,-p)$), we arrive from \eqref{eq:gksfromXeta} to
\begin{equation*} r(g_k) = p_k, \ k=1,2,3,  \end{equation*}
for each $(p_1,p_2,x)$ small enough. To understand the projections $qg_k$, following the proof of the Lemma and using the fact that $\b_\pi$ is an anti-Poisson map, we first obtain that
\begin{equation*}
lim_{\e \to 0} X(z_k+\e e^{i\theta}) = \varphi_{\pi,p=-\sigma_kp_k}^{u=\theta/\pi}(X(z_k^+)),
\end{equation*}
recalling that $\varphi^u_{\pi,p}$ denotes the flow of eq. \eqref{eq:Ppeq}.
Inserting this limit into the integral defining the delta's and using $X(z_k^+)= \a_\pi(g_k^{-\sigma_k})$ by definition, we get (recall that the inverse is given by $(x,p)\mapsto (x,-p)$ so that the $q$-projection remains constant)
$$ \delta_{z_k}(X)=\underset{\e \to 0}{lim}\int_{a_k(\e)}^{b_k(\e)} X(z_k + \e e^{i\theta}) \  \frac{d\theta}{\pi} = \int_0^1 \varphi^t_{\pi, p=-\sigma_k p_k}(\a_\pi(qg_k, -\sigma_k p_k)) = qg_k$$
where the last equality follows from the defining property of $\a_\pi$, eq. \eqref{eq:alphapi}. This proves \eqref{eq:gkprojs}.

In particular, since $S_\pi$ is a generating function for $G_\pi$ and using $\delta_{z_3}(X)=x$ from \eqref{eq:pdes}, it follows that $p_3 = \partial_x S_\pi (p_1,p_2,x)$. We thus obtained the following result: if $(X,\eta,p_3)$ is a family of solutions of $(PDE)^\pi$, then 
\begin{equation}\label{eq:p3neccond} p_3 =_{0\times 0 \times M} \partial_x S_\pi.\end{equation}

To continue with the characterization of families of solutions, we next observe that the above lemma implies that $g$ actually spans a smooth triangle in the $\a_\pi$-fiber. Taking the map $c:D_*\to \Delta_2$ into the standard $2$-simplex defined in \eqref{eq:defc2simplex}, we see that
$$ \hat g\equiv \hat g_{X,\eta} := g \circ c^{-1}: int(\Delta_2) \to T^*M$$
extends by the above Lemma to a smooth map defined on $\Delta_2$ (recall from Example \ref{ex:conjpiconst} that $c(z_k + \e e^{i\theta})$ is an $\e$-family of curves approaching the corresponding edge of $\Delta_2$ as $\e\to 0$). This motivates the following definition.

\begin{definition}\label{def:triangGpi}
	A {\bf $G_\pi$-triangle generated by $(p_1,p_2,x)$} is a smooth map
	$$ \hat g: \Delta_2 \to T^*M$$
	satisfying:
	\begin{eqnarray}
	r\hat g(1,0) &=& 0 \text{ so $\hat g(1,0)=(y,0)$ is an identity,} \nonumber \\
	\hat g(1-u,0) &=& \ham_{u}^{p_2\b_\pi}(y,0), \	\hat g(0,u) = \ham_{u}^{p_1\b_\pi} \ham_{t=1}^{p_2\b_\pi}(y,0), \
	\hat g(1-u,u) = \ham_u^{p_3\b_\pi} (y,0), \nonumber \\
	q \hat g(0,1) &=& x, \label{eq:triangle}
	\end{eqnarray} 
	where $p_3 = \partial_x S_\pi (p_1,p_2,x)$. A (normalized) {\bf family of $G_\pi$-triangles} is a smooth map
	$$ (M^*\times M^* \times M) \times \Delta_2 \dtod{0\times 0\times M \times \Delta_2} T^*M, (p_1,p_2,x, c) \mapsto \hat g_{p_1,p_2,x}(c)$$
	which defines a $G_\pi$-triangle for each generator $(p_1,p_2,x)$ and such that $\hat g_{0,0,x} = (x,0)$ is the constant map.
\end{definition}

Let us verify that $\hat g \equiv \hat g_{X,\eta}$ as defined above from a family of solutions $(X,\eta,p_3)$ actually defines a family of $G_\pi$-triangles. We already observed that $\hat g$ defines a smooth map $\Delta_2\to T^*M$. Moreover, for each $(p_1,p_2,x)$ fixed, the conditions \eqref{eq:triangle} follow directly from Lemma \ref{lem:limeg} and the discussion after it (in particular, using \eqref{eq:gkprojs} for verifying $q\hat g(0,1)=x$). We thus conclude the following: given a family of solutions $(X,\eta,p_3)$, the associated $\hat g_{X,\eta}$ defines a family of $G_\pi$-triangles. 

Conversely, if $\hat g$ is a family of $G_\pi$-triangles, we define a family of maps $(X,\eta,p_3)$ by
$$ (X,\eta)_{p_1,p_2,x}:= -\omega_c^\flat(R_{(\hat g_{p_1,p_2,x}\circ c)^{-1}}d(\hat g_{p_1,p_2,x}\circ c) ), \ \ p_3 := \partial_x S_\pi(p_1,p_2,x).$$
By  the definition of $c: D_*\to \Delta_2$, it follows directly that $(X,\eta,p_3)$ defines a smooth family of maps. Fixing small $(p_1,p_2,x)$, the first line of PDE's in \eqref{eq:strongsol} follows directly since $(X,\eta)$ is an algebroid morphism by construction. The second line of 'boundary' conditions in \eqref{eq:strongsol} also follows directly from the defining properties of $\hat g$ and of the map $c$ (for verifying $\delta_{z_3}(X)=x$ one uses the argument underlying eq. \eqref{eq:gkprojs}). We have then proved the following:
\begin{proposition}
	With the notations above,  the assignment $\hat g \mapsto (X,\eta,p_3)$ defines a one to one correspondence between germs of families of $G_\pi$-triangles and families of solutions of $(PDE)^\pi$ around $0\times 0\times M \subset M^*\times M^* \times M$.
\end{proposition}

To show existence of families of solutions, we proceed by constructing a "(germ) canonical" family of $G_\pi$-triangles as follows. Given $(p_1,p_2,x)$ close enough to $0\times 0 \times M$, by inverse function theorem, there exists a unique $y\equiv y(p_1,p_2,x)$ such that
$$q \ham_{u=1}^{p_1\b_\pi}\ham_{u=1}^{p_2\b_\pi}( y, 0) = x,$$
depending smoothly on the parameters. Also, since the map
$$ M^*\times M \dtod{0\times M} T^*M, (p,y)\mapsto \ham_{u=1}^{p\b_\pi}(y,0)$$
is a local diffeomorphism, there exists a unique smooth curve $\tilde p\equiv \tilde p_{p_1,p_2,x}:[0,1]\to M^*$ such that
$$ \ham_{u=1}^{\tilde p(s)\b_\pi}(y_{p_1,p_2,x}) = \ham_{s}^{p_1\b_\pi}\ham_{u=1}^{p_2\b_\pi}(y_{p_1,p_2,x}).$$
We then set
$$ \hat g\equiv \hat g_{p_1,p_2,x}: \Delta_2 \to T^*M, \ \hat g(t,s) = \ham_{1-t}^{\tilde p(\frac{s}{1-t})\b_\pi}(y,0).$$
It is straightforward to verify that $\hat g$ is smooth (it comes from a an obvious smooth map defined on $[0,1]\times [0,1]$ which collapses the edge $(1,\tilde s)$ into $(y,0), \ \forall \tilde s$). Moreover, by definition it satisfies all the conditions in \eqref{eq:triangle} as well as $\hat g_{0,0,x}=x$, thus defining a family of $G_\pi$-triangles, as wanted. This concludes the proof of Theorem \ref{thm:SP}(1.).

\bigskip

Finally, the proof of Theorem \ref{thm:SP}(2.) goes as follows. Let $(X,\eta,p_3)$ be a family of solutions of $(PDE)^\pi$ and $S_P$ the associated function defined in the statement. 
We need to show that given $(g_1,g_2,m_\pi(g_1,g_2))$ any point in the graph of groupoid multiplication in $G_\pi$, which is close enough to the identities, then
$$ g_1 = (\partial_{p_1}S_P(p_1,p_2,x), p_1), \ g_2 = (\partial_{p_2}S_P(p_1,p_2,x), p_2), \ m_\pi(g_1,g_2) = (x,\partial_xS_P(p_1,p_2,x)),$$
for some small $(p_1,p_2,x)$.

On the other hand, since the PDE system corresponds to the critical points of $A'$ given in eq. \eqref{eq:A'reg}, by applying the chain rule we can compute the partial derivatives of $S_P$, yielding
\begin{eqnarray}
\partial_{p_1}S_P(p_1,p_2,x) &=& \delta_{z_1}(X_{p_1,p_2,x})\nonumber \\
\partial_{p_2}S_P(p_1,p_2,x) &=& \delta_{z_2}(X_{p_1,p_2,x}) \nonumber\\
\partial_{x}S_P(p_1,p_2,x) &=&  p_3(p_1,p_2,x). \label{eq:dersSP}
\end{eqnarray}
We thus need to show that, for each composable $(g_1,g_2)$ close enough to the identities,
\begin{equation}\label{eq:gksverif} g_k = (\delta_{z_k}(X),p_k), \ k=1,2, \ m_\pi(g_1,g_2) = (x,p_3(p_1,p_2,x)),\end{equation}
for some $(p_1,p_2,x)$ close to $(0,0,x)$. 

Given the family of solutions $(X,\eta,p_3)$, consider the integration $g:D_*\to T^*M$ as for Lemma \ref{lem:limeg} and the underlying element $(g_1,g_2,g_3)$ in the graph of multiplication defined as in \eqref{eq:gksfromXeta}. It is clear from this definition (and the definition of $G_\pi$) that every element in the graph of multiplication which is close to the identities can be obtained in this way, from a given family of solutions, by evaluating it on $(p_1,p_2,x)$ in a neighborhood of $0\times 0 \times M$. Moreover, by equations \eqref{eq:gkprojs} and \eqref{eq:p3neccond}, it follows that \eqref{eq:gksverif} indeed holds, thus showing that $S_P$ indeed defines a generating function. This concludes the proof of Theorem \ref{thm:SP}.

\bigskip

{\bf Final remarks on $A'$ and $S_\pi$.}
Theorem \ref{thm:SP} provides a non-perturbative (semiclassical) functional definition of the canonical generating function $S_\pi$ in terms of the maps underlying the PSM.

\begin{itemize}
	\item Notice that, as a consequence of Theorem \ref{thm:main2SK}, the formal Taylor expansion of $S_P$ along $t\pi$ centered at $t=0$ coincides with Kontsevich's tree-level generating function $\bar S_\pi^K$. This was expected from the PSM perspective, since $\bar S_\pi^K$ corresponds to the tree-level Feynman expansion of the same path integral for \eqref{eq:starSP}. Moreover, if we consider the formal family $\e\pi$ in the system of PDEs \eqref{eq:pdes} together with the gauge fixing $d\star \eta =0$ and solve perturbatively for $X=X_0 + \e X_1 +\dots$ and $\eta=\eta_0 + \e \eta_1 +\dots$, with $(X_0,\eta_0)$ the solution for $\pi=0$ (see Example \ref{ex:conjpiconst}), one should recover $\bar S_\pi^K$ by inserting this formal expansion into $A'$.

	\item (Nomenclature) The action $S_P$ can be called \emph{Hamilton-Jacobi action} (by analogy with the similar situation in classical mechanics) and the solutions $(X,\eta)$ of \eqref{eq:pdes} can be called \emph{instantons} for the PSM (with insertions).
	
	\item Different solutions of \eqref{eq:pdes} (equiv. different triangles in $G_\pi$ with the same boundary) can be thought of as corresponding to different gauge fixings for the PSM. Theorem \ref{thm:SP} is a "gauge invariant" result in that the corresponding generating function (germ) is always $S_\pi$. 
	
	\item The SGA equation \eqref{eq:SGA} for $S_\pi$ can be understood from the functional perspective as follows. In terms of discs, the familiar picture (also used for associativity of the star product) consists of thinking of a disk with $4$ marked points on the boundary which can be 'pinched' into two discs glued at a boundary point (in two different ways, providing the SGA identity). This gluing of solutions on the disk corresponds to gluing of the underlying triangles along edges. In this way, the SGA equation can be understood in terms of the familiar simplicial picture in Lie theory (filling a tetrahedron) and using the "gauge invariance" above to get the SGA identity using deformations.
	
	 \item The key relation between the functional perspective involving $A'$ and the local groupoid $G_\pi$ is ultimately related to the extension of the deltas to $\delta_{z_k}(X)$ in eq. \eqref{eq:deltaext}. This choice of extension results in a precise connection to the integral \eqref{eq:alphapi} defining the realization $\a_\pi$ which determines the whole $G_\pi$ structure.
	 
	 \item (Relation to Gauge theory) For a linear Poisson structure so that $M\simeq \gg^*$ (see Example \ref{ex:cotangentH}), the equations for $\eta \in \Omega^1(D,\gg)$ in \eqref{eq:pdes} decouple from those for $X$. The first equation says that $\eta$ defines a principal connection on $M\times H$ (here $Lie(H)=\gg$ is an integrating Lie group) whose the curvature is concentrated at the points $z_k$ with "intensity" given by the $p_k\in \gg$. The map $X:D_*\to \gg^*$ can be obtained by parallel transport in the coadjoint bundle $D\times \gg^*$. It is interesting to notice that the BCH formula (embodied in $S_\pi$ as in Example \ref{ex:Sbch}) can then be recovered by evaluating the functional $A'$ on these gauge-theoretic objects $(X,\eta)$.
	 
	 \item (Non-coordinate cases) It is clear that the picture with $G_\pi$-triangles generalizes naturally to any local symplectic groupoid $G$ integrating an arbitrary Poisson $(M,\pi)$. A corresponding system of PDEs can be written with the aid of a connection and should be the natural substitute for the coordinate system \eqref{eq:pdes} and for their role in the PSM.
	 
\end{itemize}

\medskip

\begin{appendix}

\section{Appendix: Kontsevich graphs, Butcher series and networks}\label{app:butK}

In this Appendix, we first recall basic definitions involving Kontsevich graphs and symbols as well as of Butcher series parameterized by rooted trees. In Subsection \ref{asub:networks}, we introduce a certain type of graphs built from rooted trees which we call "networks of rooted trees" and define associated symbols. These networks make the bridge between certain Butcher series expressions for generating functions to the Kontsevich-trees generating function \eqref{eq:defSK}, as explained in Section \ref{subsec:graphexp}.

\subsection{Kontsevich trees and their symbols}\label{asub:Ktrees}

A \textbf{Kontsevich graph} of type $(n,m)$ is a graph $(V,E)$ whose
vertex set is partitioned in two sets of vertices $V=V^{a}\sqcup V^{g}$,
the \textbf{aerial vertices} $V^{a}=\{1,\dots,n\}$ and the \textbf{terrestrial
	vertices} $V^{g}=\{\bar{1},\dots,\bar{m}\}$ such that 
\begin{itemize}
	\item all edges start from the set $V^{a}$, 
	\item loops are not allowed, 
	\item there are exactly two edges going out of a given vertex $k\in V^{a}$, 
	\item the two edges going out of $k\in V^{a}$ are ordered, the first ("left") one
	being denoted by $e_{k}^{1}$ and the second ("right") one by $e_{k}^{2}$. 
\end{itemize}
An \textbf{aerial edge} is an edge whose end vertex is aerial, and
a \textbf{terrestrial edge} is an edge whose end vertex is terrestrial.
We denote by $G_{n,m}$ the set of Kontsevich graphs of type $(n,m)$.

Given a Poisson structure $\pi$ and a Kontsevich graph $\Gamma\in G_{n,m}$,
one can associate a $m$-differential operator $B_{\Gamma}(\pi)$
on $\R^{n}$ in the following way: For $f_{1},\dots,f_{m}\in C^{\infty}(\R^{n})$,
we define 
\[
B_{\Gamma}(\pi)(f_{1}\dots,f_{m}):=\sum_{I:E_{\Gamma}\rightarrow\{1,\dots,d\}}\big[\prod_{k\in V_{\Gamma}^{a}}(\prod_{\substack{e\in E_{\Gamma}\\
		e=(*,k)
	}
}\partial_{I(e)})\pi^{I(e_{k}^{1})I(e_{k}^{2})}\big]\prod_{i\in V_{\Gamma}^{g}}\big(\prod_{\substack{e\in E_{\Gamma}\\
		e=(*,i)
	}
}\partial_{I(e)}\big)f_{i}.
\]
The symbol $\hat{B}_{\Gamma}$ of $B_{\Gamma}$ is defined by 
\[
B_{\Gamma}(e^{p_{1}x},\dots,e^{p_{m}x})=\hat{B}_{\Gamma}(p_{1},\dots,p_{m})e^{(p_{1}+\dots+p_{m})x}.
\]

\begin{example}\label{ex:kontsymb}
	Figure \ref{fig:network_Kgraph} (b) shows an example of a Kontsevich graph $\Gamma\equiv \Gamma_\rho$ of type $(n,2)$.
	The corresponding symbol is given by
	$$ \hat{B}_\Gamma(\pi)(p_1,p_2) = \pi^{k_1k_2} \ \pi^{j_1j_2} \ \partial_{k_2}\partial_{j_1}\pi^{i_1i_2} \ \partial_{k_1}\pi^{l_1l_2} \ p_{1i_1}p_{1l_1}p_{2i_2}p_{2l_2}p_{2j_2}.$$
	The order of the arrows (left vs. right arrows) is important because flipping, for example,
	the order of the first aerial vertex order would introduce a sign, since $\pi^{ij}=-\pi^{ji}$. 
\end{example}
$T_{n,2}$ is a subset of $G_{n,2}$ that we now define:
\begin{definition}
	Let $\Gamma\in G_{m,n}$ be a Kontsevich graph. The \textbf{interior}
	of $\Gamma$ is the graph $\Gamma_{i}$ obtained from $\Gamma$ by
	removing all terrestrial vertices and terrestrial edges. A Kontsevich
	graph is a \textbf{Kontsevich tree} if its interior is a tree in the
	usual sense (i.e. it has no cycles). We denote by $T_{n,m}$ the set
	of Kontsevich's trees of type $(n,m)$. 
\end{definition}

We will not need a detailed presentation of Kontsevich weights $\Gamma \mapsto W_\Gamma\in \R$, the reader can find a rough account with the conventions that we follow in this paper in \cite[App. A]{CD}.

\subsection{Rooted trees, elementary differentials and Butcher series}\label{asub:RT}

A \textbf{graph} is the data $(V,E)$ of a finite set of vertices
$V=\{v_{1},\dots,v_{n}\}$ together with a set of edges $E$, which
is a subset of $V\times V$. The number of vertices is called the
\textbf{degree} of the graph and is denoted by $|\Gamma|$. We think
of $(v_{1},v_{2})\in E$ as an arrow that starts at the vertex $v_{1}$
and ends at $v_{2}$.
Two graphs are \textbf{isomorphic} if there is a bijection between
their vertices that respects theirs edges. The set $\bar{\Gamma}$
of all isomorphic graphs $ $to a given graph $\Gamma$ is called
a \textbf{topological graph}.
A \textbf{symmetry} of a graph is an automorphism of the graph (i.e.
a relabeling of its vertices that leaves the graph unchanged). The
group of symmetries of a given graph $\Gamma$ will be denoted by
$\sym(\Gamma)$. Note that the number of symmetries of all graphs
sharing the same underlying topological graph is equal; we define
the \textbf{symmetry coefficient} $\sigma(\overline{\Gamma})$ of
a topological graph $\overline{\Gamma}$ to be the number of elements
in $\sym(\Gamma)$, where $\Gamma\in\overline{\Gamma}$.

A \textbf{rooted tree} is a graph that (1) contains no cycle, (2)
has a distinguished vertex called the \textbf{root}, (3) whose set
of edges is oriented toward the root. We will denote the set of rooted
trees by $RT$ and the set of topological rooted trees by $[RT]$.
The set of topological rooted trees can be described recursively as
follows: The single vertex graph $\bullet$ is in $[RT]$ and if $t_{1},\dots,t_{n}\in[RT]$
then so is 
\[
t=[t_{1},\dots,t_{n}]_{\bullet},
\]
where the bracket is to be thought as grafting the roots of $t_{1},\dots,t_{n}$
to a new root, which is symbolized by the subscript $\bullet$ in
the bracket $[\,,\dots,\,]$. We can represent graphically topological rooted trees as the $\gamma\equiv \gamma(\rho)=[\bullet,\bullet]_\bullet$ depicted in Figure \ref{fig:network_Kgraph} (a) (without the extra orientation of edges present in that figure).

\begin{remark}
	Since we are dealing with topological rooted trees, the ordering in
	$[t_{1},\dots,t_{m}]_{\bullet}$ is not important (for instance, we
	do not distinguish between $[\bullet,[\bullet]_{\bullet}]_{\bullet}$
	and $[[\bullet]_{\bullet},\bullet]_{\bullet}$).
\end{remark}

Let $X=X^{i}(x)\partial_{x^i}$ be a vector field on $\R^{d}$. We define the \textbf{elementary differential} of $X$ recursively as follows:
	For the single vertex tree, we define $D_{\bullet}^{u}X=X^{u}(x)$,
	and for $t=[t_{1},\dots,t_{m}]$ in $[RT]$, we define 
	\begin{equation}
	D_{t}^{u}X(x)=D_{t_{1}}^{i_{1}}X(x)\cdots D_{t_{m}}^{i_{m}}X(x) \ \partial_{i_{1}}\dots\partial_{i_{m}}X^{u}(x),\label{eq:rec. form.}
	\end{equation}
	where we used the Einstein summation convention. 
Similarly, given a smooth function $H:\R^d \to \R$ we define recursively for $t=[t_{1},\dots,t_{m}]$,
$$ F^{H,X}_{t}(x) = D^{i_1}_{t_1}X(x)\cdots D^{i_m}_{t_m}X(x) \ \partial_{i_1}\dots\partial_{i_m}H (x), \ F^{H,X}_{\bullet}(x)=H(x).$$
Following \cite{Bu}, one has that the $n$-th iterated Lie derivative is given by
\begin{equation}\label{eq:Xn}
L_X^n x^i = \sum_{t \in[ RT], \ |t|=n} \frac{|t|!}{t! \sigma(t)} D^i_t X,
\end{equation}
where the \emph{tree factorial} is recursively defined as $t!= |t| t_1! \cdots t_k!$ for $t=[t_1,..,t_k]$, $\bullet ! =1$. The symmetry factor $\sigma(t)$ also admits a recursive formula, namely, $\sigma([t_1^{n_1},..,t_1^{n_k}]) = n_1!\sigma(t_1)^{n_1} ...  n_k!\sigma(t_k)^{n_k}$ where the trees $t_1,..,t_k$ are assumed different and the exponent $n_i$ denotes it is repeated $n_i$-times.

Given a map $a:[RT] \to \R$ and a vector field $X\equiv X^i(x)\partial_{x^i}$ on $\R^n$ the associated {\bf Butcher series} is defined by
$$ B(a, \e X, x) = x + \sum_{t \in [RT]} \frac{\e^{|t|}}{\sigma(t)} \ a_t \ D_t X |_x.$$
One can extend this assignment to $B(a, \e X, H)\equiv H|_{y= B(a, \e X, x)}$ for any function $H\equiv H(x)$ by expanding formally around $\e=0$ and using the symbols $F^{H,X}_t$ defined above (see "S-series" in \cite{CHV}).
The reader is referred to the foundational work of Butcher \cite{Bu} on elementary differentials and the use of trees and Butcher series in ordinary differential equations for more information (see also \cite{CHV}).

\begin{example}
	Following \cite[Thm. 21]{CD}, the formal Taylor expansion around $t=0$ of the realization map $\a_{t\pi}(x,p)$ defined by eq. \eqref{eq:alphapi} is given by the Butcher series of eq. \eqref{eq:tayapi}. The coefficients $t\mapsto c^B_t$ generalize Bernoulli numbers and can be computed by iterated integrals (see \cite[Thm. 24]{CD}).
\end{example}

\subsection{Networks of rooted trees and their symbols}\label{asub:networks}

The motivation for introducing networks of rooted trees comes from considering elementary differentials of the form $D_\gamma^iX$, with $\gamma\in [RT]$ and $X=X^{p_1\a}$ a hamiltonian vector field in $M\times M^*$ as in \eqref{eq:Xpa} ($M$ is a coordinate space), and substituting the Butcher series \eqref{eq:aVgen} in place of $\a:M\times M^*\dtod{M\times 0} M$. The idea is that each derivative in the elementary differential acting on each term of $\a$ produces, by the Leibniz rule, a series of terms which can be arranged into these "network" graphs. To account for the two terms in the hamiltonian vector field $X^{p_1\a}$ we consider an extra orientation on the ("skeleton") edges of $\gamma$, which tells us which term is acting.

With these considerations, a {\bf network of rooted trees} $\rho \in \NRT$ consists of the following data: 
\begin{itemize}
	\item a 
	rooted tree $\g\equiv \g(\rho)\in RT$, 
	called the \emph{skeleton of $\rho$}, endowed with an additional orientation on its edges (besides the natural one towards the root);
	
	\item for each vertex $v$ of $\g$, a rooted tree $\rho(v) \in \RT$;
	
	\item for each additionally oriented edge $e$ of $\g$ going from $v$ to $w$, a pair of vertices $\rho(e)_1 \in \rho(v) \ ,\rho(e)_2 \in \rho(w)$.
\end{itemize}
We also consider the case in which there is an extra marked vertex $\rho^* \in \rho(r(\g))$, where $r(\g)\in \g$ is the root, and in this case we denote $\rho \in \NRT^*$. The sum of all the vertices on the various $\rho(v), \ v\in \g,$ yields the total number of vertices in the network, denoted $|\rho|$, and we say $\rho \in \NRT_{|\rho|}$. 
An example of a network of rooted trees is illustrated in Figure \ref{fig:network_Kgraph} (a).
%

\begin{figure}
	\hskip2.5cm \includegraphics[scale=0.35]{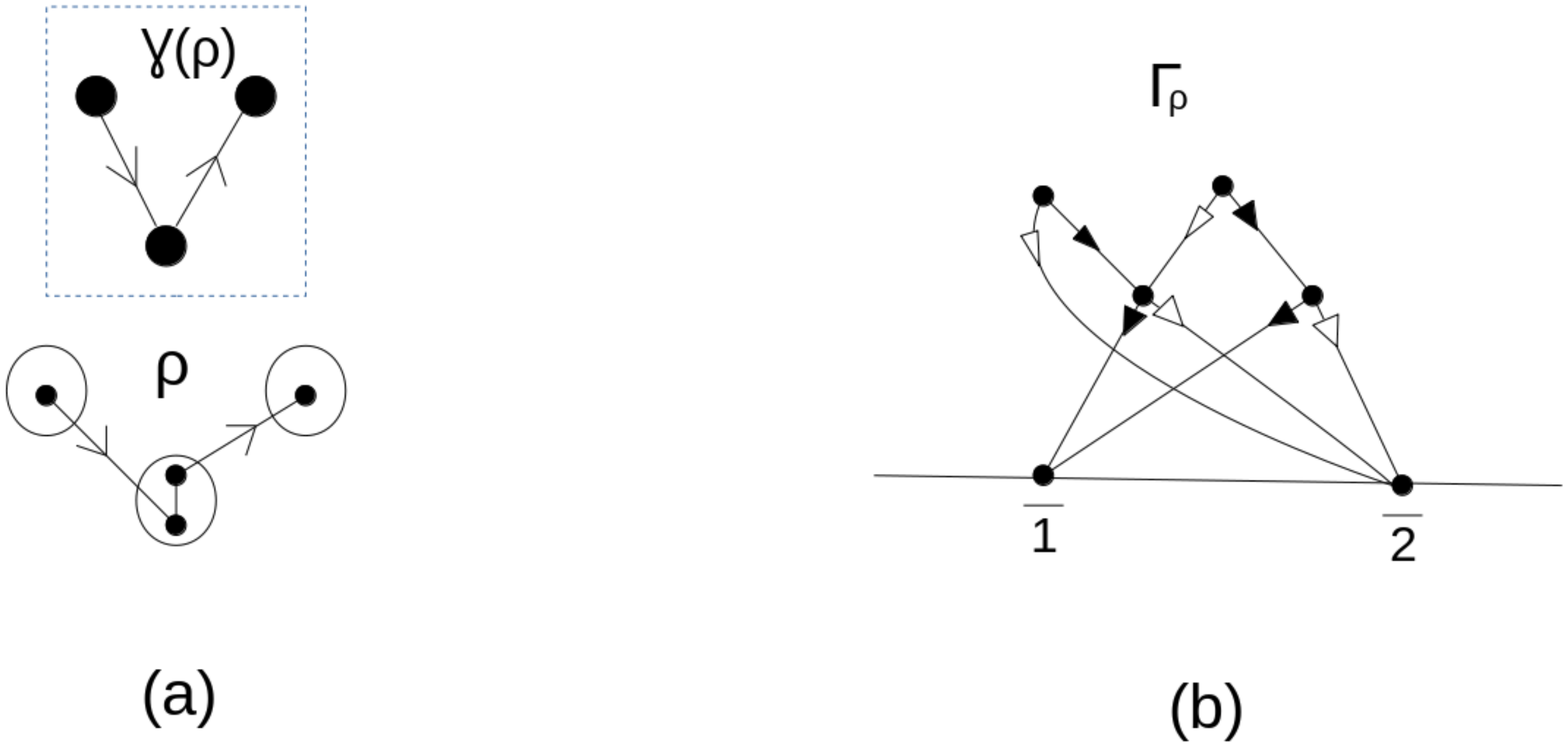} 
	\caption{\label{fig:network_Kgraph} In (a) we have a network of rooted trees $\rho \in NRT_{4}$ and we depicted its skeleton $\gamma(\rho)\in [RT]_3$ together with the additional orientation on its edges. In (b) we have the associated Kontsevich graph $\Gamma_\rho$ of type $(n,2)$, following the assignment $\rho\mapsto \Gamma_\rho$ of Section \ref{subsec:graphexp}. The terrestrial vertices are on a horizontal line and the aerial vertices are placed above them. The first ("left") arrow stemming out of an aerial vertex has a solid black head, while the second ("right") one has a hollow white one.  }
\end{figure}

Consider a vector field $V=V^j(x,p)\partial_{x^j}$ as in eq. \eqref{eq:aVgen} and $p_2 \in M^*$. We now associate a {\bf symbol map}
$$ \rho \mapsto \sym_\rho \equiv \sym_\rho^{p_2,V} \in C^\infty(T^*M),$$
to each network, by giving symbolic rules. For each internal vertex in some $\rho(v), \ v\in \g$, we write $p_{2k}V^k$ if it is the root or $V^k$ otherwise. For each internal edge in each tree $\rho(v)$, we write $\partial_{x^k}$ on the left if the edge is arriving or on the right if it is departing towards the root. For each skeleton-oriented edge $e \in \g$ connecting $\rho(e)_1$ to $\rho(e)_2$, we write $\pm \partial_{p_j}$ on the left at $\rho(e)_1$ and $\partial_{x^j}$ also on the left at $\rho(e)_2$, where we take $-1$ if the orientation of $e$ is towards the root of $\g$ or $+1$ otherwise. When $\rho \in \NRT^*$ has an extra marked vertex $\rho^*$ as above, the associated symbol requires an additional choice of 'decoration' by $x^j$ or $p_j$ and is denoted
$$ \sym_\rho^{x^j} \ or \ \sym_\rho^{p_j}.$$
This symbol is computed with the same rules as before and by adding at $\rho^*$ the extra symbol $-\partial_{p_j}$ or $\partial_{x^j}$ on the left, respectively for decorations $x^j$ or $p_j$.

\begin{example}\label{ex:netsymb}
	Consider the network $\rho \in NRT_4$ of Figure \ref{fig:network_Kgraph} (a). Following the above rules we obtain that the corresponding symbol is given by
	$$ \sym_\rho^{p_2,V}(x,p) = - p_{2k_1} \partial_{p_{j_1}}V^{k_1}|_{(x,p)} \ \partial_{p_{j_2}}V^{l_2}|_{(x,p)} \ \partial_{x^{j_1}}\partial_{x^{l_2}}V^{k_2}|_{(x,p)} p_{2k_2} \ \partial_{x^{j_2}}V^{k_3}|_{(x,p)} p_{2k_3}.$$
	When $V=V_\pi= \pi^{ij} p_j \partial_{x^i}$ as in eq. \eqref{eq:tayapi}, then
	$$ \sym_\rho^{p_2,V}(x,p_1) = - \hat{B}_{\Gamma_\rho}(\pi)(p_1,p_2),$$
	where $\Gamma_\rho$ is the associated Kontsevich graph, computed following Section \ref{subsec:graphexp} and depicted in Figure \ref{fig:network_Kgraph} (b), and its Kontsevich symbol was computed in Example \ref{ex:kontsymb}.
\end{example}

\end{appendix}

\end{document}